\documentclass[11pt]{amsart}
\usepackage[numbers, square]{natbib}
\usepackage{amssymb}
\usepackage{amsmath}
\usepackage{amsfonts}
\usepackage{geometry}
\usepackage{amsthm}
\usepackage{hyperref}
\usepackage{wrapfig}

\usepackage[dvipsnames]{xcolor}

\providecommand{\U}[1]{\protect\rule{.1in}{.1in}}
\theoremstyle{plain}

\newtheorem{corollary}{Corollary}

\newtheorem{lemma}{Lemma}

\newtheorem{proposition}{Proposition}
\newtheorem{remark}{Remark}

\newtheorem{theorem}{Theorem}
\numberwithin{equation}{section}

\newcommand{\Ve}{V_{\varepsilon}}
\newcommand{\fe}{f_{\varepsilon}}

\begin{document}

	\title[Quantitative uniqueness for bi-Laplace equations]{Quantitative uniqueness for bi-Laplace equations with  potentials}

	\author{Long Teng}
	\address{Department of Mathematics, Louisiana State University, Baton Rouge, LA, USA}
	\email{lteng2@lsu.edu}

	\author{Zhiwei Wang}
    \address{ Department of Mathematics, Louisiana State University, Baton Rouge, LA, USA}
	\email{zwang30@lsu.edu}
	
    \author{Jiuyi Zhu}
	\address{Department of Mathematics, Louisiana State University, Baton Rouge, LA, USA}
	\email{zhu@math.lsu.edu}
	
	
	
	\begin{abstract}
		We study quantitative unique continuation for bi-Laplace equations
		\[
		\Delta^{2}u+V(x)u=0
		\]
		by introducing some novel weighted frequency functions.  We establish quantitative
	three-ball inequalities and vanishing-order
		bounds for bounded and H\"older continuous potentials.  Three-ball inequalities are built on rescaling invariant  weighted frequency functions. The vanishing order results are shown by related, but different weighted frequency functions.
        \end{abstract}


	\subjclass[2020]{35B60, 35J30, 35J40, 35J48}
	
	\keywords{Quantitative unique continuation, bi-Laplace equations, frequency functions, three-ball inequality.}
	
	\maketitle

	\section{Introduction}
	
	In this paper, we study the quantitative unique continuation  (or quantitative uniqueness) for the bi-Laplace
	equation for various potentials. We consider the model problem 
	\begin{equation}\label{eq:intro-bilaplace-model}
		\Delta^{2}u+V(x)u=0
		\qquad\text{in }B_{R}(0)\subset\mathbb R^{n},
	\end{equation}
    where $B_R(0)$ is a ball centered at the origin with radius $R\geq 20$, $n\geq 2$ is the dimension, and $V(x)$ is a possibly complex-valued function.
	The strong unique continuation property asserts that, if a solution of a
	differential equation vanishes of infinite order at one point, then it must be
	identically zero in the connected component under consideration.  
    If the strong unique continuation property  
holds,
    quantitative
	unique continuation asks for a quantitative form of this assertion. That is, if a
	solution is not identically zero, how large can its order of vanishing be, and
	how does this upper bound depend on the size and regularity of the coefficients?
We say the order of vanishing for a smooth solution at $x_0$ is $l$, if $l$ is the largest integer
such that $D^\alpha u(x_0) = 0$ for all $|\alpha|<l$, where $\alpha$ is a multi-index. We can also define the order of vanishing by \begin{align*}
    \operatorname{ord}_{x_0}u
=\limsup_{r\to 0} 
  \frac{\log\|u\|_{L^\infty(B_r(x_0))}}{\log r}
\end{align*}for locally bounded solutions.
 After a normalization, one wants an explicit lower bound for
	\(\|u\|_{L^{\infty}(B_{r})}\)  as \(r\to 0\).
Specifically, we aim to derive the form
	\begin{equation}\label{eq:intro-vanishing-form}
		\|u\|_{L^{\infty}(B_{r}(x_0))}
		\ge r^{C a},
	\end{equation}
	where \(a \) is an explicit power depending on  the norms of the potentials. Then the estimate (\ref{eq:intro-vanishing-form}) indicates that the order of vanishing for $u$ at $x_0$ is at most $Ca$.
    To obtain (\ref{eq:intro-vanishing-form}), one usually needs to show a quantitative three-ball inequality with possibly a  sharp dependence on the norm of potentials.

    The strong unique continuation property holds for a wide range of second order elliptic equations even with rough potentials, see e.g. \cite{JerisonKenig1985, S90, W92, KochTataru2001} and references therein. In recent years, much attention is devoted to study quantitative unique continuation for second order elliptic equations.
	Let us briefly recall some classical literature on this topic. The classical and celebrated example arises from the  Laplace
	eigenfunctions on a compact manifold. The Laplace eigenfunction $\varphi_{\lambda}$ satisfies 
	\begin{equation}\label{eq:intro-eigenfunction}
	-\Delta_{g}\varphi_{\lambda}=\lambda\varphi_{\lambda}
		\qquad\text{on} \ \mathcal{M},
	\end{equation}
    where $\mathcal{M}$ is a smooth and compact manifold.  The sharp vanishing-order was shown to be  \(C\lambda^{1/2}\) by Donnelly-Fefferman in \cite{DonnellyFefferman1988}. The sharpness of the vanishing order can be seen from spherical harmonics. The obtained three-ball inequality and doubling inequality in  \cite{DonnellyFefferman1988}, which are used to obtain the sharp vanishing order,  play an important role in the study of Hausdorff measure of nodal sets for Laplace eigenfunctions, see e.g. \cite{lo18}, \cite{LM18}.
    
    For the Schr\"odinger equation
	\begin{equation}\label{eq:intro-schrodinger-model}
		-\Delta u+V(x)u=0
	\end{equation}
with bounded potentials $\|V\|_{L^\infty}\leq M$, motivated by the study  of Anderson localization for the Bernoulli model, Bourgain and Kenig in \cite{BourgainKenig2005} obtained the bound  \(CM^{2/3}\) for the order of vanishing under some normalized conditions. For complex-valued potential $V(x)$, the exponent $2/3$ in the vanishing order $CM^{2/3}$  is sharp  based on the nontrivial counter-example by Meshkov in \cite{Meshkov1992}.
  If one assumes additional regularity on
	\(V(x)\) in (\ref{eq:intro-schrodinger-model}), 
    the order of vanishing can be improved based on the  study of Laplace eigenfunctions.
    For $V\in C^{1}$ in  (\ref{eq:intro-schrodinger-model}),
    Bakri and Zhu independently  obtained the sharp upper bound $C(\|V\|_{C^1}+1)^{1/2}$ for the order of vanishing,
	which matches the $C\lambda^{1/2}$ order of the eigenfunctions in \cite{DonnellyFefferman1988}. If $V(x)$ is H\"older continuous, i.e. $V(x)\in {C^{\beta}}$ , the order of vanishing  was recently shown to be $C(\|V\|_{C^\beta}+1)^{\frac{2}{\beta+3}}$  in \cite{TWZ26}.  See e.g. \cite{DaveyZhu2018}, \cite{DaveyZhu2019}, \cite{Davey2020} for the study of quantitative unique continuation property for singular potential $V(x)$,  and see e.g. \cite{BanerjeeGarofalo2016}, \cite{LW23} for  the study of this property near boundary. See also \cite{K07} for the application of quantitative unique continuations to various partial differential equations and research areas.
	
	A useful way of passing from qualitative uniqueness to quantitative estimates is
	to prove a three-ball inequality.  In its simplest form, for
	\(0<r_{1}<r_{2}<r_{3}\), such an estimate reads
	\begin{equation}\label{eq:intro-three-ball-model}
		\|u\|_{X(B_{r_{2}})}
		\le
		\exp(\mathcal E)
		\|u\|_{X(B_{r_{1}})}^{\theta}
		\|u\|_{X(B_{r_{3}})}^{1-\theta},
		\qquad 0<\theta<1,
	\end{equation}
	where \(X\) is usually \(L^{2}\) or \(L^{\infty}\).  Once the exponent
	\(\mathcal E\) is explicit, an iteration or propagation-of-smallness argument
	turns \eqref{eq:intro-three-ball-model} into a vanishing-order estimate such as
	\eqref{eq:intro-vanishing-form}.  There are two standard approaches to obtain
	\eqref{eq:intro-three-ball-model}, i.e. Carleman estimates and frequency functions.
	Carleman estimates are robust and work well for rough potential, see e.g. \cite{H83}, \cite{JerisonKenig1985}.
    While
	frequency functions have the advantage that, when a suitable monotonicity formula
	is available, they reveal the exact growth scale of the solutions.  The
	frequency function method goes back to Almgren and was developed for elliptic
	operators by Garofalo-Lin \cite{GarofaloLin1986, GL87c}.  See \cite{Kukavica1998, Zhu2016, BG16,TWZ26} for recent
	frequency function treatments of quantitative unique continuation for second
	order equations.

	For higher order elliptic equations, the picture is less complete. The strong unique continuation property holds for higher order elliptic equations, see e.g. \cite{CG99},  \cite{CK10}, \cite{L07}, \cite{GR19}. Hence it is interesting to characterize the vanishing
order or its related properties by the potential functions
 for  higher order elliptic equations.  We study the simple bi-Laplace model  in \eqref{eq:intro-bilaplace-model}. 
    The bi-Laplace model  is primarily used in physics and engineering to model mechanical systems like plate bending, elasticity, and fluid dynamics.
	The third author \cite{Zhu2016} studied the quantitative uniqueness for \eqref{eq:intro-bilaplace-model}
	by a variant of the frequency function and obtained explicit vanishing-order
	bounds  \((1+\|V\|_{L^{\infty}})\). Later on, using the Carleman estimates,  the desirable upper bound  \((1+\|V\|_{L^{\infty}}^{1/3})\) was shown in \cite{Z18}. Using the frequency function approach, an upper bound \((1+\|V\|_{L^{\infty}}^{1/4}+\|\nabla V\|)\) of vanishing order for the bi-Laplace model in (\ref{eq:intro-bilaplace-model}) was derived in \cite{LiuTianYang2025} for $V\in C^{1}$. For various bi-Laplace models,  some qualitative or quantitative three-ball inequality related to inverse problems, free boundary problem, or zero-level sets of solutions  were studied in e.g. \cite{Zhu2019}, \cite{DG26}, \cite{ARV19}. 
    
    
	
	We now state the main results on three-ball inequalities and vanishing order estimates for bi-Laplace. Using a rescaling invariant frequency function, we can show the following quantitative rescaling invariant three-ball inequality 
	
	\begin{theorem}\label{thm:Linfty-three-ball}
		Let \(u\) solve (\ref{eq:intro-bilaplace-model}) with
	$\|V\|_{L^{\infty}(B_{R})}\le M$ for some large constant $M$. There exists $0<\theta_0<1$ such that
		 \begin{align} \label{kaokao-1}
	\|u\|_{L^\infty(B_{3r})}
			\le
	\exp\!\left(C(MR^{4})^{1/3}\right)
	\|u\|_{L^\infty(B_r)}^{\theta_0}
	\|u\|_{L^\infty(B_{30r})}^{1-\theta_0}
		\end{align}
        for any $0<r\leq \frac{R}{40}$, where $C$ depends on $n$.
		\end{theorem}

To study the vanishing order of solutions in (\ref{eq:intro-bilaplace-model}), we assume some normalization conditions. Assume
\begin{equation}\label{eq:intro-second-order-normalization}
		\|u\|_{L^{\infty}(B_{20})}\le \mathbb{M},
		\qquad
		\|u\|_{L^{\infty}(B_{1})}\ge1
	\end{equation}
for some large constant $\mathbb{M}$.

Using a different frequency function rather than the one in the proof of Theorem \ref{thm:Linfty-three-ball}, we can show the following theorem.
\begin{theorem}\label{thm-2}
Let \(u\) solve (\ref{eq:intro-bilaplace-model}) under the assumption (\ref{eq:intro-second-order-normalization})
with
	$\|V\|_{L^{\infty}(B_{R})}\le M$. Then the maximum vanishing order for $u$ in $B_{1/2}$ is at most $C(M^{1/3}+\log \mathbb{M})$.
	\end{theorem}
    
	Inspired by \cite{DonnellyFefferman1988} and \cite{Zhu2016}, we expect the additional regularity of potential functions will improve the order of the vanishing. See \cite{LZ22} for the sharp three-ball inequalities and doubling inequalities for  bi-Laplace eigenvalue problems in real analytic domains. However, it has been challenging due to the structure of bi-Laplace operators. 
    For the H\"older continuous potential $V(x)$, using a rescaling invariant frequency function, we can show a refined bound for the three-ball inequality.  
	
	\begin{theorem}\label{thm:holder-linfty-three-ball}
Let \(u\) solve (\ref{eq:intro-bilaplace-model}) with $\|V\|_{C^{0,\beta}(B_{R})}\le G$
        for any constant $0<\beta<1$ and large constant $G$.
	 There exists $0<\theta_0<1$ such that
     \begin{align}\label{three-hh}
	\|u\|_{L^\infty(B_{3r})}
			\le
			\exp\!\left(CG^{1/4}\right)
	\|u\|_{L^\infty(B_r)}^{\theta_0}
	\|u\|_{L^\infty(B_{30r})}^{1-\theta_0}
		\end{align} 	
 for any $0<r\leq \frac{R}{40}$, where $C$ depends on $n$ and $\beta$.
	\end{theorem}

Under a similar normalization condition as follows,
\begin{equation}\label{eq:intro-second-order-normalization-h}
		\|u\|_{L^{\infty}(B_{10})}\le \mathbb{M},
		\qquad
		\|u\|_{L^{\infty}(B_{r_\ast})}\ge 1 
	\end{equation}
where $r_\ast=\frac{G^{-1/4}}{100}$
, we can show the following vanishing order results by a slightly different frequency function.
	\begin{theorem} \label{thm-4}
Let \(u\) solve (\ref{eq:intro-bilaplace-model}) under the assumption (\ref{eq:intro-second-order-normalization-h}) with $\|V\|_{C^{0,\beta}(B_{R})}\le G$.
Then the maximum vanishing order for $u$ in $B_{r_\ast/2}$ is at most $C(G^{1/4}+\log\mathbb{M})$.
	\end{theorem}

    Let us discuss the frequency function approach in the study of quantitative unique continuation. We employ a variant of frequency function other than  Almgren frequency function introduced for unique continuation property in \cite{GarofaloLin1986}. For the Schr\"odinger equation (\ref{eq:intro-schrodinger-model}), this variant of frequency function is given as
    \begin{align}
        N(r)=\frac{D(r)}{H(r)}=\frac{\int_{B_{r}}|\nabla u|^{2}\mu_{r}^{\alpha}\,dx}{\int_{B_{r}} |u|^2\mu_{r}^{\alpha-1}\,dx}
        \label{weight-fre}
    \end{align}
where $\mu_r=r^2-|x|^2$ and $\alpha>0$ is  some parameter. By choosing some appropriate $\alpha$, one is able to minimize the exponent appeared in constant of the three-ball inequality, i.e. minimize $\mathcal{E}$ in (\ref {eq:intro-three-ball-model}).
This free parameter $\alpha$ is very useful in the quantitative study of solutions. To construct the frequency function (\ref{weight-fre}), one first establishes $H(r)=\int_{B_{r}} |u|^2\mu_{r}^{\alpha-1}\,dx$ in the denominator of this quotient. The numerator $D(r)$ of $N(r)$ usually follows from the derivative of $H(r)$.
This weighted frequency  function was introduced in \cite{K20}, then applied in \cite{Zhu2016}  for $ V\in C^1$ and in \cite{Davey2025} for $V\in L^\infty$ for the sharp vanishing order of solutions in (\ref{eq:intro-schrodinger-model}).
To apply the frequency function approach for bi-Laplace model (\ref{eq:intro-bilaplace-model}), a natural way is to 
rewrite \eqref{eq:intro-bilaplace-model}
	as a second order elliptic system
	\begin{equation}\label{eq:intro-system}
		w=\Delta u,
		\qquad
		\Delta w=-V(x)u.
	\end{equation}
Then one is led to construct 
\begin{equation}\label{eq:intro-bounded-frequency-mass-n}
	H(r)=\int_{B_{r}}\bigl(|u|^2+|w|^2\bigr)
		\mu_{r}^{\alpha-1}\,dx.
	\end{equation}
However, the investigation of $H(r)$ in (\ref{eq:intro-bounded-frequency-mass-n}) does not give the desirable three-ball inequality and vanishing order estimates. See the applications of (\ref{eq:intro-bounded-frequency-mass-n}) in \cite{Zhu2016} and \cite{LiuTianYang2025}. The first novelty of the paper is to introduce a new variant of frequency function by choosing 
\begin{equation}\label{eq:intro-bounded-frequency-mass-v}
	H(r)=\int_{B_{r}}\bigl(|u|^2+\lambda r^4|w|^2\bigr)
		\mu_{r}^{\alpha-1}\,dx,
	\end{equation}
where $\lambda>0$ is a small free parameter. By choosing the appropriate two free parameters $\alpha$ and $\lambda$, one is able to improve the previous literature on quantitative unique continuation. Note that the $r^4$ term for $r^4|w|^2$ in $H(r)$ allows the later constructed frequency function to be rescaling invariant since $w=\triangle u$.

The second novelty of the paper lies in the study of the  vanishing order. The constructed frequency function from (\ref{eq:intro-bounded-frequency-mass-v}) does not imply the sharp vanishing order, see Remark \ref{rem-1} in Section 3.  To  overcome the difficulty, we instead introduce  $H(r)$ in the form
	\begin{equation}\label{eq:intro-bounded-frequency-mass}
H(r)=\int_{B_{r}}\bigl(|u|^2+\lambda r^{3}|w|^2\bigr)
		\mu_{r}^{\alpha-1}\,dx
	\end{equation}
    for Theorem \ref{thm-2}. While for H\"older continuous case,
    we adopt
	\begin{equation}\label{eq:intro-smooth-frequency-mass}
		H(r)=\int_{B_{r}}\bigl(|u|^2+\lambda r^{2}|w|^2\bigr)
		\mu_{r}^{\alpha-1}\,dx
	\end{equation}
    for Theorem \ref{thm-4}. Although $H(r)$ does not lead to the rescaling invariant frequency function, they are used to get rid of ``bad terms" in the estimates of the order of vanishing.

    Let us discuss why the exponent $M^{1/3}$ in Theorem \ref{thm:Linfty-three-ball}, Theorem \ref{thm-2}, and $G^{1/4}$ in Theorem \ref{thm:holder-linfty-three-ball}, Theorem \ref{thm-4} are possibly sharp for complex-valued potentials  $V(x)$. Note that the frequency functions and Carleman estimates do not distinguish real-valued and complex-valued functions. As it is mentioned that $CM^{2/3}$ is the sharp exponent for second order elliptic equations in \cite{BourgainKenig2005} for the vanishing order based on the complex-valued counter-example constructed in \cite{Meshkov1992}. The exponent $CM^{2/3}$ does not seem to be intuitive from the frequency function approach. However, it appears from the Carleman-estimate perspective. The standard $L^2$ type Carleman estimates without diving deep in the role of weight functions  
    state roughly as
    \begin{align}\label{car-1}
    \tau^{3/2} \|e^{\tau \phi} f\|_{L^2}+  \tau^{1/2}\|e^{\tau \phi} \nabla f\|_{L^2}\leq C \|e^{\tau \phi} \triangle f\|_{L^2}
    \end{align}
	where $\phi(x)$ is some appropriate weight function, $\tau$ is some large constant, and $f$ is any smooth function with compact support. To incorporate the $L^\infty$ norm of $V(x)$ into the Carleman estimates, one usually assumes $\tau^{3/2}\geq CM\geq \|V\|_{L^\infty}$. Thus, the lower bound of $\tau> CM^{2/3}$ provides the vanishing order results. If we iterate the Carleman estimates (\ref{car-1}) twice, we arrive at the Carleman estimates 
    \begin{align}\label{car-1-bi}
    \tau^{3} \|e^{\tau \phi} f\|_{L^2} \leq C \|e^{\tau \phi} \triangle^2 f\|_{L^2}.
    \end{align}
    Then $\tau^3\geq CM$ is needed to apply the Carleman estimates (\ref{car-1-bi}) to the model \eqref{eq:intro-bilaplace-model}. Therefore, the lower bound of $\tau > CM^{1/3}$ gives the desirable bounds in Theorem \ref{thm-2}, see e.g. \cite{Z18}.

The fact that the H\"older continuous $V(x)\in C^{\beta}$ provides the sharp vanishing order for the bi-Laplace (\ref{eq:intro-bilaplace-model}) seems to be out of our expectation. Let us explain it by an example intuitively. Consider the second order elliptic equations
    \begin{align}
        \triangle u+\lambda^2 V_1(x)u=0,
        \label{lapalce-2}
    \end{align}
    where we assume $\lambda>0$ to be any large constant and $\|V_1\|_{C^{2, \beta}}\leq \tilde{C}$ for some fixed $\tilde{C}$. It is known from \cite{Zhu2016} that the sharp vanishing order of $u$ is at most $\|\lambda^2 V_1\|^{1/2}_{C^{2,\beta}}\leq \tilde{C}\lambda$. Taking the Laplace operator to both sides of (\ref{lapalce-2}) leads to 
    \begin{align}\label{example-1}
        \triangle^2 u+2\lambda^2\nabla V_1\cdot\nabla u+(\lambda^2\triangle V_1-\lambda^4V^2_1) u=0,
    \end{align}
    where at least $C^2$ regularity of $V_1$ is required.
   Let us ignore the first derivative term $2\lambda^2\nabla V_1\cdot\nabla u$, which will be studied in a forthcoming paper,  since we do not consider it in our model. Denote $V(x)=(\lambda^2\triangle V_1-\lambda^4V^2_1)$. Hence $\|V\|_{C^{\beta}}\leq C\lambda^4$.  By the conclusions in Theorem \ref{thm:holder-linfty-three-ball} and Theorem \ref{thm-4}, we learn that the vanishing order of $u$ is at most $C\|V\|^{1/4}_{C^{\beta}}\leq C\lambda$, which  matches the sharp vanishing order for  the second order elliptic equations  (\ref{lapalce-2}) since the solution $u$ in (\ref{example-1}) arises from the solution in  (\ref{lapalce-2}).


The paper is organized as follows.  Section 2 is devoted to the proof of 
	Theorem~\ref{thm:Linfty-three-ball}.  We establish the 
	almost-monotonicity formula for the weighted frequency function and then convert it into an \(L^{2}\)-type, and $L^\infty$ type 
	three-ball inequality.   In Section 3, we prove 
	 Theorem~\ref{thm-2} by exploiting a slightly different frequency function. Section 4 proves the quantitative three-ball inequality for H\"older continuous potential in Theorem \ref{thm:holder-linfty-three-ball}.
    Section 5 is to obtain the proof of Theorem \ref{thm-4} using a variant frequency function. 
	The letters \(C\), \(C_{i}\), $\tilde{C}$, $c$ denote positive constants
	which may vary from line to line, but they do not depend on $u, M, \mathbb{M}$ and $G$. Throughout the paper, \(B_{r}(x)\) denotes the ball
	centered at $x$ with radius $r$ and \(B_{r}\) denotes the ball
	centered at the origin with radius $r$.

   {\bf Acknowledgment:} All underlying ideas and insights are from the authors.  We did not use the Generative AI tool.

	\section{Three-ball inequality of bi-Laplace equations with bounded potential}
	
In this section, we study quantitative unique continuation for bounded potentials. Let \begin{equation}\label{eq:bilaplace-equation}
		\Delta^{2}u+V(x)u=0
		\qquad\text{in }B_{R},
	\end{equation}
with $\|V\|_{L^{\infty}(B_{R})}\leq M$. Set
	\begin{equation}\label{eq:M-definition}
		w:=\Delta u,
		\qquad
		\mu_{r}(x):=r^{2}-|x|^{2}.
	\end{equation}
	We introduce two constant parameters
		$\alpha\ge 2$ and 
$\lambda>0,$ which are to be determined.
	To introduce the frequency functions, for \(0<r<R\),  we define
	\begin{equation}\label{eq:H-definition}
		H(r):=\int_{B_{r}}\bigl(|u|^2+\lambda r^{4}|w|^2\bigr)\mu_{r}^{\alpha-1}\,dx,
	\end{equation}
	\begin{equation}\label{eq:D-definition}
		D(r):=\int_{B_{r}}\bigl(|\nabla u|^{2}+\lambda r^{4}|\nabla w|^{2}\bigr)\mu_{r}^{\alpha}\,dx
		+
		4\lambda\alpha r^{4}\int_{B_{r}}|w|^2\mu_{r}^{\alpha-1}\,dx,
	\end{equation}
	\begin{equation}\label{eq:L-definition}
		L(r):=\int_{B_{r}}\bigl(u\Delta u+\lambda r^{4}w\Delta w\bigr)\mu_{r}^{\alpha}\,dx,
	\end{equation}
	and
	\begin{equation}\label{eq:N-definition}
		N(r):=\frac{D(r)}{H(r)}.
	\end{equation}
	
We will prove the monotone property for $N(r)$, then use it to show the three-ball inequality.	We break down the proof   into a sequence of lemmas.
	We start it with a weighted differentiation identity, which will be used often in the whole paper.
	\begin{lemma}\label{lem:weighted-differentiation}
		Let \(\gamma\ge 1\), and define
		\begin{equation}\label{eq:F-definition}
			F(r):=\int_{B_{r}}f(x)\mu_{r}(x)^{\gamma}\,dx.
		\end{equation}
		Then
		\begin{equation}\label{eq:F-prime-first}
			F'(r)
			=
			\frac{2\gamma+n}{r}F(r)
			+
			\frac{1}{r}\int_{B_{r}}\nabla f(x)\cdot x\,\mu_{r}(x)^{\gamma}\,dx,
		\end{equation}
		and also
		\begin{equation}\label{eq:F-prime-second}
			F'(r)
			=
			\frac{2\gamma+n}{r}F(r)
			+
			\frac{1}{2(\gamma+1)r}\int_{B_{r}}\Delta f(x)\,\mu_{r}(x)^{\gamma+1}\,dx.
		\end{equation}
	\end{lemma}
	The proof of Lemma \ref{lem:weighted-differentiation} just involves integration by parts arguments and has been shown in e.g. \cite{Davey2025}. For readers' convenience, we include its proof in the Appendix.

To show the monotone property for $N(r)$, we need to consider the derivatives of $H(r)$ and $D(r).$ We show the derivative of \(H(r)\) in the following.
	\begin{lemma}\label{lem:H-prime}
		The function \(H(r)\) in \eqref{eq:H-definition} satisfies
		\begin{equation}\label{eq:H-prime}
			H'(r)
			=
			\frac{2\alpha+n-2}{r}H(r)+\frac{D(r)+L(r)}{\alpha r}.
		\end{equation}
	\end{lemma}
	
	\begin{proof}
		Write
		\begin{equation*}\label{eq:H-split}
			H(r)=H_{u}(r)+\lambda r^{4}H_{w}(r),
		\end{equation*}
		where
		\begin{equation*}\label{eq:Hu-Hw}
			H_{u}(r):=\int_{B_{r}}|u|^2\mu_{r}^{\alpha-1}\,dx,
			\qquad
			H_{w}(r):=\int_{B_{r}}|w|^2\mu_{r}^{\alpha-1}\,dx.
		\end{equation*}
		Applying Lemma \ref{lem:weighted-differentiation} with \(\gamma=\alpha-1\) and \(f=|u|^2\), we derive that
		\begin{align}
			H_{u}'(r)
			&=
			\frac{2\alpha+n-2}{r}H_{u}(r)
			+
			\frac{1}{2\alpha r}\int_{B_{r}}\Delta(|u|^2)\mu_{r}^{\alpha}\,dx \notag\\
			&=
			\frac{2\alpha+n-2}{r}H_{u}(r)
			+
			\frac{1}{\alpha r}
			\left(
			\int_{B_{r}}|\nabla u|^{2}\mu_{r}^{\alpha}\,dx
			+
			\int_{B_{r}}u\Delta u\,\mu_{r}^{\alpha}\,dx
			\right).
			\label{eq:Hu-prime}
		\end{align}
		Similarly, applying Lemma \ref{lem:weighted-differentiation} with \(\gamma=\alpha-1\) and \(f=|w|^2\), we have
		\begin{align}
			H_{w}'(r)
			&=
			\frac{2\alpha+n-2}{r}H_{w}(r)
			+
			\frac{1}{2\alpha r}\int_{B_{r}}\Delta(|w|^2)\mu_{r}^{\alpha}\,dx \notag\\
			&=
			\frac{2\alpha+n-2}{r}H_{w}(r)
			+
			\frac{1}{\alpha r}
			\left(
			\int_{B_{r}}|\nabla w|^{2}\mu_{r}^{\alpha}\,dx
			+
			\int_{B_{r}}w\Delta w\,\mu_{r}^{\alpha}\,dx
			\right).
		\label{eq:Hw-prime}
		\end{align}
		Hence
		\begin{align}
			\frac{d}{dr}\bigl(\lambda r^{4}H_{w}(r)\bigr)
			&=
			4\lambda r^{3}H_{w}(r)+\lambda r^{4}H_{w}'(r) \notag\\
			&=
			\frac{2\alpha+n-2}{r}\lambda r^{4}H_{w}(r)
			+
			\frac{1}{\alpha r}
			\left(
			\lambda r^{4}\int_{B_{r}}|\nabla w|^{2}\mu_{r}^{\alpha}\,dx
			\right. \notag\\
			&\quad\left.
			+
			4\lambda\alpha r^{4}\int_{B_{r}}|w|^2\mu_{r}^{\alpha-1}\,dx
			+
			\lambda r^{4}\int_{B_{r}}w\Delta w\,\mu_{r}^{\alpha}\,dx
			\right).
			\label{eq:lambda-r4-Hw-prime}
		\end{align}
		Adding \eqref{eq:Hu-prime} and \eqref{eq:lambda-r4-Hw-prime}, and recalling \eqref{eq:D-definition} and \eqref{eq:L-definition}, we obtain \eqref{eq:H-prime}.
	\end{proof}

	We study the derivative of \(D(r)\) as follows.
	\begin{lemma}\label{lem:D-prime}
		The function \(D\) defined in \eqref{eq:D-definition} satisfies
		\begin{align}
			D'(r)
			&=
			\frac{2\alpha+n-2}{r}D(r)
			-
			\frac{1}{4\alpha r}
			\int_{B_{r}}
			\bigl(|w|^2+\lambda r^{4}(\Delta w)^{2}\bigr)\mu_{r}^{\alpha+1}\,dx \notag\\
			&\quad
            +
			\frac{4\alpha}{r}
			\int_{B_{r}}
			\left[
			\left(x\cdot \nabla u-\frac{w\mu_{r}}{4\alpha}\right)^{2}
			+
			\lambda r^{4}
			\left(x\cdot \nabla w+2w-\frac{\Delta w\,\mu_{r}}{4\alpha}\right)^{2}
			\right]
			\mu_{r}^{\alpha-1}\,dx.
			\label{eq:D-prime}
		\end{align}
	\end{lemma}
	
	\begin{proof}
		Define
		\begin{equation*}\label{eq:Du-Dw-Hw}
			D_{u}(r):=\int_{B_{r}}|\nabla u|^{2}\mu_{r}^{\alpha}\,dx,
			\qquad
			\widetilde D_{w}(r):=\int_{B_{r}}|\nabla w|^{2}\mu_{r}^{\alpha}\,dx,
			\qquad
			H_{w}(r):=\int_{B_{r}}|w|^2\mu_{r}^{\alpha-1}\,dx.
		\end{equation*}
		Then
		\begin{equation}\label{eq:D-split}
			D(r)=D_{u}(r)+\lambda r^{4}\widetilde D_{w}(r)+4\lambda\alpha r^{4}H_{w}(r).
		\end{equation}
		
		We first compute \(D_{u}'(r)\). By Lemma \ref{lem:weighted-differentiation} with \(\gamma=\alpha\) and \(f=|\nabla u|^{2}\), we have
		\begin{equation*}\label{eq:Du-prime-start}
			D_{u}'(r)
			=
			\frac{2\alpha+n}{r}D_{u}(r)
			+
			\frac{1}{r}\int_{B_{r}}\nabla(|\nabla u|^{2})\cdot x\,\mu_{r}^{\alpha}\,dx.
		\end{equation*}
		Since
		$\nabla(|\nabla u|^{2})\cdot x
		=
		2\sum_{i,j=1}^{n}\partial_{ij}u\,\partial_{i}u\,x_{j},$
		we obtain
		\begin{equation}\label{eq:Du-prime-middle}
			D_{u}'(r)
			=
			\frac{2\alpha+n}{r}D_{u}(r)
			+
			\frac{2}{r}
			\int_{B_{r}}
			\sum_{i,j=1}^{n}\partial_{ij}u\,\partial_{i}u\,x_{j}\,\mu_{r}^{\alpha}\,dx.
		\end{equation}
		Integrating by parts in the \(x_{i}\)-variable and using the fact \(\mu_{r}=0\) on \(\partial B_{r}\), we can show that
		\begin{align}
			\int_{B_{r}}
			\sum_{i,j=1}^{n}\partial_{ij}u\,\partial_{i}u\,x_{j}\,\mu_{r}^{\alpha}\,dx
			&=
			-\int_{B_{r}}(x\cdot \nabla u)\Delta u\,\mu_{r}^{\alpha}\,dx
			-\int_{B_{r}}|\nabla u|^{2}\mu_{r}^{\alpha}\,dx
			\notag\\
			&\quad
			+
			2\alpha\int_{B_{r}}(x\cdot \nabla u)^{2}\mu_{r}^{\alpha-1}\,dx.
			\label{eq:Du-ibp}
		\end{align}
		Since \(\Delta u=w\), substituting \eqref{eq:Du-ibp} into \eqref{eq:Du-prime-middle} yields
		\begin{align}
			D_{u}'(r)
			&=
			\frac{2\alpha+n-2}{r}D_{u}(r)
			+
			\frac{4\alpha}{r}\int_{B_{r}}(x\cdot \nabla u)^{2}\mu_{r}^{\alpha-1}\,dx
			-
			\frac{2}{r}\int_{B_{r}}(x\cdot \nabla u)\,w\,\mu_{r}^{\alpha}\,dx
			\notag\\
			&=
			\frac{2\alpha+n-2}{r}D_{u}(r)
			+
			\frac{4\alpha}{r}\int_{B_{r}}
			\left(x\cdot \nabla u-\frac{w\mu_{r}}{4\alpha}\right)^{2}\mu_{r}^{\alpha-1}\,dx
			-
			\frac{1}{4\alpha r}\int_{B_{r}}|w|^2\mu_{r}^{\alpha+1}\,dx.
			\label{eq:Du-prime}
		\end{align}
		
		We now compute the derivative of the \(\widetilde D_{w}(r)\)-part in (\ref{eq:D-split}). Repeating the same arguments with \(w\) in place of \(u\), we obtain
		\begin{align}
			\widetilde D_{w}'(r)
			&=
			\frac{2\alpha+n-2}{r}\widetilde D_{w}(r)
			+
			\frac{4\alpha}{r}\int_{B_{r}}
			\left(x\cdot \nabla w-\frac{\Delta w\,\mu_{r}}{4\alpha}\right)^{2}\mu_{r}^{\alpha-1}\,dx
			\notag\\
			&\quad
			-
			\frac{1}{4\alpha r}\int_{B_{r}}(\Delta w)^{2}\mu_{r}^{\alpha+1}\,dx.
			\label{eq:Dw-tilde-prime}
		\end{align}
		Next we study $H'(r)$. By (\ref{eq:F-prime-second}) in Lemma \ref{lem:weighted-differentiation} with \(\gamma=\alpha-1\) and \(f=|w|^2\), it holds that
		\begin{equation}\label{eq:Hw-prime-again}
			H_{w}'(r)
			=
			\frac{2\alpha+n-2}{r}H_{w}(r)
			+
			\frac{\widetilde D_{w}(r)+L_{w}(r)}{\alpha r},
		\end{equation}
		where
		\begin{equation*}\label{eq:Lw-definition}
			L_{w}(r):=\int_{B_{r}}w\Delta w\,\mu_{r}^{\alpha}\,dx.
		\end{equation*}
		Therefore
		\begin{align}
			\frac{d}{dr}\bigl(\lambda r^{4}\widetilde D_{w}(r)+4\lambda\alpha r^{4}H_{w}(r)\bigr)
			&=
			4\lambda r^{3}\widetilde D_{w}(r)
			+\lambda r^{4}\widetilde D_{w}'(r)
			+
			16\lambda\alpha r^{3}H_{w}(r)
			+
			4\lambda\alpha r^{4}H_{w}'(r)
			\notag\\
			&=
			\frac{2\alpha+n-2}{r}\bigl(\lambda r^{4}\widetilde D_{w}(r)+4\lambda\alpha r^{4}H_{w}(r)\bigr)
			\notag\\
			&\quad
			+
			\frac{4\alpha\lambda r^{4}}{r}
			\int_{B_{r}}
			\left(x\cdot \nabla w-\frac{\Delta w\,\mu_{r}}{4\alpha}\right)^{2}\mu_{r}^{\alpha-1}\,dx
			\notag\\
			&\quad
			-
			\frac{\lambda r^{4}}{4\alpha r}
			\int_{B_{r}}(\Delta w)^{2}\mu_{r}^{\alpha+1}\,dx
			\notag\\
			&\quad
			+
			\frac{4\lambda r^{4}}{r}\Bigl(2\widetilde D_{w}(r)+L_{w}(r)+4\alpha H_{w}(r)\Bigr).
			\label{eq:w-part-before-square}
		\end{align}
		
		We now rewrite the last term on the right-hand side of \eqref{eq:w-part-before-square}. Integration by parts shows that
		\begin{equation*}\label{eq:w-ibp-identity}
			2\alpha\int_{B_{r}}w(x\cdot \nabla w)\mu_{r}^{\alpha-1}\,dx
			=
			\widetilde D_{w}(r)+L_{w}(r).
		\end{equation*}
		Hence
		\begin{align*}
			4\alpha\int_{B_{r}}
			w\left(x\cdot \nabla w-\frac{\Delta w\,\mu_{r}}{4\alpha}\right)\mu_{r}^{\alpha-1}\,dx
			&=
			2\widetilde D_{w}(r)+L_{w}(r).
		\end{align*}
		Therefore
		\begin{align}\label{eq:w-part-rewritten}
			\frac{4\lambda r^{4}}{r}\Bigl(2\widetilde D_{w}(r)+L_{w}(r)+4\alpha H_{w}(r)\Bigr)
			&=
			\frac{4\alpha\lambda r^{4}}{r}
			\int_{B_{r}}
			\left[
			4w\left(x\cdot \nabla w-\frac{\Delta w\,\mu_{r}}{4\alpha}\right)+4|w|^2
			\right]\mu_{r}^{\alpha-1}\,dx.
		\end{align}
		Substituting \eqref{eq:w-part-rewritten} into \eqref{eq:w-part-before-square}, and using 
		\[
		\left(x\cdot \nabla w+2w-\frac{\Delta w\,\mu_{r}}{4\alpha}\right)^{2}
		=
		\left(x\cdot \nabla w-\frac{\Delta w\,\mu_{r}}{4\alpha}\right)^{2}
		+
		4w\left(x\cdot \nabla w-\frac{\Delta w\,\mu_{r}}{4\alpha}\right)
		+
		4|w|^2,
		\]
		 we obtain
		\begin{align}
			\frac{d}{dr}\bigl(\lambda r^{4}\widetilde D_{w}(r)+4\lambda\alpha r^{4}H_{w}(r)\bigr)
			&=
			\frac{2\alpha+n-2}{r}\bigl(\lambda r^{4}\widetilde D_{w}(r)+4\lambda\alpha r^{4}H_{w}(r)\bigr)
			\notag\\
			&\quad
			+
			\frac{4\alpha\lambda r^{4}}{r}
			\int_{B_{r}}
			\left(x\cdot \nabla w+2w-\frac{\Delta w\,\mu_{r}}{4\alpha}\right)^{2}\mu_{r}^{\alpha-1}\,dx
			\notag\\
			&\quad
			-
			\frac{\lambda r^{4}}{4\alpha r}
			\int_{B_{r}}(\Delta w)^{2}\mu_{r}^{\alpha+1}\,dx.
			\label{eq:w-part-after-square}
		\end{align}
		Finally, adding \eqref{eq:Du-prime} and \eqref{eq:w-part-after-square}, and using \eqref{eq:D-split}, we obtain \eqref{eq:D-prime}.
	\end{proof}

    We are ready to show the monotonicity results by considering the derivative of $N(r)$.
	\begin{lemma}\label{lem:N-prime}
		Suppose that \(u\) solves \eqref{eq:bilaplace-equation}. Define
		\begin{equation}\label{eq:Ntilde-definition-v}
			\widetilde N(r)
			:=
			N(r)
			+
			\frac{1}{4\alpha\lambda}\log r
			+
			\frac{\lambda M^{2}}{32\alpha}r^{8}.
		\end{equation}
        Then $\widetilde N(r)$ is nondecreasing in $(0, R).$
	\end{lemma}
	
	\begin{proof}
		Define
		\begin{align*}
			J(r)
			&:=
			2\alpha
			\int_{B_{r}}
			\left[
			u\left(x\cdot \nabla u-\frac{w\mu_{r}}{4\alpha}\right)
			+
			\lambda r^{4}w\left(x\cdot \nabla w+2w-\frac{\Delta w\,\mu_{r}}{4\alpha}\right)
			\right]
			\mu_{r}^{\alpha-1}\,dx.
		\end{align*}
		We first compute \(J(r)\).  Integration by parts arguments show that
		\begin{equation}\label{eq:u-J-identity}
			2\alpha
			\int_{B_{r}}
			u\left(x\cdot \nabla u-\frac{w\mu_{r}}{4\alpha}\right)\mu_{r}^{\alpha-1}\,dx
			=
			\int_{B_{r}}|\nabla u|^{2}\mu_{r}^{\alpha}\,dx
			+
			\frac12\int_{B_{r}}u\Delta u\,\mu_{r}^{\alpha}\,dx.
		\end{equation}
		Similarly,
		\begin{align}
			2\alpha\lambda r^{4}
			\int_{B_{r}}
			w\left(x\cdot \nabla w+2w-\frac{\Delta w\,\mu_{r}}{4\alpha}\right)\mu_{r}^{\alpha-1}\,dx
			=&
			\lambda r^{4}\int_{B_{r}}|\nabla w|^{2}\mu_{r}^{\alpha}\,dx
			+
			4\lambda\alpha r^{4}\int_{B_{r}}|w|^2\mu_{r}^{\alpha-1}\,dx \notag\\
			&+
			\frac12\lambda r^{4}\int_{B_{r}}w\Delta w\,\mu_{r}^{\alpha}\,dx.
            \label{eq:w-J-identity}
		\end{align}
		Adding \eqref{eq:u-J-identity} and \eqref{eq:w-J-identity}, and recalling \eqref{eq:D-definition} and \eqref{eq:L-definition}, we can check that
		\begin{equation}\label{eq:J-equals-D-plus-half-L}
			J(r)=D(r)+\frac12L(r).
		\end{equation}
		Hence
		\begin{equation}\label{eq:D-and-D-plus-L}
			D(r)=J(r)-\frac12L(r),
			\qquad
			D(r)+L(r)=J(r)+\frac12L(r).
		\end{equation}
		
		Now, since \(N(r)=D(r)/H(r)\), then
		\begin{equation}\label{eq:H2Nprime}
			H^2(r)N'(r)=D'(r)H(r)-H'(r)D(r).
		\end{equation}
		Substituting \eqref{eq:H-prime}, \eqref{eq:D-prime}, and \eqref{eq:D-and-D-plus-L} into \eqref{eq:H2Nprime}, we obtain
		\begin{align}
			H^2(r)N'(r)
			&=
			\frac{4\alpha}{r}
			\Biggl[
			H(r)
			\int_{B_{r}}
			\left(
			\left(x\cdot \nabla u-\frac{w\mu_{r}}{4\alpha}\right)^{2}
			+
			\lambda r^{4}
			\left(x\cdot \nabla w+2w-\frac{\Delta w\,\mu_{r}}{4\alpha}\right)^{2}
			\right)
			\mu_{r}^{\alpha-1}\,dx
			\notag\\
			&\qquad\qquad
			-
			\left(
			\int_{B_{r}}
			\left[
			u\left(x\cdot \nabla u-\frac{w\mu_{r}}{4\alpha}\right)
			+
			\lambda r^{4}w\left(x\cdot \nabla w+2w-\frac{\Delta w\,\mu_{r}}{4\alpha}\right)
			\right]
			\mu_{r}^{\alpha-1}\,dx
			\right)^{2}
			\Biggr]
			\notag\\
			&\qquad
			+
			\frac{1}{4\alpha r}
			\left(
			L(r)^{2}
			-
			H(r)\int_{B_{r}}\bigl(|w|^2+\lambda r^{4}(\Delta w)^{2}\bigr)\mu_{r}^{\alpha+1}\,dx
			\right).
			\label{eq:H2Nprime-expanded}
		\end{align}
		By Cauchy--Schwarz inequality, the large bracket in \eqref{eq:H2Nprime-expanded} is nonnegative. Therefore
		\begin{equation}\label{eq:Nprime-preliminary}
			N'(r)
			\ge
			-\frac{1}{4\alpha r\,H(r)}
			\int_{B_{r}}\bigl(|w|^2+\lambda r^{4}(\Delta w)^{2}\bigr)\mu_{r}^{\alpha+1}\,dx.
		\end{equation}
        
		We now estimate the terms on the right-hand side of \eqref{eq:Nprime-preliminary}. Since \(\mu_{r}(x)\le r^{2}\), we have
		\begin{equation}\label{eq:w2-bound}
			\int_{B_{r}}|w|^2\mu_{r}^{\alpha+1}\,dx
			\le
			r^{4}\int_{B_{r}}|w|^2\mu_{r}^{\alpha-1}\,dx.
		\end{equation}
		Recalling the definition \eqref{eq:H-definition} of \(H(r)\), we deduce from \eqref{eq:w2-bound} that
		\begin{equation}\label{eq:w2-bound-by-H}
			\int_{B_{r}}|w|^2\mu_{r}^{\alpha+1}\,dx
			\le
			\frac{1}{\lambda}H(r).
		\end{equation}
		
		Next, since \(w=\Delta u\) and \(u\) solves \eqref{eq:bilaplace-equation}, we have
	$\Delta w=-Vu.$
		Therefore
		\begin{align}
			\lambda r^{4}\int_{B_{r}}(\Delta w)^{2}\mu_{r}^{\alpha+1}\,dx
			&=
			\lambda r^{4}\int_{B_{r}}V^{2}|u|^2\mu_{r}^{\alpha+1}\,dx \notag\\
			&\le
			\lambda M^{2}r^{8}\int_{B_{r}}|u|^2\mu_{r}^{\alpha-1}\,dx \notag\\
			&\le
			\lambda M^{2}r^{8}H(r).
			\label{eq:Delta-w-bound}
		\end{align}
		Taking \eqref{eq:w2-bound-by-H} and \eqref{eq:Delta-w-bound} into consideration, from \eqref{eq:Nprime-preliminary}, we obtain 
	\begin{equation}\label{eq:N-prime-lower}
			N'(r)
			\ge
			-\frac{1}{4\alpha\lambda\,r}
			-
			\frac{\lambda M^{2}}{4\alpha}r^{7}
			\qquad\text{for all }0<r<R.
		\end{equation} 
        Differentiating \eqref{eq:Ntilde-definition-v} and using \eqref{eq:N-prime-lower}, we obtain
		\[
		\widetilde N'(r)
		=
		N'(r)+\frac{1}{4\alpha\lambda\,r}+\frac{\lambda M^{2}}{4\alpha}r^{7}
		\ge 0.
		\]
		Then the lemma is arrived.
	\end{proof}

    In addition, we need to show a bound for $L(r)$.
	\begin{lemma}\label{lem:L-bound}
		For every \(0<r<R\), one has
		\begin{equation}\label{eq:L-bound}
			|L(r)|
			\le
			\left(
			\frac{1}{2\sqrt{\lambda}}
			+
			\frac{Mr^{4}\sqrt{\lambda}}{2}
			\right)H(r).
		\end{equation}
	\end{lemma}
	
	\begin{proof}
		Since \(\Delta u=w\) and \(\Delta w=-Vu\), we have
		\begin{equation}\label{eq:L-rewrite}
			L(r)
			=
			\int_{B_{r}}uw\,\mu_{r}^{\alpha}\,dx
			-
			\lambda r^{4}\int_{B_{r}}Vuw\,\mu_{r}^{\alpha}\,dx.
		\end{equation}
		For the first term, by Cauchy--Schwarz inequality and \eqref{eq:H-definition}, we show that
		\begin{align}
			\left|\int_{B_{r}}uw\,\mu_{r}^{\alpha}\,dx\right|
			&\le
			\left(\int_{B_{r}}|u|^2\mu_{r}^{\alpha-1}\,dx\right)^{1/2}
			\left(\int_{B_{r}}|w|^2\mu_{r}^{\alpha+1}\,dx\right)^{1/2}
			\notag\\
			&\le
			\frac{1}{\sqrt{\lambda}}
			\left(\int_{B_{r}}|u|^2\mu_{r}^{\alpha-1}\,dx\right)^{1/2}
			\left(\lambda r^{4}\int_{B_{r}}|w|^2\mu_{r}^{\alpha-1}\,dx\right)^{1/2}
			\notag\\
			&\le
			\frac{1}{2\sqrt{\lambda}}H(r).
			\label{eq:first-L-term}
		\end{align}
	In the same manner, we get
		\begin{align}
			\left|\lambda r^{4}\int_{B_{r}}Vuw\,\mu_{r}^{\alpha}\,dx\right|
			&\le
			\lambda Mr^{4}
			\left(\int_{B_{r}}|u|^2\mu_{r}^{\alpha-1}\,dx\right)^{1/2}
			\left(\int_{B_{r}}|w|^2\mu_{r}^{\alpha+1}\,dx\right)^{1/2}
			\notag\\
			&\le
			Mr^{4}\sqrt{\lambda}
			\left(\int_{B_{r}}|u|^2\mu_{r}^{\alpha-1}\,dx\right)^{1/2}
			\left(\lambda r^{4}\int_{B_{r}}|w|^2\mu_{r}^{\alpha-1}\,dx\right)^{1/2}
			\notag\\
			&\le
			\frac{Mr^{4}\sqrt{\lambda}}{2}H(r).
			\label{eq:second-L-term}
		\end{align}
		The combination of \eqref{eq:L-rewrite}, \eqref{eq:first-L-term}, and \eqref{eq:second-L-term} yields \eqref{eq:L-bound}.
	\end{proof}

    Thanks to the monotonicity results of $\widetilde N(r)$, we are able to derive the three-ball inequality for \(H\).
\begin{lemma}\label{lem:weighted-three-ball-H-bilap}
	Assume
$	0<r_{1}<r_{2}<2r_{2}<r_{3}<R.$
	Define
	\begin{equation}\label{eq:lemma6-a0-b0}
		a_{0}:=\log\frac{r_{3}}{2r_{2}},
		\qquad
		b_{0}:=\log\frac{2r_{2}}{r_{1}},
		\qquad
		\kappa:=\frac{a_{0}}{a_{0}+b_{0}}
		=
		\frac{\log r_{3}-\log(2r_{2})}{\log r_{3}-\log r_{1}}.
	\end{equation}
	Also denote
	\begin{equation}\label{eq:lemma6-Xi}
		\Xi(r_{1},r_{2},r_{3})
		:=
		\frac{
			a_{0}\,\left(((\log r_{1})^{2}-\log(2r_{2}))^{2}\right)
			+
			b_{0}\,\left((\log r_{3})^{2}-(\log(2r_{2}))^{2}\right)
		}{
			a_{0}+b_{0}
		},
	\end{equation}
	and
	\begin{equation}\label{eq:lemma6-Lambda}
		\Lambda(r_{1},r_{2},r_{3})
		:=
		\frac{a_{0}b_{0}}{a_{0}+b_{0}}.
	\end{equation}
	Then there exists a universal constant \(C>0\) such that
	\begin{equation}\label{eq:lemma6-conclusion}
		H(2r_{2})
		\le
		\exp\!\left(
		C\left[
		\frac{\Xi(r_{1},r_{2},r_{3})}{\alpha^{2}\lambda}
		+
		\frac{\lambda M^{2}R^{8}}{\alpha^{2}}
		+
		\frac{\Lambda(r_{1},r_{2},r_{3})}{\alpha\sqrt{\lambda}}
		+
		\frac{MR^{4}\sqrt{\lambda}}{\alpha}
		\right]
		\right)
		H^{\kappa}(r_{1})H^{1-\kappa}(r_{3}).
	\end{equation}
\end{lemma}

\begin{proof}
	Recall (\ref{eq:H-prime}) gives that
	\begin{equation}\label{eq:lemma6-Hprime-over-H-start}
		\frac{H'(r)}{H(r)}
		=
		\frac{2\alpha+n-2}{r}
		+
		\frac{N(r)}{\alpha r}
		+
		\frac{L(r)}{\alpha r\,H(r)},
	\end{equation}
	and $\widetilde N(r)$ in (\ref{eq:Ntilde-definition-v}) is nondecreasing.
	By Lemma~\ref{lem:L-bound}, we have for every \(0<r<R\),
	\begin{equation}\label{eq:lemma6-L-over-H}
		\left|\frac{L(r)}{H(r)}\right|
		\le
		\frac{1}{2\sqrt{\lambda}}+\frac{Mr^{4}\sqrt{\lambda}}{2}.
	\end{equation}
	
	From \eqref{eq:Ntilde-definition-v} and \eqref{eq:lemma6-Hprime-over-H-start},  using \eqref{eq:lemma6-L-over-H}, we obtain the two-sided bounds
	\begin{align}
		\frac{H'(r)}{H(r)}
		&\ge
		\frac{2\alpha+n-2}{r}
		+
		\frac{\widetilde N(r)}{\alpha r}
		-
		\frac{\log r}{4\alpha^{2}\lambda\,r}
		-
		\frac{\lambda M^{2}}{32\alpha^{2}}r^{7}
		-
		\frac{1}{2\alpha r\sqrt{\lambda}}
		-
		\frac{M\sqrt{\lambda}}{2\alpha}r^{3},
		\label{eq:lemma6-two-sided-lower}
		\\
		\frac{H'(r)}{H(r)}
		&\le
		\frac{2\alpha+n-2}{r}
		+
		\frac{\widetilde N(r)}{\alpha r}
		-
		\frac{\log r}{4\alpha^{2}\lambda\,r}
		-
		\frac{\lambda M^{2}}{32\alpha^{2}}r^{7}
		+
		\frac{1}{2\alpha r\sqrt{\lambda}}
		+
		\frac{M\sqrt{\lambda}}{2\alpha}r^{3}.
		\label{eq:lemma6-two-sided-upper}
	\end{align}
	
	We are going to integrate these inequalities on the two intervals \([2r_{2},r_{3}]\) and \([r_{1},2r_{2}]\).
	Since \(\widetilde N\) is nondecreasing, then
	$\widetilde N(r)\ge \widetilde N(2r_{2})$ for every $r\in[2r_{2},r_{3}].$
	Set
	\begin{equation}\label{eq:lemma6-A}
		A:=2\alpha+n-2+\frac{\widetilde N(2r_{2})}{\alpha}.
	\end{equation}
	Integrating \eqref{eq:lemma6-two-sided-lower} from \(2r_{2}\) to \(r_{3}\), we get
	\begin{align}
		\log\frac{H(r_{3})}{H(2r_{2})}
		&=
		\int_{2r_{2}}^{r_{3}}\frac{H'(r)}{H(r)}\,dr
		\notag\\
		&\ge
		A\int_{2r_{2}}^{r_{3}}\frac{dr}{r}
		-
		\frac{1}{4\alpha^{2}\lambda}\int_{2r_{2}}^{r_{3}}\frac{\log r}{r}\,dr
		-
		\frac{\lambda M^{2}}{32\alpha^{2}}\int_{2r_{2}}^{r_{3}}r^{7}\,dr
		\notag\\
		&\quad
		-
		\frac{1}{2\alpha\sqrt{\lambda}}\int_{2r_{2}}^{r_{3}}\frac{dr}{r}
		-
		\frac{M\sqrt{\lambda}}{2\alpha}\int_{2r_{2}}^{r_{3}}r^{3}\,dr.
		\label{eq:lemma6-integrate-lower-start}
	\end{align}
	The first integral on the right hand side of (\ref{eq:lemma6-integrate-lower-start}) is given as 
	\begin{align}
		\int_{2r_{2}}^{r_{3}}\frac{dr}{r}
		&=
		\log\frac{r_{3}}{2r_{2}}
		=
		a_{0}.
		\label{eq:lemma6-int-4}
	\end{align}
	From\eqref{eq:lemma6-int-4} and \eqref{eq:lemma6-integrate-lower-start}, we obtain
	\begin{align}
		\log\frac{H(r_{3})}{H(2r_{2})}
		&\ge
		Aa_{0}
		-
		E_{3},
		\label{eq:lemma6-lower-final}
	\end{align}
	where
	\begin{align}
		E_{3}
		&:=
		\frac{1}{8\alpha^{2}\lambda}
		\left((\log r_{3})^{2}-(\log(2r_{2}))^{2} \right)
		+
		\frac{\lambda M^{2}}{256\alpha^{2}}
		\bigl(r_{3}^{8}-(2r_{2})^{8}\bigr)
		+
		\frac{a_{0}}{2\alpha\sqrt{\lambda}}
		\notag \\&\quad +
		\frac{M\sqrt{\lambda}}{8\alpha}
		\bigl(r_{3}^{4}-(2r_{2})^{4}\bigr).
		\label{eq:lemma6-E3}
	\end{align}
	From \eqref{eq:lemma6-lower-final}, it follows that
	\begin{equation}\label{eq:lemma6-A-upper}
		A
		\le
		\frac{1}{a_{0}}
		\log\frac{H(r_{3})}{H(2r_{2})}
		+
		\frac{E_{3}}{a_{0}}.
	\end{equation}
	
	Integrating \eqref{eq:lemma6-two-sided-upper} from \(r_{1}\) to \(2r_{2}\), by monotonicity of $\widetilde N(r)$, we obtain
	\begin{align}
		\log\frac{H(2r_{2})}{H(r_{1})}
		&=
		\int_{r_{1}}^{2r_{2}}\frac{H'(r)}{H(r)}\,dr
		\notag\\
		&\le
		A\int_{r_{1}}^{2r_{2}}\frac{dr}{r}
		-
		\frac{1}{4\alpha^{2}\lambda}
		\int_{r_{1}}^{2r_{2}}\frac{\log r}{r}\,dr
		-
		\frac{\lambda M^{2}}{32\alpha^{2}}\int_{r_{1}}^{2r_{2}}r^{7}\,dr
		\notag\\
		&\quad
		+
		\frac{1}{2\alpha\sqrt{\lambda}}\int_{r_{1}}^{2r_{2}}\frac{dr}{r}
		+
		\frac{M\sqrt{\lambda}}{2\alpha}\int_{r_{1}}^{2r_{2}}r^{3}\,dr.
		\label{eq:lemma6-integrate-upper-start}
	\end{align}
Note that 
	\begin{align}
		\int_{r_{1}}^{2r_{2}}\frac{dr}{r}
		&=
		\log\frac{2r_{2}}{r_{1}}
		=
		b_{0},
		\label{eq:lemma6-int-8}
	\end{align}
	Substituting  \eqref{eq:lemma6-int-8} into \eqref{eq:lemma6-integrate-upper-start}, we can achieve
	\begin{align}
		\log\frac{H(2r_{2})}{H(r_{1})}
		&\le
		Ab_{0}
		+
		E_{1},
		\label{eq:lemma6-upper-final}
	\end{align}
with
	\begin{align}
		E_{1}
		&:=
		\frac{1}{8\alpha^{2}\lambda}
		\left(\log r_{1})^{2}-(\log(2r_{2}))^{2}\right)
		-
		\frac{\lambda M^{2}}{256\alpha^{2}}
		\bigl((2r_{2})^{8}-r_{1}^{8}\bigr)
		+
		\frac{b_{0}}{2\alpha\sqrt{\lambda}}
		\notag \\ &\quad +
		\frac{M\sqrt{\lambda}}{8\alpha}
		\bigl((2r_{2})^{4}-r_{1}^{4}\bigr).
		\label{eq:lemma6-E1}
	\end{align}

	Substituting \eqref{eq:lemma6-A-upper} into \eqref{eq:lemma6-upper-final}, we have
	\begin{align}
		\log\frac{H(2r_{2})}{H(r_{1})}
		&\le
		\frac{b_{0}}{a_{0}}
		\log\frac{H(r_{3})}{H(2r_{2})}
		+
		\frac{b_{0}}{a_{0}}E_{3}
		+
		E_{1}.
		\label{eq:lemma6-before-rearrange}
	\end{align}
	Rearranging \eqref{eq:lemma6-before-rearrange}, we obtain
	\begin{align}
		(a_{0}+b_{0})\log H(2r_{2})
		&\le
		a_{0}\log H(r_{1})
		+
		b_{0}\log H(r_{3})
		+
		a_{0}E_{1}
		+
		b_{0}E_{3}.
		\label{eq:lemma6-log-convex}
	\end{align}
	Taking exponential to  \eqref{eq:lemma6-log-convex}, and using \eqref{eq:lemma6-a0-b0}, we get
	\begin{equation}\label{eq:lemma6-pre-final}
		H(2r_{2})
		\le
		\exp\!\left(
		\frac{a_{0}E_{1}+b_{0}E_{3}}{a_{0}+b_{0}}
		\right)
		H(r_{1})^{\kappa}H(r_{3})^{1-\kappa}.
	\end{equation}
	
	We estimate the weighted $\frac{a_{0}E_{1}+b_{0}E_{3}}{a_{0}+b_{0}}$ term by term.
	For the logarithmic-square terms, by definition \eqref{eq:lemma6-Xi}, it holds that
	\begin{align}
		\frac{1}{8\alpha^{2}\lambda}\big[\frac{a_{0}}{a_{0}+b_{0}}
		\left((\log r_{1})^{2}-(\log(2r_{2}))^{2}\right)
		+
		\frac{b_{0}}{a_{0}+b_{0}}
		\left((\log r_{3})^{2}-(\log(2r_{2}))^{2}\right)\big]
		=
		\frac{\Xi(r_{1},r_{2},r_{3})}{8\alpha^{2}\lambda}.
		\label{eq:lemma6-Xi-term}
	\end{align}
	It is true that
	\begin{align}
		\frac{a_{0}}{a_{0}+b_{0}}
		\cdot
		\frac{\lambda M^{2}}{256\alpha^{2}}
		\bigl((2r_{2})^{8}-r_{1}^{8}\bigr)
		+
		\frac{b_{0}}{a_{0}+b_{0}}
		\cdot
		\frac{\lambda M^{2}}{256\alpha^{2}}
		\bigl(r_{3}^{8}-(2r_{2})^{8}\bigr)
		\le
		\frac{\lambda M^{2}R^{8}}{256\alpha^{2}}.
		\label{eq:lemma6-r8-term}
	\end{align}
	
	For the \(\frac{1}{\alpha\sqrt{\lambda}}\)-terms, using \eqref{eq:lemma6-Lambda},
	\begin{align}
		\frac{a_{0}}{a_{0}+b_{0}}\cdot \frac{b_{0}}{2\alpha\sqrt{\lambda}}
		+
		\frac{b_{0}}{a_{0}+b_{0}}\cdot \frac{a_{0}}{2\alpha\sqrt{\lambda}}
		=
		\frac{\Lambda(r_{1},r_{2},r_{3})}{\alpha\sqrt{\lambda}},
		\label{eq:lemma6-Lambda-term}
	\end{align}
where $\Lambda(r_{1},r_{2},r_{3})$ is given in \eqref{eq:lemma6-Lambda}.

	For the \(r^{4}\)-terms,
	we obtain
	\begin{align}
		\frac{a_{0}}{a_{0}+b_{0}}
		\cdot
		\frac{M\sqrt{\lambda}}{8\alpha}
		\bigl((2r_{2})^{4}-r_{1}^{4}\bigr)
		+
		\frac{b_{0}}{a_{0}+b_{0}}
		\cdot
		\frac{M\sqrt{\lambda}}{8\alpha}
		\bigl(r_{3}^{4}-(2r_{2})^{4}\bigr)
		\le
		\frac{MR^{4}\sqrt{\lambda}}{8\alpha}.
		\label{eq:lemma6-r4-term}
	\end{align}
	
	Combining \eqref{eq:lemma6-Xi-term}--\eqref{eq:lemma6-r4-term}, we conclude that
	\begin{equation}\label{eq:lemma6-exponent-bound}
		\frac{a_{0}E_{1}+b_{0}E_{3}}{a_{0}+b_{0}}
		\le
		C\left[
		\frac{\Xi(r_{1},r_{2},r_{3})}{\alpha^{2}\lambda}
		+
		\frac{\lambda M^{2}R^{8}}{\alpha^{2}}
		+
		\frac{\Lambda(r_{1},r_{2},r_{3})}{\alpha\sqrt{\lambda}}
		+
		\frac{MR^{4}\sqrt{\lambda}}{\alpha}
		\right].
	\end{equation}
	Substituting the upper bounds \eqref{eq:lemma6-exponent-bound} into \eqref{eq:lemma6-pre-final} yields \eqref{eq:lemma6-conclusion}.
\end{proof}

We aim to get rid of the weight function $\mu_r^{\alpha-1}$ in $H(r)$. We define
	\begin{equation}\label{eq:prop1-h}
		h(r):=\int_{B_{r}}\bigl(|u|^2+\lambda r^{4}|w|^2\bigr)\,dx.
	\end{equation}
    We can prove the following three-ball inequality.
\begin{proposition}\label{prop:unweighted-three-ball-bilap}
Assume $	0<r_{1}<r_{2}<2r_{2}<r_{3}<R.$ It holds that
	\begin{equation}\label{eq:prop1-conclusion}
		h(r_{2})
		\le
		\exp\!\left(
		C\left[
		\alpha
		+
		\frac{\Xi(r_{1},r_{2},r_{3})}{\alpha^{2}\lambda}
		+
		\frac{\lambda M^{2}R^{8}}{\alpha^{2}}
		+
		\frac{\Lambda(r_{1},r_{2},r_{3})}{\alpha\sqrt{\lambda}}
		+
		\frac{MR^{4}\sqrt{\lambda}}{\alpha}
		\right]
		\right)
		h(r_{1})^{\kappa}h(r_{3})^{1-\kappa},
	\end{equation}
	where \(\kappa\) is given by \eqref{eq:lemma6-a0-b0}.
\end{proposition}

\begin{proof}
	On one hand, since \(\mu_{r}(x)\le r^{2}\) for every \(x\in B_{r}\), we have
	\begin{align}
		H(r)
		&=
		\int_{B_{r}}\bigl(|u|^2+\lambda r^{4}|w|^2\bigr)\mu_{r}^{\alpha-1}\,dx
		\le
		r^{2\alpha-2}h(r).
		\label{eq:prop1-H-upper}
	\end{align}
	On the other hand, for every \(x\in B_{r_{2}}\), it is true that
	\[
	\mu_{2r_{2}}(x)=(2r_{2})^{2}-|x|^{2}\ge 4r_{2}^{2}-r_{2}^{2}=3r_{2}^{2}.
	\]
	Therefore
	\begin{align}
		H(2r_{2})
		&=
		\int_{B_{2r_{2}}}\bigl(|u|^2+\lambda (2r_{2})^{4}|w|^2\bigr)\mu_{2r_{2}}^{\alpha-1}\,dx
		\notag\\
		&\ge
		(3r_{2}^{2})^{\alpha-1}
		\int_{B_{r_{2}}}\bigl(|u|^2+\lambda (2r_{2})^{4}|w|^2\bigr)\,dx
		\notag\\
		&=
		(3r_{2}^{2})^{\alpha-1}h(r_{2}).
		\label{eq:prop1-H-lower}
	\end{align}

	Applying Lemma~\ref{lem:weighted-three-ball-H-bilap}, and then using \eqref{eq:prop1-H-upper} for \(r=r_{1}\) and \(r=r_{3}\), we obtain
	\begin{align}
		h(r_{2})
		&\le
		(3r_{2}^{2})^{1-\alpha}
		\exp\!\left(
		C\left[
		\frac{\Xi(r_{1},r_{2},r_{3})}{\alpha^{2}\lambda}
		+
		\frac{\lambda M^{2}R^{8}}{\alpha^{2}}
		+
		\frac{\Lambda(r_{1},r_{2},r_{3})}{\alpha\sqrt{\lambda}}
		+
		\frac{MR^{4}\sqrt{\lambda}}{\alpha}
		\right]
		\right)
		\notag\\
		&\quad\times
		r_{1}^{\kappa(2\alpha-2)}
		r_{3}^{(1-\kappa)(2\alpha-2)}
		h^{\kappa}(r_{1})h^{1-\kappa}(r_{3}).
		\label{eq:prop1-pre-final}
	\end{align}
	Note from \eqref{eq:lemma6-a0-b0} that
	\[
	\kappa\log r_{1}+(1-\kappa)\log r_{3}=\log(2r_{2}),
	\]
	and therefore
	\begin{equation}\label{eq:prop1-radius-identity}
		r_{1}^{\kappa(2\alpha-2)}r_{3}^{(1-\kappa)(2\alpha-2)}=(2r_{2})^{2\alpha-2}.
	\end{equation}
	Substituting \eqref{eq:prop1-radius-identity} into \eqref{eq:prop1-pre-final}, we have
	\begin{align}
		h(r_{2})
		&\le
		e^{C\alpha}
		\exp\!\left(
		C\left[
		\frac{\Xi(r_{1},r_{2},r_{3})}{\alpha^{2}\lambda}
		+
		\frac{\lambda M^{2}R^{8}}{\alpha^{2}}
		+
		\frac{\Lambda(r_{1},r_{2},r_{3})}{\alpha\sqrt{\lambda}}
		+
		\frac{MR^{4}\sqrt{\lambda}}{\alpha}
		\right]
		\right)
		h(r_{1})^{\kappa}h(r_{3})^{1-\kappa}.
		\label{eq:prop1-radius-factor}
	\end{align}
    Then \eqref{eq:prop1-conclusion} follows.
\end{proof}

Furthermore, we optimize the three-ball inequality by choosing appropriate $\alpha$ and $\lambda$.
\begin{corollary}\label{cor:optimized-three-ball-bilap-v}
	Assume that \(MR^{4}\ge 1\).
	Then there exists a universal constant \(C>0\) such that
	\begin{equation}\label{eq:cor2-conclusion-v}
		h(r_{2})
		\le
		\exp\!\left(
		C\Bigl[
		(1+\Xi(r_{1},r_{2},r_{3})+\Lambda(r_{1},r_{2},r_{3}))(MR^{4})^{1/3}
		\Bigr]
		\right)
		h(r_{1})^{\kappa}h(r_{3})^{1-\kappa},
	\end{equation}
	where \(h\) is defined by \eqref{eq:prop1-h} and \(\kappa\) is given by \eqref{eq:lemma6-a0-b0}.
	
\end{corollary}

\begin{proof}
Fix $r_1, r_2,$ and $ r_3$, we want to optimize the exponent in \eqref{eq:prop1-conclusion}.  We may assume 
\begin{align}
\frac{\Xi(r_{1},r_{2},r_{3})}{\alpha^2 \lambda}\geq \frac{\Lambda(r_{1},r_{2},r_{3})}{\alpha  \sqrt{\lambda}}
\label{assum-1}
\end{align}
or
\begin{align}
\frac{\lambda M^{2}R^{8}}{\alpha^{2}}\geq\frac{\sqrt{\lambda}MR^{4}}{\alpha}.
\label{assum-2}
\end{align}
Otherwise, \begin{align} 
\frac{1}{\alpha \sqrt{\lambda}}\leq C, \quad \mbox{and} \quad \frac{\sqrt{\lambda}MR^{4}}{\alpha}\leq C. \label{new-o}\end{align} 

We first 
 assume the case that both (\ref{assum-1}) and (\ref{assum-2}) hold. To optimize the exponent in \eqref{eq:prop1-conclusion}, we let
\begin{align}
  \alpha =\frac{\lambda M^{2}R^{8}}{\alpha^{2}}=\frac{C}{\alpha^2 \lambda}.
\end{align}
Thus, we choose
\begin{equation}\label{eq:cor2-choice-v}
		\lambda:=(MR^{4})^{-1},
		\qquad
		\alpha:=\lfloor (MR^{4})^{1/3}\rfloor+2.
	\end{equation}
	Substituting \eqref{eq:cor2-choice-v} into the exponent in \eqref{eq:prop1-conclusion}, we obtain
	\begin{align}
		\alpha
		&+
		\frac{\Xi(r_{1},r_{2},r_{3})}{\alpha^{2}\lambda}
		+
		\frac{\lambda M^{2}R^{8}}{\alpha^{2}}
		+
	\frac{\Lambda(r_{1},r_{2},r_{3})}{\alpha\sqrt{\lambda}}
		+
		\frac{MR^{4}\sqrt{\lambda}}{\alpha} \notag \\
		&\le
		C\Bigl[
(1+\Xi(r_{1},r_{2},r_{3})+\Lambda(r_{1},r_{2},r_{3}))(MR^{4})^{1/3}
		\Bigr].
		\label{eq:cor2-exponent-final}
	\end{align}
    If we assume either (\ref{assum-1}) or  (\ref{assum-2}) holds, by the same arguments as discussed before, we can show the exact upper bound in (\ref{eq:cor2-exponent-final}).

At last, we consider the case (\ref{new-o}) holds.
    Since $\alpha$ is assumed to be large, then ${C\alpha}$ will play the dominant role in the exponent in \eqref{eq:prop1-conclusion}. Solving (\ref{new-o}) will give that $\alpha\geq C(M R^2)^{1/2}$. Thus, we get
    \begin{align}
		\alpha+
		\frac{\Xi(r_{1},r_{2},r_{3})}{\alpha^{2}\lambda}
		+
		\frac{\lambda M^{2}R^{8}}{\alpha^{2}}
		+
	\frac{\Lambda(r_{1},r_{2},r_{3})}{\alpha\sqrt{\lambda}}
		+
		\frac{MR^{4}\sqrt{\lambda}}{\alpha} 
		\geq C(M R^2)^{1/2}.
        \end{align}

    Therefore, the case  that (\ref{assum-1}) or  (\ref{assum-2}) holds provides a better bound.
	From the upper bound in \eqref{eq:cor2-exponent-final},  \eqref{eq:prop1-conclusion} gives  \eqref{eq:cor2-conclusion-v}. 
	
\end{proof}

We are ready to give the $L^\infty$ three-ball inequality for $u$.
	
	\begin{proposition}
	    \label{thm:Linfty-three-ball-1}
		Let
		$0<r_{1}<2r_1<r_{2}<4r_{2}<r_{3}<2r_{3}<R.$ Denote
		\begin{equation}\label{eq:def-alpha3-beta3-v}
			\alpha_{3}:=\log\frac{r_{3}}{4r_{2}},
			\qquad
			\beta_{3}:=\log\frac{4r_{2}}{r_{1}},
			\qquad
			\theta:=\frac{\alpha_{3}}{\alpha_{3}+\beta_{3}}.
		\end{equation}
		Also define
		\begin{equation}\label{eq:def-Xi3-Lambda3}
			\Xi_{3}:=\Xi(r_{1},2r_{2},r_{3}),
			\qquad
			\Lambda_{3}:=\Lambda(r_{1},2r_{2},r_{3}),
		\end{equation}
		where
		$\Xi(\cdot,\cdot,\cdot),\ \Lambda(\cdot,\cdot,\cdot)$
		are exactly the quantities introduced in Lemma \ref{lem:weighted-three-ball-H-bilap}.
		There exists a constant \(C=C(n,R)>0\) such that
		\begin{align}
			\|u\|_{L^\infty(B_{r_{2}})}
			\le{}&	\left(  \frac{r_{1}^{\theta}  r_{3}^{1-\theta} }{r_{2} }\right)^{n/2}
			\exp\!\left(
			C\bigl[(1+\Xi_{3}+\Lambda_{3})(MR^{4})^{1/3}\bigr]
			\right)
	\|u\|_{L^\infty(B_{2r_1})}^{\theta}
			\|u\|_{L^\infty(B_{2r_3})}^{1-\theta}.
			\label{eq:Linfty-three-ball-v}
		\end{align}
\end{proposition}
	
	\begin{proof}
		Applying the local \(L^\infty\)-\(L^2\) estimate, see Lemma 3 in \cite{Z18}, we obtain
		\begin{equation}\label{linfty3ball}
			\|u\|_{L^\infty(B_{r_{2}})}
			\le
			C(n)\,(1+M(2r_{2})^{4})^{C(n)}(2r_{2}-r_{2})^{-n/2}\,
			\|u\|_{L^{2}(B_{2r_{2}})}.
		\end{equation}
		It follows that
		\begin{equation}\label{eq:middle-by-h}
			\|u\|_{L^\infty(B_{r_{2}})}
			\le
			C(1+Mr_{2}^{4})^{C}r_{2}^{-n/2}\,h(2r_{2})^{1/2}.
		\end{equation}
		
		Since
		$0<r_{1}<r_{2}<4r_{2}<r_{3}<R,$
		from Corollary \ref{cor:optimized-three-ball-bilap-v}, we have 
		\begin{equation}\label{eq:h-three-ball}
			h(2r_{2})
			\le
			\exp\!\left(
			C\bigl[(1+\Xi_{3}+\Lambda_{3})(MR^{4})^{1/3}\bigr]
			\right)
			h(r_{1})^{\theta}h(r_{3})^{1-\theta},
		\end{equation}
		where \(\theta\) is given by \eqref{eq:def-alpha3-beta3-v}. Taking square roots in
		\eqref{eq:h-three-ball}, and substituting the result into \eqref{eq:middle-by-h}, we get
		\begin{align}
			\|u\|_{L^\infty(B_{r_{2}})}
			\le{}&
			C(1+Mr_{2}^{4})^{C}r_{2}^{-n/2}
			\exp\!\left(
			C\bigl[(1+\Xi_{3}+\Lambda_{3})(MR^{4})^{1/3}+1\bigr]
			\right)
			\notag\\
			&\times
			h(r_{1})^{\theta/2}h(r_{3})^{(1-\theta)/2}.
			\label{eq:middle-before-endpoints-bounded}
		\end{align}
		
		Let \(r<\rho<R\). From the definition of \(h(r)\),
		it holds that
		\begin{equation}\label{eq:h-general-1}
			h(r)^{1/2}
			\le
			\|u\|_{L^{2}(B_{r})}
			+
			r^{2}\|u\|_{W^{2,2}(B_{r})}.
		\end{equation}
		Also, by the local \(W^{2,2}\)-\(L^\infty\) estimate  from Lemma \ref{cor:local-W22-Linfty-correct} in the Appendix, it holds that
		\begin{equation}\label{eq:W22-by-Linfty}
			\|u\|_{W^{2,2}(B_{r})}
			\le
			C(n,R)\,(1+M\rho^{4})(\rho-r)^{-2}\rho^{n/2}
			\|u\|_{L^\infty(B_{\rho})}.
		\end{equation}
		Substituting \eqref{eq:W22-by-Linfty} into
		\eqref{eq:h-general-1} gives that
		\begin{equation}\label{eq:h-general-2}
			h(r)^{1/2}
			\le
			C(n,R)\,
			\Bigl[
			r^{n/2}
			+r^{2}(1+M\rho^{4})(\rho-r)^{-2}\rho^{n/2}
			\Bigr]
			\|u\|_{L^\infty(B_{\rho})}.
		\end{equation}

		Applying \eqref{eq:h-general-2} first with \((r,\rho)=(r_{1}, 2r_1)\), we get
		\begin{equation}\label{eq:h-r1-bounded}
			h(r_{1})^{1/2}
			\le
			C(n,R) M r_1^{n/2}\,
			\|u\|_{L^\infty(B_{2r_1})}.
		\end{equation}
		Similarly, applying \eqref{eq:h-general-2} with \((r,\rho)=(r_{3}, 2r_3)\) yields that
		\begin{equation}\label{eq:h-r3-bounded}
			h(r_{3})^{1/2}
			\le
			C(n,R)M r_3^{n/2}\,
			\|u\|_{L^\infty(B_{2r_{3}})}.
		\end{equation}
		
The combination of \eqref{eq:h-r1-bounded}, \eqref{eq:h-r3-bounded} 
		and \eqref{eq:middle-before-endpoints-bounded} gives
		\begin{align}
	\|u\|_{L^\infty(B_{r_{2}})}
			\le{}&
			\left(  \frac{r_{1}^{\theta}  r_{3}^{1-\theta} }{r_{2} }\right)^{n/2}
			\exp\!\left(
			C\bigl[(1+\Xi_{3}+\Lambda_{3})(MR^{4})^{1/3}\bigr]
			\right)
	\|u\|_{L^\infty(B_{2r_1})}^{\theta}
	\|u\|_{L^\infty(B_{2r_3})}^{1-\theta}.
		\end{align}
		Thus, we arrive at \eqref{eq:Linfty-three-ball-v}.
	\end{proof}


We show the proof of Theorem \ref{thm:Linfty-three-ball} from the last Proposition.
      \begin{proof}[Proof of Theorem \ref{thm:Linfty-three-ball}]
The estimates (\ref{kaokao-1}) and $\theta_0$   follows from \eqref{eq:Linfty-three-ball-v} and  \eqref{eq:def-alpha3-beta3-v} in the last  proposition by setting $r_1=r/2$, $r_2=3r$ and $r_3=15r$.
        \end{proof}

    \section{A variant frequency function and vanishing order for bounded potentials}

In this section, we study a variant frequency function from the one in the last section, and apply it to show the vanishing order of solutions in Theorem \ref{thm-2}. We still use the same notation as the last section to define frequency functions. Let $w=\Delta u$, $\alpha\ge 2$ and $\lambda>0.$
	
	For \(0<r<R\), we define
	\begin{equation}\label{eq:H-definition-v}
H(r):=\int_{B_{r}}\bigl(|u|^2+\lambda r^{3}|w|^2\bigr)\mu_{r}^{\alpha-1}\,dx.
	\end{equation}
    Note that $H(r)=\int_{B_{r}}\bigl(|u|^2+\lambda r^{4}|w|^2\bigr)\mu_{r}^{\alpha-1}\,dx$ in the last section. Similarly, denote
	\begin{equation}\label{eq:D-definition-v}
		D(r):=\int_{B_{r}}\bigl(|\nabla u|^{2}+\lambda r^{3}|\nabla w|^{2}\bigr)\mu_{r}^{\alpha}\,dx
		+
		3\lambda\alpha r^{3}\int_{B_{r}}|w|^2\mu_{r}^{\alpha-1}\,dx,
	\end{equation}
	\begin{equation}\label{eq:L-definition-v}
		L(r):=\int_{B_{r}}\bigl(u\Delta u+\lambda r^{3}w\Delta w\bigr)\mu_{r}^{\alpha}\,dx,
	\end{equation}
	and
	\begin{equation}\label{eq:N-definition-v}
		N(r):=\frac{D(r)}{H(r)}.
	\end{equation}

To study the monotonicity  of (\ref{eq:N-definition-v}), we need to study the derivatives of $H(r)$ and $D(r)$. Since some of the calculations are similar to those in the last section, we will not repeat all the calculations and just show the differences.

Applying the arguments in Lemma \ref{lem:H-prime} to $H(r)$ in (\ref{eq:H-definition-v}), we obtain that
		\begin{equation}\label{eq:H-prime-v}
			H'(r)
			=
			\frac{2\alpha+n-2}{r}H(r)+\frac{D(r)+L(r)}{\alpha r}.
		\end{equation}

Applying the arguments in Lemma \ref{lem:D-prime} to $D(r)$ in (\ref{eq:D-definition-v}), we have
\begin{align}
			D'(r)
			&=
			\frac{2\alpha+n-2}{r}D(r)-
			\frac{1}{4\alpha r}
			\int_{B_{r}}
			\bigl(|w|^2+\lambda r^{3}(\Delta w)^{2}\bigr)\mu_{r}^{\alpha+1}\,dx \notag\\
			&\quad
			+
			\frac{4\alpha}{r}
			\int_{B_{r}}
			\left[
			\left(x\cdot \nabla u-\frac{w\mu_{r}}{4\alpha}\right)^{2}
			+
			\lambda r^{3}
			\left(x\cdot \nabla w+\frac32 w-\frac{\Delta w\,\mu_{r}}{4\alpha}\right)^{2}
			\right]
			\mu_{r}^{\alpha-1}\,dx.
			\label{eq:D-prime-v}
		\end{align}

Thanks to (\ref{eq:H-prime-v}) and (\ref{eq:D-prime-v}), we can show the monotone property of $N(r)$.
\begin{lemma}\label{lem:N-prime-v}
		Define
		\begin{equation}\label{eq:Ntilde-definition}
			\widetilde N(r)
			:=
			N(r)
			+
			\frac{r}{4\alpha\lambda}
			+
			\frac{\lambda M^{2}}{28\alpha}r^{7}.
		\end{equation}
		Then \(\widetilde N(r)\) is nondecreasing on \((0,R)\).
\end{lemma}
	\begin{proof}
		Define
		\begin{align}
			J(r)
			&:=
			2\alpha
			\int_{B_{r}}
			\left[
			u\left(x\cdot \nabla u-\frac{w\mu_{r}}{4\alpha}\right)
			+
			\lambda r^{3}w\left(x\cdot \nabla w+\frac32 w-\frac{\Delta w\,\mu_{r}}{4\alpha}\right)
			\right]
			\mu_{r}^{\alpha-1}\,dx.
			\label{eq:J-definition}
		\end{align}
		We first compute \(J(r)\). By integration by parts arguments, we have
		\begin{equation}\label{eq:u-J-identity-v}
			2\alpha
			\int_{B_{r}}
			u\left(x\cdot \nabla u-\frac{w\mu_{r}}{4\alpha}\right)\mu_{r}^{\alpha-1}\,dx
			=
			\int_{B_{r}}|\nabla u|^{2}\mu_{r}^{\alpha}\,dx
			+
			\frac12\int_{B_{r}}u\Delta u\,\mu_{r}^{\alpha}\,dx.
		\end{equation}
		Similarly,
		\begin{align}\label{eq:w-J-identity-v}
			2\alpha\lambda r^{3}
			\int_{B_{r}}
			w\left(x\cdot \nabla w+\frac32 w-\frac{\Delta w\,\mu_{r}}{4\alpha}\right)\mu_{r}^{\alpha-1}\,dx
			=
			\lambda r^{3}\int_{B_{r}}|\nabla w|^{2}\mu_{r}^{\alpha}\,dx
			&+
			3\lambda\alpha r^{3}\int_{B_{r}}|w|^2\mu_{r}^{\alpha-1}\,dx \notag \\
			&+
			\frac12\lambda r^{3}\int_{B_{r}}w\Delta w\,\mu_{r}^{\alpha}\,dx.
		\end{align}
		Adding \eqref{eq:u-J-identity-v} and \eqref{eq:w-J-identity-v}, and recalling \eqref{eq:D-definition-v} and \eqref{eq:L-definition-v}, we get
		\begin{equation*}\label{eq:J-equals-D-plus-half-L-v}
			J(r)=D(r)+\frac12L(r).
		\end{equation*}
		Hence
		\begin{equation}\label{eq:D-and-D-plus-L-v}
			D(r)=J(r)-\frac12L(r),
			\qquad
			D(r)+L(r)=J(r)+\frac12L(r).
		\end{equation}
		
		Now, since \(N(r)=D(r)/H(r)\), then
		\begin{equation}\label{eq:H2Nprime-v}
			H(r)^{2}N'(r)=D'(r)H(r)-H'(r)D(r).
		\end{equation}
		Substituting \eqref{eq:H-prime-v}, \eqref{eq:D-prime-v}, and \eqref{eq:D-and-D-plus-L-v} into \eqref{eq:H2Nprime-v}, we obtain
		\begin{align}
			H(r)^{2}N'(r)
			&=
			\frac{4\alpha}{r}
			\Biggl[
			H(r)
			\int_{B_{r}}
			\left(
			\left(x\cdot \nabla u-\frac{w\mu_{r}}{4\alpha}\right)^{2}
			+
			\lambda r^{3}
			\left(x\cdot \nabla w+\frac32 w-\frac{\Delta w\,\mu_{r}}{4\alpha}\right)^{2}
			\right)
			\mu_{r}^{\alpha-1}\,dx
			\notag\\
			&\quad
			-
			\left(
			\int_{B_{r}}
			\left[
			u\left(x\cdot \nabla u-\frac{w\mu_{r}}{4\alpha}\right)
			+
			\lambda r^{3}w\left(x\cdot \nabla w+\frac32 w-\frac{\Delta w\,\mu_{r}}{4\alpha}\right)
			\right]
			\mu_{r}^{\alpha-1}\,dx
			\right)^{2}
			\Biggr]
			\notag\\
			&\quad
			+
			\frac{1}{4\alpha r}
			\left(
			L(r)^{2}
			-
			H(r)\int_{B_{r}}\bigl(|w|^2+\lambda r^{3}(\Delta w)^{2}\bigr)\mu_{r}^{\alpha+1}\,dx
			\right).
			\label{eq:H2Nprime-expanded-v}
		\end{align}
		The Cauchy--Schwarz inequality indicates that the large bracket in \eqref{eq:H2Nprime-expanded-v} is nonnegative. Therefore
		\begin{equation}\label{eq:Nprime-preliminary-v}
			N'(r)
			\ge
			-\frac{1}{4\alpha r\,H(r)}
			\int_{B_{r}}\bigl(|w|^2+\lambda r^{3}(\Delta w)^{2}\bigr)\mu_{r}^{\alpha+1}\,dx.
		\end{equation}
		
		We now estimate the two terms on the right-hand side of \eqref{eq:Nprime-preliminary-v}. Since \(\mu_{r}(x)\le r^{2}\), 
        we have
		\begin{equation}\label{eq:w2-bound-v}
	\int_{B_{r}}|w|^2\mu_{r}^{\alpha+1}\,dx
			\le
			r^{4}\int_{B_{r}}|w|^2\mu_{r}^{\alpha-1}\,dx.
		\end{equation}
		Recalling the definition \eqref{eq:H-definition-v} of \(H(r)\), we deduce from \eqref{eq:w2-bound-v} that
		\begin{equation}\label{eq:w2-bound-by-H-v}
	\int_{B_{r}}|w|^2\mu_{r}^{\alpha+1}\,dx
			\le
			\frac{r}{\lambda}H(r).
		\end{equation}
		
		Next, since \(w=\Delta u\) and
	$\Delta w=-Vu,$
then
		\begin{align}
			\lambda r^{3}\int_{B_{r}}(\Delta w)^{2}\mu_{r}^{\alpha+1}\,dx
			&=
			\lambda r^{3}\int_{B_{r}}V^{2}|u|^2\mu_{r}^{\alpha+1}\,dx \notag\\
			&\le
			\lambda M^{2}r^{7}\int_{B_{r}}|u|^2\mu_{r}^{\alpha-1}\,dx \notag\\
			&\le
			\lambda M^{2}r^{7}H(r).
			\label{eq:Delta-w-bound-v}
		\end{align}
		Substituting \eqref{eq:w2-bound-by-H-v} and \eqref{eq:Delta-w-bound-v} into \eqref{eq:Nprime-preliminary-v}, we obtain 
       \begin{equation}\label{eq:N-prime-lower-v}
			N'(r)
			\ge
			-\frac{1}{4\alpha\lambda}
			-
			\frac{\lambda M^{2}}{4\alpha}r^{6}
			\qquad\text{for all }0<r<R.
		\end{equation} 
      Hence  \(\widetilde N'(r)\geq 0\). The lemma is arrived
	\end{proof}

\begin{remark} We want to point out why we choose $H(r)=\int_{B_{r}}\bigl(|u|^2+\lambda r^{3}|w|^2\bigr)\mu_{r}^{\alpha-1}\,dx$ in this section instead of $H(r)=\int_{B_{r}}\bigl(|u|^2+\lambda r^{4}|w|^2\bigr)\mu_{r}^{\alpha-1}\,dx.$ We aim to get rid of the ``bad term" i.e. $\frac{1}{4\alpha\lambda}\log r$,
in (\ref{eq:Ntilde-definition-v}), which actually arises from (\ref{eq:w2-bound-by-H}).
Let us 
compare the differences of (\ref{eq:w2-bound-by-H-v}) and (\ref{eq:w2-bound-by-H}). The estimate (\ref{eq:w2-bound-by-H-v}) leads to the term $\frac{r}{4\alpha\lambda}$ in (\ref{eq:Ntilde-definition}) for the monotonicity. However,  the estimate (\ref{eq:w2-bound-by-H}) gives the term $\frac{1}{4\alpha\lambda}\log r$ 
in (\ref{eq:Ntilde-definition-v}) to show the monotonicity. The  $\log r$ term in (\ref{eq:Ntilde-definition-v}) is a troublesome term to estimate the vanishing order. While the term $\frac{r}{4\alpha\lambda}$ in (\ref{eq:Ntilde-definition}) is harmless. The fact that $\mu^{\alpha+1}\leq r^4 \mu^{\alpha-1}$ leads us to introduce $H(r)=\int_{B_{r}}\bigl(|u|^2+\lambda r^{3}|w|^2\bigr)\mu_{r}^{\alpha-1}\,dx$ to reach
	(\ref{eq:w2-bound-by-H-v}).		
		\end{remark}

Adapting the arguments in Proposition \ref{prop:unweighted-three-ball-bilap}, and Corollary \ref{cor:optimized-three-ball-bilap-v}, Proposition \ref{thm:Linfty-three-ball-1}, we obtain the following lemma.
\begin{lemma}\label{thm:Linfty-three-ball-v} Let $0<r_{1}<r_{2}<4r_{2}<r_{3}<2r_{3}<R.$ Denote
	\begin{equation}\label{eq:def-alpha3-beta3}
		\alpha_{3}:=\log\frac{r_{3}}{4r_{2}},
		\qquad
		\beta_{3}:=\log\frac{4r_{2}}{r_{1}},
		\qquad
		\theta:=\frac{\alpha_{3}}{\alpha_{3}+\beta_{3}}.
	\end{equation}
	Let \(MR^{4}\ge1\).  Define
	\begin{equation}\label{eq:def-A-rho-1}
		\mathcal A_{M,R}(r,\rho)
		:=
		r^{n/2}+(MR^{3})^{1/2}r^{3/2}(\rho-r)^{-2}\rho^{n/2}.
	\end{equation}
    There exists \(C=C(n,R)>0\) such that
		\begin{align}\label{eq:Linfty-three-ball-f}
			\|u\|_{L^{\infty}(B_{r_{2}})}
			\le{}&
			r_{2}^{-n/2}
\exp\!\left(C\bigl((MR^{4})^{1/3}+1\bigr)\right)\notag\\
			&\times
			\mathcal A_{M,R}(r_{1},2r_{1})^{\theta}
			\mathcal A_{M,R}(r_{3}, 2r_3)^{1-\theta}
			\|u\|_{L^{\infty}(B_{2r_{1}})}^{\theta}
			\|u\|_{L^{\infty}(B_{2r_3})}^{1-\theta},
		\end{align}
		where \(\theta\) is given in \eqref{eq:def-alpha3-beta3}.
    
\end{lemma}

\begin{proof}
		Assume $0<r_{1}<r_{2}<2r_{2}<r_{3}<R.
		$
		Let
		\begin{equation}\label{eq:lemma6-a0-b0-v}
			a_{0}:=\log\frac{r_{3}}{2r_{2}},
			\qquad
			b_{0}:=\log\frac{2r_{2}}{r_{1}},
			\qquad
			\kappa:=\frac{a_{0}}{a_{0}+b_{0}}
			=
			\frac{\log r_{3}-\log(2r_{2})}{\log r_{3}-\log r_{1}}.
		\end{equation}
		Define
		\begin{equation}\label{eq:prop1-h-v}
		h(r):=\int_{B_{r}}\bigl(|u|^2+\lambda r^{3}|w|^2\bigr)\,dx.
		\end{equation}
Following the arguments in  Proposition \ref{prop:unweighted-three-ball-bilap} gives that
		\begin{equation}\label{eq:prop1-conclusion-v}
			h(r_{2})
			\le
			\exp\!\left(
			C\left[
			\alpha
			+
			\frac{R}{\alpha^{2}\lambda}
			+
			\frac{\lambda M^{2}R^{7}}{\alpha^{2}}
			+
			\frac{\sqrt R}{\alpha\sqrt{\lambda}}
			+
			\frac{M\sqrt{\lambda}\,R^{7/2}}{\alpha}
			\right]
			\right)
			h(r_{1})^{\kappa}h(r_{3})^{1-\kappa},
		\end{equation}
		where \(\kappa\) is given by \eqref{eq:lemma6-a0-b0-v}.

  We want to optimize the three-ball inequality in (\ref{eq:prop1-conclusion-v}).
		From the reasoning in Corollary \ref{cor:optimized-three-ball-bilap-v}, by choosing
		\begin{equation}\label{eq:cor2-choice}
			\lambda:=(MR^{3})^{-1},
			\qquad
			\alpha:=\lfloor (MR^{4})^{1/3}\rfloor+2,
		\end{equation}
   where $\lfloor \cdot  \rfloor$ denotes the   greatest integer function,
	 there exists a universal constant \(C>0\) such that
		\begin{equation}\label{eq:cor2-conclusion}
			h(r_{2})
			\le
			\exp\!\left(
			C\bigl((MR^{4})^{1/3}+1\bigr)
			\right)
			h(r_{1})^{\kappa}h(r_{3})^{1-\kappa}.
		\end{equation}
		 Applying the arguments in Proposition \ref{thm:Linfty-three-ball-1}, we can show (\ref{eq:Linfty-three-ball-f}).

\end{proof}

	In the rest of the section, we apply the  three-ball inequality in $L^\infty$ norm to prove the vanishing order results.
	 We aim to show the dependence on
	\(M\) and \(\mathbb{M}\) in the argument

	\begin{proof}[Proof of Theorem \ref{thm-2}]
		We apply the \(L^\infty\) three-ball inequality stated in Lemma \ref{thm:Linfty-three-ball-v}.
		Fix
		$0<r\le \frac{1}{16},$
		and choose
        \begin{align}\label{eq:choice-radii-v}
r_{1}:=\frac{r}{2},\quad
	r_{2}:=1, \quad
			r_{3}:=8.
        \end{align}
		By the definitions in Lemma \ref{thm:Linfty-three-ball-v}, we have 
		\begin{align}
			\alpha_{3}
			=
			\log\frac{r_{3}}{4r_{2}}
			=
			\log\frac{8}{4}
			=
			\log 2,
		\qquad
			\beta_{3}
			&=
			\log\frac{4r_{2}}{r_{1}}
			=
			\log\frac{4}{r/2}
			=
			\log\frac{8}{r},
		\end{align}
        \begin{align}
			\theta
			&=
			\frac{\alpha_{3}}{\alpha_{3}+\beta_{3}}
			=
			\frac{\log 2}{\log 2+\log(8/r)}
			=
			\frac{\log 2}{\log(16/r)}.
			\label{eq:theta-v}
		\end{align}
		
From (\ref{eq:def-A-rho-1}), it holds that
		\begin{equation}\label{eq:A-small-bound-v}
			\mathcal A_{M,R}\!\left(r_1, 2r_1\right)=A_{M,R}\!\left(\frac{r}{2},r\right)
			\le
			C(n)\,(1+\sqrt{M})\,r^{n/2-1/2}.
		\end{equation}
	and	
    \begin{align}
		\mathcal A_{M,R}(r_{3}, 2r_{3})=\mathcal A_{M,R}(8,16)
			\le
			C(n)\,(1+\sqrt{M}).
			\label{eq:A-large-bound-v}
		\end{align}
	
Thanks to  Lemma \ref{thm:Linfty-three-ball-v}, we derive that
		\begin{align}
			\|u\|_{L^\infty(B_{1})}
			\le
			\exp\!\left(
			C(M^{1/3}+1)
			\right)
			\mathcal A_{M,R}\!\left(\frac{r}{2},r\right)^{\theta}
			\mathcal A_{M,R}(8,16)^{1-\theta}
			\|u\|_{L^\infty(B_{r})}^{\theta}
			\|u\|_{L^\infty(B_{16})}^{1-\theta}.
			\label{eq:three-ball}
		\end{align}
		Recall the assumption (\ref{eq:intro-second-order-normalization})
     that
		\[
		\|u\|_{L^\infty(B_{1})}\ge 1,
		\qquad
		\|u\|_{L^\infty(B_{9})}\le \|u\|_{L^\infty(B_{16})}\le \mathbb{M},
		\]
		and using \eqref{eq:A-small-bound-v} and \eqref{eq:A-large-bound-v}, we obtain
		\begin{align}
			1
			\le
			\exp\!\left(
			C(M^{1/3}+1)
			\right)
			\Bigl(C(n)(1+\sqrt{M})\,r^{\frac{n-1}{2}}\Bigr)^{\theta}
	\Bigl(C(n)\mathbb{M}(1+\sqrt{M})\Bigr)^{1-\theta}
	\|u\|_{L^\infty(B_{r})}^{\theta}.
			\label{eq:after-three-ball-v}
		\end{align}
		
		The estimate \eqref{eq:after-three-ball-v} simplifies to
		\begin{align}
			1
			\le{}&
			C(n)\mathbb{M}^{1-\theta}
			\exp\!\left(
			C(M^{1/3}+1)
			\right)
	\|u\|_{L^\infty(B_{r})}^{\theta}.
			\label{eq:before-rearranging-v}
		\end{align}

		Since \(0<r<1\), we have 
	\begin{equation}\label{eq:drop-favorable-factor-v}
			\|u\|_{L^\infty(B_{r})}
			\ge
			\exp\!\left(
			-\frac{C(M^{1/3}+\log \mathbb{M})}{\theta}
			\right).
		\end{equation}
	Note that
		$\theta^{-1}
		=
		\frac{\log(16/r)}{\log 2},$
		Substituting this into \eqref{eq:drop-favorable-factor-v}, we arrive at
\begin{equation}\label{eq:after-theta-v}
			\|u\|_{L^\infty(B_{r})}
			\ge
		 r^{{C(M^{1/3}+\log \mathbb{M})}},
		\end{equation}
		which
         shows that the order of vanishing for solutions at the origin is at most $C(M^{1/3}+\log \mathbb{M})$.
Since the frequency functions and then the three-ball inequalities in Lemma \ref{thm:Linfty-three-ball-v} are translation invariant, the same arguments apply to the balls centered at any point $x_0\in B_{1/2}$. Thus, this finishes the proof of Theorem.
        
	\end{proof}
	
\begin{remark}\label{rem-1}
    If we work on the vanishing order using the frequency functions $H(r)=\int_{B_{r}}\bigl(|u|^2+\lambda r^{4}|w|^2\bigr)\mu_{r}^{\alpha-1}\,dx$  in the last section, using the arguments of Theorem \ref{thm-2}, we will end up with a
\begin{equation}\label{eq:final-quantitative-uniqueness}
			\|u\|_{L^\infty(B_{r})}
			\ge
			\exp\!\left(
			-C\,(M^{1/3}+\log \mathbb{M})\Bigl(\log\frac{1}{r}\Bigr)^{2}
			\right).
		\end{equation}
as $\Xi_{3}$ in (\ref{eq:def-Xi3-Lambda3}) contains the $\log \frac{1}{r}$ by choosing the radius in the three-ball inequality in (\ref{eq:choice-radii-v}).
\end{remark}

	
\section{Three-Ball inequality for bi-Laplace equation  with H\"older continuous potential}

Throughout this section, we study the quantitative unique continuation for the equation
\begin{equation}
    \Delta^2 u + Vu = 0 \qquad \mathrm{in} ~~ B_R 
    \label{eqn-holder}
\end{equation}
with potential
\begin{equation}
    \|V\|_{C^{0,\beta}(B_R)}\leq G , \qquad \mbox{for} \ \beta \in (0,1).\label{VVV-1}
\end{equation}

We modify the potential $V(x)$ with the standard mollifier. Define
$$V_\varepsilon(x) = \rho_\varepsilon * V=\int_{\mathbb R^n}\rho_\varepsilon(x-y)V(y)\,dy,$$
where
	\begin{align}\label{stand-m}
	\rho_\varepsilon(x)=\frac{1}{\varepsilon^n}\rho\Bigl(\frac{x}{\varepsilon}\Bigr),
	\qquad
	\rho(x)\in C_0^\infty(B_1),
	\qquad
	\int_{B_1}\rho(x) dx=1.
	\end{align}

We show some quantitative estimates for $V_\varepsilon$ and $V_\varepsilon-V$ as the arguments in \cite{TWZ26}.
\begin{lemma} \label{lem:mollbound}
    It holds that
$\|V_\varepsilon\|_{L^\infty(B_{R-\varepsilon})}\le G$ and
	\begin{equation}\label{eq:moll-est}
		\|\nabla V_\varepsilon\|_{L^\infty(B_{R-\varepsilon})}\le C G\varepsilon^{\beta-1},
		\quad
        \|\Delta V_{\varepsilon}\|_{L^{\infty}(B_{R - \varepsilon})} \leq C G\varepsilon^{\beta - 2},
        \quad
		\|V-V_\varepsilon\|_{L^\infty(B_{R-\varepsilon})}\le C G \varepsilon^{\beta}
	\end{equation}
     for a constant $C=C(n)$.
	\end{lemma}
\begin{proof}
    For $x\in B_{R-\varepsilon}$, the convolution $V_\varepsilon(x)$ is supported in $B_R$.
		Hence we have 
        \begin{align}
		|V_\varepsilon(x)|
		\le \|V\|_{L^\infty(B_R)}\int_{\mathbb R^n}\rho_\varepsilon(y)\,dy
		\le G.
        \end{align}
Since $V_\varepsilon(x)$ is smooth in $B_{R-\varepsilon}$, it is true that
		\[
		\nabla V_\varepsilon(x)
		=\int_{B_R}\nabla\rho_\varepsilon(x-y)V(y)\,dy
		=\frac{1}{\varepsilon^{n+1}}\int_{B_R}\nabla\rho\Bigl(\frac{x-y}{\varepsilon}\Bigr)V(y)\,dy.
		\]
		Set $z=\frac{x-y}{\varepsilon}$. We get
		\[
		\nabla V_\varepsilon(x)
		=\int_{B_1}\varepsilon^{-1}\nabla\rho(z)\,V(x-\varepsilon z)\,dz.
		\]
By the fact that $\int_{B_1}\nabla\rho(z)\,dz=0$, it follows that
		\[
		\nabla V_\varepsilon(x)
		=\int_{B_1}\varepsilon^{-1}\nabla\rho(z)\bigl(V(x-\varepsilon z)-V(x)\bigr)\,dz.
		\]
The H\"older continuity of $V$ in \eqref{VVV-1} yields that
		\[
		|\nabla V_\varepsilon(x)|
		\le \varepsilon^{-1} G\int_{B_1}|\nabla\rho(z)|\,|\varepsilon z|^\beta\,dz
		\le C G \varepsilon^{\beta-1}.
		\]
        Since $V_\varepsilon(x)$ is smooth in $B_{R-\varepsilon}$, it also holds that
        \[
        \Delta V_\varepsilon(x)
        =\int_{B_R}\Delta\rho_\varepsilon(x-y)V(y)\,dy
        =\frac{1}{\varepsilon^{n+2}}\int_{B_R}\Delta\rho\Bigl(\frac{x-y}{\varepsilon}\Bigr)V(y)\,dy.
        \]
        Using the change of variables $z=\frac{x-y}{\varepsilon}$ again, we get
        \[
        \Delta V_\varepsilon(x)
        =\int_{B_1}\varepsilon^{-2}\Delta\rho(z)\,V(x-\varepsilon z)\,dz.
        \]
        By the fact that
        \[
        \int_{B_1}\Delta\rho(z)\,dz=0,
        \]
        it follows that
        \[
        \Delta V_\varepsilon(x)
        =\int_{B_1}\varepsilon^{-2}\Delta\rho(z)\bigl(V(x-\varepsilon z)-V(x)\bigr)\,dz.
        \]
        Thus, by the H\"older continuity of $V$ again,
        \[
        |\Delta V_\varepsilon(x)|
        \le \varepsilon^{-2}G\int_{B_1}|\Delta\rho(z)|\,|\varepsilon z|^\beta\,dz
        \le C G\varepsilon^{\beta-2}.
        \]
		Finally, we derive
		\[
		|V_\varepsilon(x)-V(x)|
		=\left|\int_{B_R}\bigl(V(y)-V(x)\bigr)\rho_\varepsilon(x-y)\,dy\right|
		\le G \int_{B_R}|x-y|^\beta\rho_\varepsilon(x-y)\,dy
		\le C G \varepsilon^\beta.
		\]
        This completes the proof of the lemma.
\end{proof}

We are going to define the frequency functions for bi-Laplace equations with H\"older continuous potentials.
	We keep the same notation as those in the last section. Let
	\begin{equation}\label{eq:Hdr-notation}
		w=\Delta u,
		\qquad
		\mu_{r}(x):=r^{2}-|x|^{2}.
	\end{equation}
Let
$f_\varepsilon=V_\varepsilon-V.$ Then \eqref{eqn-holder} is reduced to be
\begin{equation}\label{eq:Hdr-equation-for-w}
			\Delta w=-\Ve u+ f_{\varepsilon}u.
		\end{equation}
	For \(0<r<R-\varepsilon\), some fixed constants $\alpha>0$  and $\lambda>0$, introduce
\begin{equation}\label{eq:Hdr-H}
H(r):=\int_{B_{r}}\bigl(|u|^2+\lambda r^{4}|w|^2\bigr)\mu_{r}^{\alpha-1}\,dx,
	\end{equation}
	
	\begin{equation}\label{eq:Hdr-D}
	D(r):=\int_{B_{r}}\bigl(|\nabla u|^{2}+\lambda r^{4}|\nabla w|^{2}\bigr)\mu_{r}^{\alpha}\,dx+4\lambda\alpha r^{4}\int_{B_{r}}|w|^2\mu_{r}^{\alpha-1}\,dx,
	\end{equation}
	\begin{equation}\label{eq:Hdr-L}
		L(r):=\int_{B_{r}}\bigl(u\Delta u - \lambda r^{4}w\Ve u\bigr)\mu_{r}^{\alpha}\,dx,
	\end{equation}
	\begin{equation}\label{eq:Hdr-I}
		I(r):=D(r)+L(r),
		\qquad
		N(r):=\frac{I(r)}{H(r)}.
	\end{equation}
	We also split
	\begin{equation}\label{eq:Hdr-Lsplit}
		L(r)=L_{1}(r)+L_{2}(r),
	\end{equation}
	where
	\begin{equation}\label{eq:Hdr-L1L2}
		L_{1}(r):=\int_{B_{r}}uw\,\mu_{r}^{\alpha}\,dx,
		\qquad
		L_{2}(r):=- \lambda r^{4}\int_{B_{r}}w\Ve u\,\mu_{r}^{\alpha}\,dx.
	\end{equation}

    We first present the derivatives of $H(r)$ and $D(r).$
	\begin{lemma}\label{lem:Hdr-HD-prime}
		The following identities hold
		\begin{equation}\label{eq:Hdr-H-prime}
			H'(r)
			=
			\frac{2\alpha+n-2}{r}H(r)
			+
			\frac{1}{\alpha r}\left(D(r)+L(r)+ \int_{B_r}\lambda r^4wf_{\varepsilon}u\mu_r^{\alpha}\right),
		\end{equation}
	and
		\begin{align}
			D'(r)
			&=
			\frac{2\alpha+n-2}{r}D(r)
			-
			\frac{1}{4\alpha r}
			\int_{B_{r}}
			\bigl(|w|^2+\lambda r^{4}(\Ve u - f_{\varepsilon}u)^2\bigr)\mu_{r}^{\alpha+1}\,dx \notag\\
			&\quad
            +
			\frac{4\alpha}{r}
			\int_{B_{r}}
			\left[
			\left(
			x\cdot\nabla u-\frac{\mu_{r}}{4\alpha}w
			\right)^{2}
			+
			\lambda r^{4}
			\left(
			x\cdot\nabla w+2w+\frac{\mu_{r}}{4\alpha}(\Ve - f_{\varepsilon})u
			\right)^{2}
			\right]
			\mu_{r}^{\alpha-1}\,dx.
			\label{eq:Hdr-D-prime-V}
		\end{align}
	\end{lemma}
	
	\begin{proof}
	The proof of (\ref{eq:Hdr-H-prime}) in the lemma follows from the arguments of Lemma \ref{lem:H-prime} by using Lemma \ref{lem:weighted-differentiation}. Specifically, we just replace $\triangle w$ in \eqref{eq:Hw-prime}  and \eqref{eq:lambda-r4-Hw-prime} by $-V_\varepsilon u+f_\varepsilon u$ to achieve \eqref{eq:Hdr-H-prime}.

The identity \eqref{eq:Hdr-D-prime-V} follows from the same estimates in Lemma \ref{lem:D-prime} by using the fact $\triangle w= -V_\varepsilon u+f_\varepsilon u$ in \eqref{eq:D-prime}.
	\end{proof}

Different from the bounded potential case, we need to consider the derivative of $L$ in the H\"older continuous potential case.
	\begin{lemma}\label{lem:Hdr-L-prime}
		The derivatives of \(L_{1}\), \(L_{2}\), and \(L\) are given by
		\begin{align}
			L_{1}'(r)
			&=
			\frac{2\alpha+n}{r}L_{1}(r)
			+
			\frac{1}{2(\alpha+1)r}
			\int_{B_{r}}
			\bigl(
			2\nabla u\cdot\nabla w+|w|^2-(\Ve - f_{\varepsilon}) |u|^2
			\bigr)\mu_{r}^{\alpha+1}\,dx,
			\label{eq:Hdr-L1-prime}
			\\
			L_{2}'(r)
			&=
			\frac{2\alpha+n+4}{r}L_{2}(r)
			-
			\frac{\lambda r^{4}}{2(\alpha+1)r}
			\int_{B_{r}}
			\Bigl(
			-\Ve^2|u|^2 + \Ve \fe |u|^2 + |w|^2\Ve + wu\,\Delta \Ve \nonumber \\
            &+2\Ve\,\nabla w\cdot \nabla u
            +2u\,\nabla w\cdot \nabla \Ve
            +2w\,\nabla \Ve\cdot \nabla u
			\Bigr)\mu_{r}^{\alpha+1}\,dx.
			\label{eq:Hdr-L2-prime}
		\end{align}
	Furthermore,
		\begin{align}
			L'(r)
			&=
			\frac{2\alpha+n}{r}L(r)
			+
			\frac{4}{r}L_{2}(r)
			+
			\frac{1}{2(\alpha+1)r}
			\int_{B_{r}}
			\bigl(
			2\nabla u\cdot\nabla w+|w|^2-(\Ve-f_{\varepsilon}) |u|^2
			\bigr)\mu_{r}^{\alpha+1}\,dx
			\notag\\
			&\qquad
			-
			\frac{\lambda r^{4}}{2(\alpha+1)r}
			\int_{B_{r}}
			\Bigl(
			-\Ve^2|u|^2 + \Ve \fe |u|^2 + |w|^2\Ve + wu\,\Delta \Ve \nonumber \\
            &\qquad +2\Ve\,\nabla w\cdot \nabla u
            +2u\,\nabla w\cdot \nabla \Ve
            +2w\,\nabla \Ve\cdot \nabla u
			\Bigr)\mu_{r}^{\alpha+1}\,dx.
			\label{eq:Hdr-L-prime}
		\end{align}
	\end{lemma}
	
	\begin{proof}
		Applying Lemma \ref{lem:weighted-differentiation} with \(\gamma=\alpha\) and \(f=uw\), we obtain
		\begin{equation}\label{eq:Hdr-L1-start}
			L_{1}'(r)
			=
			\frac{2\alpha+n}{r}L_{1}(r)
			+
			\frac{1}{2(\alpha+1)r}
			\int_{B_{r}}\Delta(uw)\,\mu_{r}^{\alpha+1}\,dx.
		\end{equation}
		
		Since
		$\Delta u=w$ and $\Delta w=-\Ve u + f_{\varepsilon}u$,
		it follows that
		\[
		\Delta(uw)
		=
		-(\Ve - f_{\varepsilon})|u|^2+|w|^2+2\nabla u\cdot\nabla w.
		\]
		Substituting the last identity into \eqref{eq:Hdr-L1-start} gives \eqref{eq:Hdr-L1-prime}.
		Next we consider the derivatives of $L_2$. Introduce
		\[
		F_1(r):=\int_{B_{r}}w\Ve u\,\mu_{r}^{\alpha}\,dx.
		\]
		Then \(L_{2}(r)= - \lambda r^{4}F_1(r)\). Differentiating gives 
		\[
		L_{2}'(r)=-4\lambda r^{3}F_1(r)-\lambda r^{4}F_1'(r).
		\]
		Applying Lemma \ref{lem:weighted-differentiation} to \(F_1\), we get
		\[
		F_1'(r)
		=
		\frac{2\alpha+n}{r}F_1(r)
		+
		\frac{1}{2(\alpha+1)r}
		\int_{B_{r}}\Delta(w\Ve u)\,\mu_{r}^{\alpha+1}\,dx.
		\]
		Therefore
		\begin{equation}\label{eq:Hdr-L2-middle}
			L_{2}'(r)
			=
			\frac{2\alpha+n+4}{r}L_{2}(r)
			-
			\frac{\lambda r^{4}}{2(\alpha+1)r}
			\int_{B_{r}}\Delta(w\Ve u)\,\mu_{r}^{\alpha+1}\,dx.
		\end{equation}
        We consider the second integration in (\ref{eq:Hdr-L2-middle}). It holds that
        \begin{align}
            \Delta(w\Ve u)
            &= w\,\Delta(\Ve u)+\Ve u\,\Delta w+2\nabla w\cdot \nabla(\Ve u). \label{eq:step1}
            \end{align}
            Note that 
            \begin{align*}
            2\nabla w\cdot \nabla(\Ve u)
            &= 2\Ve \,\nabla w\cdot \nabla u+2u\,\nabla w\cdot \nabla \Ve.
            \end{align*}
          and
            \begin{align*}
            w\,\Delta(\Ve u)
            &= w\Ve\Delta u+wu\Delta \Ve+2w\,\nabla \Ve\cdot \nabla u.
            \end{align*}
            Substituting the last two inequalities into \eqref{eq:step1} gives that
            \begin{align}
            \Delta(w\Ve u)
            &= \Ve u\,\Delta w+w\Ve\,\Delta u+wu\,\Delta \Ve 
            +2\Ve\,\nabla w\cdot \nabla u
            +2u\,\nabla w\cdot \nabla \Ve
            +2w\,\nabla \Ve\cdot \nabla u \nonumber \\
            &= -\Ve^2|u|^2 + \Ve \fe |u|^2 + |w|^2\Ve + wu\,\Delta \Ve 
            +2\Ve\,\nabla w\cdot \nabla u
            +2u\,\nabla w\cdot \nabla \Ve
            \nonumber \\
            &\quad +2w\,\nabla \Ve\cdot \nabla u.
        \end{align}
		Inserting this identity into \eqref{eq:Hdr-L2-middle} yields \eqref{eq:Hdr-L2-prime}. The estimate
		\eqref{eq:Hdr-L-prime} follows by adding \eqref{eq:Hdr-L1-prime} and \eqref{eq:Hdr-L2-prime}.
	\end{proof}

    We will study the derivative of the quotient $\frac{L(r)}{H(r)}$. For convenience, we define the following notations.
		\begin{align}
			E_{1}(r)
			&:=
			-\frac{\lambda H(r)r^{4}}{2(\alpha+1)r}
			\int_{B_{r}}
			\bigl(
			2u\nabla w\cdot\nabla \Ve
			+
			2w\nabla u\cdot\nabla \Ve
			+
			uw\,\Delta \Ve
			\bigr)\mu_{r}^{\alpha+1}\,dx,
			\label{eq:Hdr-E1}
			\\
			E_{2}(r)
			&:=
			\frac{H(r)}{2(\alpha+1)r}
			\int_{B_{r}}2\nabla u\cdot\nabla w\,\mu_{r}^{\alpha+1}\,dx
			-
			\frac{\lambda H(r)r^{4}}{2(\alpha+1)r}
			\int_{B_{r}}2\Ve\nabla u\cdot\nabla w\,\mu_{r}^{\alpha+1}\,dx,
			\label{eq:Hdr-E2}
			\\
			E_{3}(r)
			&:=
			\frac{H(r)}{2(\alpha+1)r}
			\int_{B_{r}}|w|^2\mu_{r}^{\alpha+1}\,dx
			+
			\frac{\lambda H(r)r^{4}}{2(\alpha+1)r}
			\int_{B_{r}}\Ve^{2}|u|^2\mu_{r}^{\alpha+1}\,dx,
			\label{eq:Hdr-E3}
			\\
			E_{4}(r)
			&:=
			-\frac{H(r)}{2(\alpha+1)r}
			\int_{B_{r}}(\Ve - \fe)|u|^2\mu_{r}^{\alpha+1}\,dx
			-
			\frac{\lambda H(r)r^{4}}{2(\alpha+1)r}
			\int_{B_{r}}(\Ve |w|^2 + \Ve \fe |u|^2)\mu_{r}^{\alpha+1}\,dx,
			\label{eq:Hdr-E4}
			\\
			E_{5}(r)
			&:=
			-4H(r)r^{3}\lambda
			\int_{B_{r}}w\Ve u\,\mu_{r}^{\alpha}\,dx,
			\label{eq:Hdr-E5}
            \\
            E_{6}(r)
			&:=
			-\frac{L(r)\lambda r^3}{\alpha}\int_{B_r} w\fe u \mu_r^{\alpha}\,dx.
			\label{eq:Hdr-E6}
		\end{align}
        We investigate the identity for  $L'(r)H(r)-L(r)H'(r)$, which is used to study the derivative of $N(r)$.
	\begin{lemma}\label{lem:Hdr-LH-identity}
		It holds that
		\begin{align}\label{eq:Hdr-LH-main}
			L'(r)H(r)-L(r)H'(r)
			&=
			E_{1}(r)+E_{2}(r)+E_{3}(r)+E_{4}(r)+E_{5}(r)+E_{6}(r)
			\notag
            \\ &-\frac{L(r)(D(r)+L(r))}{\alpha r}
			+\frac{2}{r}L(r)H(r).
		\end{align}
	\end{lemma}
	
	\begin{proof}
   Recall that $L'$ in \eqref{eq:Hdr-L-prime}
		and $H'$ in \eqref{eq:Hdr-H-prime}. Direct calculations show that
		\begin{align*}
			L'(r)H(r)-L(r)H'(r)
			&=
			\Biggl[
			\frac{2\alpha+n}{r}L(r)
			+
			\frac{4}{r}L_{2}(r)
			+
			\frac{1}{2(\alpha+1)r}
			\int_{B_{r}}
			\bigl(
			2\nabla u\cdot\nabla w+|w|^2
			\notag\\
			& -(\Ve-f_{\varepsilon}) |u|^2
			\bigr)\mu_{r}^{\alpha+1}\,dx
			-
			\frac{\lambda r^{4}}{2(\alpha+1)r}
			\int_{B_{r}}
			\Bigl(
			-\Ve^2|u|^2 + \Ve \fe |u|^2 + |w|^2\Ve \nonumber \\
            & + wu\,\Delta \Ve +2\Ve\,\nabla w\cdot \nabla u
            +2u\,\nabla w\cdot \nabla \Ve
            +2w\,\nabla \Ve\cdot \nabla u
			\Bigr)\mu_{r}^{\alpha+1}\,dx
			\Biggr]H(r)
			\\
			&
			-
			L(r)\Biggl[
			\frac{2\alpha+n-2}{r}H(r)
			+
			\frac{1}{\alpha r}\left(D(r)+L(r)+ \int_{B_r}\lambda r^4wf_{\varepsilon}u\mu_r^{\alpha}\right)
			\Biggr].
		\end{align*}
        We aim to simplify the last inequality. Notice that
		\[
		\frac{4}{r}L_{2}(r)H(r)
		=
		-4H(r)r^{3}\lambda\int_{B_{r}}w\Ve u\,\mu_{r}^{\alpha}\,dx
		=
		E_{5}(r),
		\]
        and
        \[
        -\frac{L(r)}{\alpha r}\int_{B_r}\lambda r^4wf_{\varepsilon}u\mu_r^{\alpha} dx= E_6(r).
        \]
		It is true that
		\[
		\frac{2\alpha+n}{r}L(r)H(r)-\frac{2\alpha+n-2}{r}L(r)H(r)=\frac{2}{r}L(r)H(r).
		\]
        Taking into account of the definition of \(E_{1}(r),E_{2}(r),E_{3}(r),E_{4}(r)\), we arrive at  \eqref{eq:Hdr-LH-main}.
	\end{proof}

    In the next lemma, we study the bounds for \(L(r)\) and \(E_{1}(r),\dots,E_{6}(r)\).
	\begin{lemma}\label{lem:Hdr-E-bounds}
		For every \(0<r<R-\varepsilon\), the following estimates hold,
	\begin{equation}\label{eq:Hdr-L-over-H}
			\frac{|L(r)|}{H(r)}
			\le
			\frac{1+\lambda Gr^{4}}{2\sqrt{\lambda}},
		\end{equation}
		\begin{equation}\label{eq:Hdr-E1-bound}
			|E_{1}(r)|
			\le
			C\sqrt{\lambda}G\varepsilon^{\beta - 1}\,H^2(r)r^{4}
			+
			C\frac{\sqrt{\lambda}G\varepsilon^{\beta - 2}}{4(\alpha+1)}\,H^2(r)r^{5},
		\end{equation}
		\begin{equation}\label{eq:Hdr-E2-bound}
			|E_{2}(r)|
			\le
			\left(
			\frac{1}{2(\alpha+1)r\sqrt{\lambda}}
			+
			\frac{\sqrt{\lambda}Gr^{3}}{2(\alpha+1)}
			\right)H(r)D(r),
		\end{equation}
		\begin{equation}\label{eq:Hdr-E3-positive}
			E_{3}(r)\ge 0,
		\end{equation}
		\begin{equation}\label{eq:Hdr-E4-bound}
			E_{4}(r)\ge -\frac{GH^2(r)r^{3}}{2(\alpha+1)} - \frac{CG\varepsilon^{\beta}H^2(r)r^3}{2(\alpha + 1)} - \frac{C\lambda G^2 \varepsilon^{\beta}H^2(r)r^7}{2(\alpha + 1)},
		\end{equation}
		\begin{equation}\label{eq:Hdr-E5-bound}
			E_{5}(r)\ge -2\sqrt{\lambda}\,G\,H^2(r)r^{3},
		\end{equation}
        and
        \begin{equation}\label{eq:Hdr-E6-bound}
			E_{6}(r)\ge -\frac{CG\varepsilon^{\beta} H^2(r)r^3}{4 \alpha} - \frac{C\lambda G^2 \varepsilon^{\beta}H^2(r)r^7}{4\alpha}.
		\end{equation}
	\end{lemma}

	\begin{proof}
		Let us first prove \eqref{eq:Hdr-L-over-H}. It holds that
		\[
L(r)=\int_{B_{r}}uw\,\mu_{r}^{\alpha}\,dx-\lambda r^{4}\int_{B_{r}}\Ve uw\,\mu_{r}^{\alpha}\,dx.
		\]
	Thus
		\begin{align}
			|L(r)|
			&\le
			(1+\lambda Gr^{4})\int_{B_{r}}|u||w|\,\mu_{r}^{\alpha}\,dx.
            \label{wawa-1}
		\end{align}
			Since \(\mu_{r}\le r^{2}\), applying Cauchy-Schwartz inequality gives that
		\[
2\sqrt{\lambda}\,|u||w|\,\mu_{r}^{\alpha}
		\le
		\bigl(|u|^2+\lambda r^{4}|w|^2\bigr)\mu_{r}^{\alpha-1}.
		\]
 Integrating the last inequality indicates that
	\begin{align}
2\sqrt{\lambda}\int_{B_{r}}|u||w|\,\mu_{r}^{\alpha}\,dx
		\le H(r).
        \label{repeat-1}
		\end{align}
        Together with the last inequality and (\ref{wawa-1}), we show
	 \eqref{eq:Hdr-L-over-H}.
		
		Next we prove \eqref{eq:Hdr-E1-bound}. Set
		\[
		J_{1}(r):=
		\int_{B_{r}}
		\bigl(
		2u\nabla w\cdot\nabla \Ve
		+
		2w\nabla u\cdot\nabla \Ve
		+
		uw\,\Delta \Ve
		\bigr)\mu_{r}^{\alpha+1}\,dx.
		\]
		Since
		\[
		\operatorname{div}(2uw\nabla \Ve)
		=
		2u\nabla w\cdot\nabla \Ve
		+
		2w\nabla u\cdot\nabla \Ve
		+
		2uw\,\Delta \Ve,
		\]
		it follows that
		\[
		J_{1}(r)
		=
		\int_{B_{r}}\operatorname{div}(2uw\nabla \Ve)\,\mu_{r}^{\alpha+1}\,dx
		-
		\int_{B_{r}}uw\,\Delta \Ve\,\mu_{r}^{\alpha+1}\,dx.
		\]
		Using the fact that \(\mu_{r}^{\alpha+1}=0\) on \(\partial B_{r}\) and taking integration by parts yields
		\begin{align*}
			J_{1}(r)
			&=
			-2\int_{B_{r}}uw\,\nabla \Ve\cdot\nabla(\mu_{r}^{\alpha+1})\,dx
			-
			\int_{B_{r}}uw\,\Delta \Ve\,\mu_{r}^{\alpha+1}\,dx.
		\end{align*}
		Since
$\nabla(\mu_{r}^{\alpha+1})=-2(\alpha+1)x\,\mu_{r}^{\alpha},$
		we obtain
		\begin{align*}
			|J_{1}(r)|
			&\le
			4(\alpha+1)\int_{B_{r}}|u||w|\,|x|\,|\nabla \Ve|\,\mu_{r}^{\alpha}\,dx
			+
			\int_{B_{r}}|u||w|\,|\Delta \Ve|\,\mu_{r}^{\alpha+1}\,dx
			\\
			&\le
			4(\alpha+1)CG\varepsilon^{\beta - 1} r\int_{B_{r}}|u||w|\,\mu_{r}^{\alpha}\,dx
			+
			CG\varepsilon^{\beta - 2} r^{2}\int_{B_{r}}|u||w|\,\mu_{r}^{\alpha}\,dx.
		\end{align*}
		It follows from (\ref{repeat-1}) that
		\[
		|J_{1}(r)|
		\le
		\left(
		\frac{2(\alpha+1)CG\varepsilon^{\beta - 1}r}{\sqrt{\lambda}}
		+
		\frac{CG\varepsilon^{\beta - 2}r^{2}}{2\sqrt{\lambda}}
		\right)H(r).
		\]
		The definition of \(E_{1}(r)\) implies that
		\begin{align*}
			|E_{1}(r)|
			&\le
			\frac{\lambda H(r)r^{4}}{2(\alpha+1)r}|J_{1}(r)|
			\\
			&\le
	C\sqrt{\lambda}G\varepsilon^{\beta - 1} r^{4}H^2(r)
			+
\frac{C\sqrt{\lambda}G\varepsilon^{\beta - 2}r^{5}}{4(\alpha+1)}H^2(r).
		\end{align*}
		Hence we arrive at \eqref{eq:Hdr-E1-bound}.
		
		We now turn to the proof of  \eqref{eq:Hdr-E2-bound}. Since \(\mu_{r}^{\alpha+1}\le r^{2}\mu_{r}^{\alpha}\), the application of  Cauchy-Schwartz inequality gives that
		\begin{align*}
			\left|
			\int_{B_{r}}2\nabla u\cdot\nabla w\,\mu_{r}^{\alpha+1}\,dx
			\right|
			&\le
			2r^{2}
			\left(
			\int_{B_{r}}|\nabla u|^{2}\mu_{r}^{\alpha}\,dx
			\right)^{1/2}
			\left(
			\int_{B_{r}}|\nabla w|^{2}\mu_{r}^{\alpha}\,dx
			\right)^{1/2}.
		\end{align*}
		Using the fact that
		\[
		 \int_{B_{r}}|\nabla u|^{2}\mu_{r}^{\alpha}\,dx
		+\lambda r^{4}\int_{B_{r}}|\nabla w|^{2}\mu_{r}^{\alpha}\,dx \leq C D(r),
		\] 
		we have
		\[
		2r^{2}
		\left(
		\int_{B_{r}}|\nabla u|^{2}\mu_{r}^{\alpha}\,dx
		\right)^{1/2}
		\left(
		\int_{B_{r}}|\nabla w|^{2}\mu_{r}^{\alpha}\,dx
		\right)^{1/2}
		\le
		\frac{D(r)}{\sqrt{\lambda}}.
		\]
		Therefore
		\[
		\left|
		\int_{B_{r}}2\nabla u\cdot\nabla w\,\mu_{r}^{\alpha+1}\,dx
		\right|
		\le
		\frac{C D(r)}{\sqrt{\lambda}}.
		\]
		Similarly, the estimate of $V_\varepsilon$ in Lemma \ref{lem:mollbound} gives that
		\[
		\left|
		\int_{B_{r}}2\lambda r^4\Ve\nabla u\cdot\nabla w\,\mu_{r}^{\alpha+1}\,dx
		\right|
		\le
		\frac{C\lambda r^4G}{\sqrt{\lambda}}D(r).
		\]
		Substituting the last two inequalities into the definition of
		\(E_{2}(r)\), we obtain \eqref{eq:Hdr-E2-bound}.
		
		The definition of \(E_{3}(r)\) immediately implies the estimates  \eqref{eq:Hdr-E3-positive}.
		
		Next we prove \eqref{eq:Hdr-E4-bound}. Since \(\mu_{r}^{\alpha+1}\le r^{4}\mu_{r}^{\alpha-1}\), from the estimates of $V_\varepsilon$ in \eqref{eq:moll-est}, we have
		\begin{align*}
			E_{4}(r)
			&\ge
			-\frac{GH(r)}{2(\alpha+1)r}
			\int_{B_{r}}
			\bigl(
			|u|^2+\lambda r^{4}|w|^2
			\bigr)\mu_{r}^{\alpha+1}\,dx
            +\frac{H(r)}{2(\alpha + 1)r}\int_{B_r}f_{\varepsilon}|u|^2 \mu_r^{\alpha + 1}\,dx 
			\\
			&
            - \frac{\lambda H(r) r^4}{2(\alpha + 1)r}\int_{B_r}\Ve \fe |u|^2 \mu_r^{\alpha + 1}\,dx \\
			&
            \ge
			-\frac{GH(r)}{2(\alpha+1)r}\cdot r^{4}
			\int_{B_{r}}
			\bigl(
			|u|^2+\lambda r^{4}|w|^2
			\bigr)\mu_{r}^{\alpha-1}\,dx
            - \frac{CG\varepsilon^{\beta}H}{2(\alpha + 1)r} \cdot r^4\int_{B_r}|u|^2 \mu_r^{\alpha - 1}\,dx \\&- \frac{C\lambda G^2\varepsilon^{\beta}H(r)r^4}{2(\alpha + 1)r} \cdot r^4\int_{B_r}|u|^2 \mu_{r}^{\alpha - 1}\,dx
			\\
			&\geq
			-\frac{GH^2(r)r^{3}}{2(\alpha+1)} - \frac{CG\varepsilon^{\beta}H^2(r)r^3}{2(\alpha + 1)} - \frac{C\lambda G^2 \varepsilon^{\beta}H^2(r)r^7}{2(\alpha + 1)}.
		\end{align*}
		Thus, we arrive at \eqref{eq:Hdr-E4-bound}.
		Thanks to (\ref{repeat-1}), we have
		\[
		E_{5}(r)
		\ge
		-4\lambda GH(r)r^{3}
\int_{B_{r}}|u||w|\,\mu_{r}^{\alpha}\,dx
		\ge
		-2\sqrt{\lambda}\,G\,H^2(r)r^{3},
		\]
		which proves \eqref{eq:Hdr-E5-bound}.
        
        Finally,  thanks to (\ref{repeat-1}), the estimates of $f_\varepsilon=V_\varepsilon-V$ in (\ref{eq:moll-est}), and  \eqref{eq:Hdr-L-over-H}, we can  derive
        \begin{align*}
            E_6(r) &= -\frac{L}{\alpha r}\int_{B_r}\lambda r^4 w \fe u \mu_r^{\alpha}\,dx \\
            &\geq - \frac{1}{\alpha r}\left( \frac{1 + \lambda G r^4}{2\sqrt{\lambda}}\right)H(r)\lambda r^4 
CG\varepsilon^{\beta}\frac{H(r)}{2\sqrt{\lambda}} \\
            &\geq -\frac{CG\varepsilon^{\beta} H^2(r)r^3}{4 \alpha} - \frac{C\lambda G^2 \varepsilon^{\beta}H^2(r)r^7}{4\alpha}.
        \end{align*}
        This proves \eqref{eq:Hdr-E6-bound}. Hence we finish the proof of the lemma.
        
	\end{proof}

    We are going to study the frequency function $N(r)$.
	Next we consider a differential inequality for \(N(r)\). For ease of the presentation, we introduce the following notations,
		\begin{equation}\label{eq:Hdr-B-def}
			B(r):=\frac{1+\lambda Gr^{4}}{2\sqrt{\lambda}},
		\end{equation}
		\begin{equation}\label{eq:Hdr-b-def}
			b(r):=
			\frac{1}{2(\alpha+1)r\sqrt{\lambda}}
			+
			\frac{\sqrt{\lambda}Gr^{3}}{2(\alpha+1)},
		\end{equation}
		\begin{equation}\label{eq:Hdr-a-def}
			a(r):=
			\sqrt{\lambda}G\varepsilon^{\beta - 1}\,r^{4}
			+
			\frac{\sqrt{\lambda}G\varepsilon^{\beta - 2}r^{5}}{4(\alpha+1)}
			+
            \frac{Gr^{3}}{2(\alpha+1)}
            +
            \frac{CG\varepsilon^{\beta}r^3}{2(\alpha + 1)}
            +
            \frac{\lambda G^2 \varepsilon^{\beta}r^7}{2(\alpha + 1)}
            +
            \frac{G\varepsilon^{\beta}r^3}{4 \alpha}
            +
            \frac{\lambda G^2 \varepsilon^{\beta}r^7}{4\alpha},
		\end{equation}
		\begin{equation}\label{eq:Hdr-A-def-corrected}
			A(r):=b(r)+\frac{B(r)}{\alpha r},
		\end{equation}
		\begin{equation}\label{eq:Hdr-d-def}
			d(r):=
			\frac{1}{4\alpha\lambda r}
			+
			\frac{\lambda G^{2}r^{7}}{4\alpha}
            +
            \frac{\lambda G^2 \varepsilon^{2\beta}r^7}{4\alpha}
            +
            \frac{\lambda G^2 \varepsilon^{\beta}r^7}{4\alpha}.
		\end{equation}
		
		Set
		\begin{equation}\label{eq:Hdr-P-def-corrected}
			P(r):=
			d(r)+a(r)+\left(b(r)+\frac{2}{r}\right)B(r)+\frac{2B^2(r)}{\alpha r}.
		\end{equation}
			Specifically,
		\begin{align}
			A(r)
			&=
			\frac{1}{2(\alpha+1)r\sqrt{\lambda}}
			+
			\frac{\sqrt{\lambda}Gr^{3}}{2(\alpha+1)}
			+
			\frac{1+\lambda Gr^{4}}{2\alpha r\sqrt{\lambda}},
			\label{eq:Hdr-A-expanded}
		\end{align}
        \begin{align}
			P(r)
			&= \nonumber
			\frac{1}{4\alpha\lambda r}
			+
			\frac{\lambda G^{2}r^{7}}{4\alpha}
			+
			\sqrt{\lambda}\,CG\varepsilon^{\beta - 1}\,r^{4}
			+
			\frac{\sqrt{\lambda}\,CG\varepsilon^{\beta - 2} r^{5}}{4(\alpha+1)}\,
			+
            \frac{Gr^{3}}{2(\alpha+1)}
             \\
            &+ \frac{CG\varepsilon^{\beta}r^3}{2(\alpha + 1)}
            +
            \frac{C\lambda G^2 \varepsilon^{\beta}r^7}{2(\alpha + 1)}
            +
            \frac{CG\varepsilon^{\beta}r^3}{4 \alpha}
            +
            \frac{\lambda G^2 \varepsilon^{2\beta}r^7}{4\alpha}
            +
            \frac{\lambda G^2 \varepsilon^{\beta}r^7}{4\alpha} \nonumber \\
            &+
            \frac{C\lambda G^2 \varepsilon^{\beta}r^7}{4\alpha}
            +
			\left(
			\frac{1}{2(\alpha+1)r\sqrt{\lambda}}
			+
			\frac{\sqrt{\lambda}Gr^{3}}{2(\alpha+1)}
			+
			\frac{2}{r}
			\right)\frac{1+\lambda Gr^{4}}{2\sqrt{\lambda}}
			+
			\frac{(1+\lambda Gr^{4})^{2}}{2\alpha\lambda r}.
			\label{eq:Hdr-P-expanded}
		\end{align}
        We can show the following monotonicity result for $N(r).$
	\begin{proposition}\label{prop:Hdr-N-differential-corrected}
		Let $N(r)$ be given in (\ref{eq:Hdr-I}), $A(r)$ and $P(r)$ be given in (\ref{eq:Hdr-A-expanded}) and (\ref{eq:Hdr-P-expanded}).
		For every \(0<r<R-\varepsilon\), it holds that
		\begin{equation}\label{eq:Hdr-N-differential-corrected}
			N'(r)+A(r)N(r)\ge -P(r).
		\end{equation}
	
	\end{proposition}
	
	\begin{proof}
	From (\ref{eq:Hdr-I}), the quotient rule leads to 
		\begin{equation}\label{eq:Hdr-N-quotient}
			N'(r)
			=
			\frac{D'(r)H(r)-D(r)H'(r)}{H^2(r)}
			+
			\frac{L'(r)H(r)-L(r)H'(r)}{H^2(r)}.
		\end{equation}
	We study the first quotient in the right hand side of (\ref{eq:Hdr-N-quotient}). Since $D(r)$ and $H(r)$ are the same as
		those in the bounded-potential case in the last section. By the same argument as \eqref{eq:Nprime-preliminary}, we have
        \begin{equation}\label{eq:Hdr-Nprime-preliminary}
			\frac{D'(r)H(r) - H'(r)D(r)}{H^2(r)}
			\ge
			-\frac{1}{4\alpha r\,H(r)}
			\int_{B_{r}}\bigl(|w|^2+\lambda r^{4}(\Delta w)^{2}\bigr)\mu_{r}^{\alpha+1}\,dx.
		\end{equation}
       Since it holds that $\Delta w = -\Ve u + \fe u $, substituting it into \eqref{eq:Hdr-Nprime-preliminary} and using Lemma \ref{lem:mollbound}, we have
        \begin{align}
            \frac{D'(r)H(r) - H'(r)D(r)}{H^2(r)} &\geq -\frac{1}{4\alpha r\,H(r)}
			\int_{B_{r}}\bigl(|w|^2+\lambda r^{4}(-\Ve u+ \fe u)^{2}\bigr)\mu_{r}^{\alpha+1}\,dx \\
            &\geq -\frac{1}{4\alpha r\,H(r)}
			\int_{B_{r}}\bigl(|w|^2+\lambda r^{4}(\Ve^2|u|^2 + \fe^2|u|^2 - 2\Ve\fe |u|^2)\bigr)\mu_{r}^{\alpha+1}\,dx \\
            &\geq -\frac{H(r)/\lambda + \lambda r^4 G^2r^4H(r) + \lambda r^4G^2\varepsilon^{2\beta}r^4H(r) + \lambda r^4G^2\varepsilon^{\beta}r^4H(r)}{4\alpha rH(r)}.
        \end{align}
        Then we have
		\begin{equation}\label{eq:Hdr-D-part-lower}
			\frac{D'(r)H(r)-D(r)H'(r)}{H^2(r)}
			\ge
			-\frac{1}{4\alpha\lambda r}
			-
			\frac{\lambda G^{2}r^{7}}{4\alpha}
            -
            \frac{\lambda G^2 \varepsilon^{2\beta}r^7}{4\alpha}
            -
            \frac{\lambda G^2 \varepsilon^{\beta}r^7}{4\alpha}
			=
			-d(r).
		\end{equation}
		
		Next we study the second quotient in (\ref{eq:Hdr-N-quotient}).		
		Recall  Lemma~\ref{lem:Hdr-LH-identity}.
		Dividing both sides of (\ref{eq:Hdr-LH-main}) by \(H^2(r)\), and then using Lemma~\ref{lem:Hdr-E-bounds} and definition of $a(r), b(r)$, we obtain
		\begin{align}
			\frac{L'(r)H(r)-L(r)H'(r)}{H^2(r)}
			&\ge
			-a(r)
			-
			b(r)\frac{D(r)}{H(r)}
			-
			\frac{N(r)}{\alpha r}\frac{L(r)}{H(r)}
			+
			\frac{2}{r}\frac{L(r)}{H(r)}.
			\label{eq:Hdr-Lpart-first}
		\end{align}
		Since 
		\begin{equation}\label{eq:Hdr-D-over-H}
			\frac{D(r)}{H(r)}
			=
			\frac{D(r)+L(r)}{H(r)}-\frac{L(r)}{H(r)}
			=
			N(r)-\frac{L(r)}{H(r)}.
		\end{equation}
		Substituting \eqref{eq:Hdr-D-over-H} into \eqref{eq:Hdr-Lpart-first}, we get
		\begin{align}
			\frac{L'(r)H(r)-L(r)H'(r)}{H^2(r)}
			&\ge
			-a(r)-b(r)N(r)
			+
			\left(
			b(r)+\frac{2}{r}-\frac{N(r)}{\alpha r}
			\right)\frac{L(r)}{H(r)}.
			\label{eq:Hdr-Lpart-second}
		\end{align}
		
		Next we give a 
		  lower bound for \(N(r)\).
		Since \(D(r)\ge 0\), it is obvious that
		\[
		N(r)=\frac{D(r)+L(r)}{H(r)}\ge \frac{L(r)}{H(r)}.
		\]
		Thanks to \eqref{eq:Hdr-L-over-H}, it follows that
		\begin{equation}\label{eq:Hdr-N-lower}
			N(r)\ge -\,\frac{|L(r)|}{H(r)}\ge -\,B(r).
		\end{equation}
		Therefore
		\begin{equation}\label{eq:Hdr-absN-by-N}
			|N(r)|\le N(r)+2B(r).
		\end{equation}
		Indeed, if \(N(r)\ge 0\), then the inequality is trivial. If \(N(r)<0\), then
		\(N(r)\ge -B(r)\). Hence
		\[
		|N(r)|=-N(r)\le N(r)+2B(r).
		\]
		
		Furthermore, we  
	 control the \(\frac{L(r)}{H(r)}\)-term.
		By \eqref{eq:Hdr-L-over-H} again, we have
		\begin{equation}\label{eq:Hdr-L-over-H-by-B}
			\left|\frac{L(r)}{H(r)}\right|\le B(r).
		\end{equation}
		Therefore, it follows from \eqref{eq:Hdr-Lpart-second} that,
		\begin{align}
			\frac{L'(r)H(r)-L(r)H'(r)}{H^2(r)}
			&\ge
			-a(r)-b(r)N(r)
			-
			\left(
			b(r)+\frac{2}{r}+\frac{|N(r)|}{\alpha r}
			\right)B(r).
			\label{eq:Hdr-Lpart-third}
		\end{align}
		Using \eqref{eq:Hdr-absN-by-N}, we obtain
		\[
		\frac{|N(r)|}{\alpha r}B(r)
		\le
		\frac{B(r)}{\alpha r}N(r)+\frac{2B^{2}(r)}{\alpha r}.
		\]
		Substituting this into \eqref{eq:Hdr-Lpart-third} yields that
		\begin{align}
			\frac{L'(r)H(r)-L(r)H'(r)}{H^2(r)}
			&\ge
			-a(r)-b(r)N(r)
			-\left(b(r)+\frac{2}{r}\right)B(r)
			-\frac{B(r)}{\alpha r}N(r)
			-\frac{2B^{2}(r)}{\alpha r}
			\notag\\
			&=
			-\left[a(r)+\left(b(r)+\frac{2}{r}\right)B(r)+\frac{2B^{2}(r)}{\alpha r}\right]
			-\left[b(r)+\frac{B(r)}{\alpha r}\right]N(r).
			\label{eq:Hdr-Lpart-final}
		\end{align}
		
	We are ready to estimate $N'(r)$.
		Combining \eqref{eq:Hdr-N-quotient}, \eqref{eq:Hdr-D-part-lower}, and \eqref{eq:Hdr-Lpart-final}, we obtain
		\[
		N'(r)
		\ge
		-d(r)
		-\left[a(r)+\left(b(r)+\frac{2}{r}\right)B(r)+\frac{2B^{2}(r)}{\alpha r}\right]
		-\left[b(r)+\frac{B(r)}{\alpha r}\right]N(r).
		\]
		That is,
		\[
		N'(r)+A(r)N(r)\ge -P(r),
		\]
		where \(A(r)\) and \(P(r)\) are given by \eqref{eq:Hdr-A-def-corrected} and
		\eqref{eq:Hdr-P-def-corrected}.
	\end{proof}
	
	The monotone property of \(N(r)\) provides the monotonicity of the frequency function. Define
	\begin{equation}\label{eq:Hdr-calN}
		\mathcal N(r):=
		e^{\int_{r_{0}}^{r}A(t)\,dt}N(r)
		+
\int_{r_{0}}^{r}e^{\int_{r_{0}}^{s}A(t)\,dt}P(s)\,ds.
	\end{equation}
	\begin{corollary}\label{cor:Hdr-monotone-quantity}
	Fix any \(r_{0}\in(0, \ R-\varepsilon)\), then \(\mathcal N(r)\) is nondecreasing in \((r_{0},R)\).
\end{corollary}

\begin{proof}
	Differentiating \eqref{eq:Hdr-calN}, we obtain
	\begin{align*}
		\mathcal N'(r)
		&=
		e^{\int_{r_{0}}^{r}A(t)\,dt}\bigl(N'(r)+A(r)N(r)\bigr)
		+
		e^{\int_{r_{0}}^{r}A(t)\,dt}P(r)
		\\
		&=
		e^{\int_{r_{0}}^{r}A(t)\,dt}\bigl(N'(r)+A(r)N(r)+P(r)\bigr).
	\end{align*}
	Thanks to Proposition~\ref{prop:Hdr-N-differential-corrected},
	\[
	N'(r)+A(r)N(r)+P(r)\ge 0.
	\]
	Hence \(\mathcal N'(r)\ge 0\), which verifies the conclusion of the lemma.
\end{proof}

In the next lemma, we show two simple comparison inequalities for \(N(r)\).
\begin{lemma}\label{lem:Hdr-N-compare}
	For \(0<a<b<R\), define
	\begin{equation}\label{eq:Hdr-Phi}
\Phi(a,b):=\int_{a}^{b}A(t)\,dt.
	\end{equation}
	Then
	\begin{equation}\label{eq:Hdr-N-upper}
		N(a)\le e^{\Phi(a,b)}N(b)+\int_{a}^{b}e^{\Phi(a,s)}P(s)\,ds,
	\end{equation}
	and
	\begin{equation}\label{eq:Hdr-N-lower-compare}
		N(b)\ge e^{-\Phi(a,b)}N(a)-\int_{a}^{b}e^{-\Phi(s,b)}P(s)\,ds.
	\end{equation}
\end{lemma}

\begin{proof}
	Since \(\mathcal N(r)\) is nondecreasing,
	that is, 	
$\mathcal N(a)\le \mathcal N(b).$
	Then  \eqref{eq:Hdr-calN} shows that
	\[
	e^{\int_{r_{0}}^{a}A(t)\,dt}N(a)
	+
	\int_{r_{0}}^{a}e^{\int_{r_{0}}^{s}A(t)\,dt}P(s)\,ds
	\le
	e^{\int_{r_{0}}^{b}A(t)\,dt}N(b)
	+
	\int_{r_{0}}^{b}e^{\int_{r_{0}}^{s}A(t)\,dt}P(s)\,ds.
	\]
	Subtracting the second  integral of the left hand side of the last inequality from both sides, we obtain
	\begin{align}
	e^{\int_{r_{0}}^{a}A(t)\,dt}N(a)
	\le
	e^{\int_{r_{0}}^{b}A(t)\,dt}N(b)
	+
	\int_{a}^{b}e^{\int_{r_{0}}^{s}A(t)\,dt}P(s)\,ds.
    \label{compare-1}
	\end{align}
	Multiplying (\ref{compare-1}) by \(e^{-\int_{r_{0}}^{a}A(t)\,dt}\), we get
	\[
	N(a)\le e^{\int_{a}^{b}A(t)\,dt}N(b)+\int_{a}^{b}e^{\int_{a}^{s}A(t)\,dt}P(s)\,ds,
	\]
	where  \eqref{eq:Hdr-N-upper} is arrived.
	
	Multiplying (\ref{compare-1})  by \(e^{-\int_{r_{0}}^{b}A(t)\,dt}\) gives that
	\[
	N(b)\ge e^{-\int_{a}^{b}A(t)\,dt}N(a)-\int_{a}^{b}e^{-\int_{s}^{b}A(t)\,dt}P(s)\,ds.
	\]
	Thus, we obtain \eqref{eq:Hdr-N-lower-compare}.
\end{proof}

Thanks to the monotonicity of $\mathcal{N}(r)$, we are able to show a weighted three-ball inequality for \(H(r)\). We still need to introduce some notations.
	Define
	\begin{equation}\label{eq:Hdr-c1-c2}
		c_{1}:=(2\alpha+n-2)\log\frac{2r_{2}}{r_{1}},
		\qquad
		c_{2}:=(2\alpha+n-2)\log\frac{r_{3}}{2r_{2}},
	\end{equation}
	\begin{equation}\label{eq:Hdr-beta-gamma}
		\beta_1:=\frac{1}{\alpha}\int_{r_{1}}^{2r_{2}}\frac{e^{\Phi(t,2r_{2})}}{t}\,dt,
		\qquad
		\gamma_1:=\frac{1}{\alpha}\int_{2r_{2}}^{r_{3}}\frac{e^{-\Phi(2r_{2},t)}}{t}\,dt,
	\end{equation}
	\begin{equation}\label{eq:Hdr-R1}
		R_{1}:=
		\frac{1}{\alpha}\int_{r_{1}}^{2r_{2}}
		\frac{1}{t}
		\left(
		\int_{t}^{2r_{2}}e^{\Phi(t,s)}P(s)\,ds
		\right)\,dt, \quad 
		R_{2}:=
		\frac{1}{\alpha}\int_{2r_{2}}^{r_{3}}
		\frac{1}{t}
		\left(
		\int_{2r_{2}}^{t}e^{-\Phi(s,t)}P(s)\,ds
		\right)\,dt,
	\end{equation}
    \begin{equation}\label{eq:Hdr-S1}
        S_1 := 
        \int_{r_1}^{2r_2}\frac{C\sqrt{\lambda}t^3G\varepsilon^{\beta}}{2\alpha}\,dt,
   \quad
        S_2 := \int_{2r_2}^{r_3}\frac{C\sqrt{\lambda}t^3G\varepsilon^{\beta}}{2\alpha}\,dt,
    \end{equation}
	\begin{equation}\label{eq:Hdr-LambdaH}
		\kappa:=\kappa(r_1, r_2, r_3):=\frac{\gamma_1}{\beta_1+\gamma_1},
	\quad
		\Lambda_{H}:=
		\kappa c_{1}-(1-\kappa)c_{2}+\kappa R_{1}+(1-\kappa)R_{2}+\kappa S_{1}+(1-\kappa)S_{2}.
	\end{equation}
\begin{proposition}\label{prop:Hdr-H-three-ball-from-calN}
	Let
	$0<r_{1}<r_{2}<2r_{2}<r_{3}<R-\varepsilon.$
	Then
	\begin{equation}\label{eq:Hdr-H-three-ball}
		H(2r_{2})\le e^{\Lambda_{H}}H^{\kappa}(r_{1})H^{1-\kappa}(r_{3}).
	\end{equation}
\end{proposition}

\begin{proof}
	From  (\ref{eq:Hdr-H-prime}), the following identity holds
	\[
	\frac{H'(r)}{H(r)}
	=
	\frac{2\alpha+n-2}{r}
	+
	\frac{N(r)}{\alpha r} + \frac{\int_{B_r}\lambda r^4 w f_{\varepsilon}u\mu_r^{\alpha}}{H\alpha r}.
	\]
    From (\ref{repeat-1}), we get
    \begin{equation}
        \frac{\Bigl|\int_{B_r}\lambda r^4 w f_{\varepsilon}u\mu_r^{\alpha}\Bigr|}{H\alpha r} \leq  \frac{C\lambda r^4 G\varepsilon^{\beta}H}{2\sqrt{\lambda}H\alpha r } = \frac{C\sqrt{\lambda}r^3G\varepsilon^{\beta}}{2\alpha}.
    \end{equation}
   Thus, it holds that
    \begin{equation}\label{HHH-1}
    \frac{2\alpha+n-2}{r}
	+
	\frac{N(r)}{\alpha r} - \frac{C\sqrt{\lambda}r^3 G\varepsilon^{\beta}}{2\alpha} \leq
    \frac{H'(r)}{H(r)}
	\leq
	\frac{2\alpha+n-2}{r}
	+
	\frac{N(r)}{\alpha r} + \frac{C\sqrt{\lambda}r^3G\varepsilon^{\beta}}{2\alpha}.
    \end{equation}
	Integrating the right hand side of the last inequality from \(r_{1}\) to \(2r_{2}\), we obtain
	\begin{equation}\label{eq:Hdr-left-raw}
		\log\frac{H(2r_{2})}{H(r_{1})}
		\leq
		c_{1}
		+
		\frac{1}{\alpha}\int_{r_{1}}^{2r_{2}}\frac{N(t)}{t}\,dt + S_1.
	\end{equation}
	By \eqref{eq:Hdr-N-upper}, for \(t\in[r_{1},2r_{2}]\),
	\[
	N(t)\le e^{\Phi(t,2r_{2})}N(2r_{2})+\int_{t}^{2r_{2}}e^{\Phi(t,s)}P(s)\,ds.
	\]
	Substituting the last inequality into \eqref{eq:Hdr-left-raw} gives that
	\begin{align}
		\log\frac{H(2r_{2})}{H(r_{1})}
		&\le
		c_{1}
		+
		\frac{1}{\alpha}\int_{r_{1}}^{2r_{2}}\frac{e^{\Phi(t,2r_{2})}}{t}\,dt\,N(2r_{2})
		\notag\\
		&\quad
		+
		\frac{1}{\alpha}\int_{r_{1}}^{2r_{2}}\frac{1}{t}
		\left(
		\int_{t}^{2r_{2}}e^{\Phi(t,s)}P(s)\,ds
		\right)\,dt +S_1
		\notag\\
		&=
		c_{1}+\beta_1 N(2r_{2})+R_{1} + S_1.
		\label{eq:Hdr-left}
	\end{align}
	
	Similarly, integrating the left side of (\ref{HHH-1}) from \(2r_{2}\) to \(r_{3}\) yields
	\begin{equation}\label{eq:C2-right-raw}
		\log\frac{H(r_{3})}{H(2r_{2})}
		\geq
		c_{2}
		+
		\frac{1}{\alpha}\int_{2r_{2}}^{r_{3}}\frac{N(t)}{t}\,dt - S_2.
	\end{equation}
	By \eqref{eq:Hdr-N-lower-compare}, for \(t\in[2r_{2},r_{3}]\), we have
	\[
	N(t)\ge e^{-\Phi(2r_{2},t)}N(2r_{2})-\int_{2r_{2}}^{t}e^{-\Phi(s,t)}P(s)\,ds.
	\]
	Substituting this into \eqref{eq:C2-right-raw} shows that
	\begin{align}
		\log\frac{H(r_{3})}{H(2r_{2})}
		&\ge
		c_{2}
		+
		\frac{1}{\alpha}\int_{2r_{2}}^{r_{3}}\frac{e^{-\Phi(2r_{2},t)}}{t}\,dt\,N(2r_{2})
		\notag\\
		&\quad
		-
		\frac{1}{\alpha}\int_{2r_{2}}^{r_{3}}\frac{1}{t}
		\left(
		\int_{2r_{2}}^{t}e^{-\Phi(s,t)}P(s)\,ds
		\right)\,dt - S_2
		\notag\\
		&=
		c_{2}+\gamma_1 N(2r_{2})-R_{2} - S_2.
		\label{eq:C2-right}
	\end{align}
	From \eqref{eq:C2-right}, we deduce
	\[
	N(2r_{2})
	\le
	\frac{1}{\gamma_1}
	\left(
	\log\frac{H(r_{3})}{H(2r_{2})}-c_{2}+R_{2} + S_2
	\right).
	\]
	Substituting the last inequality into \eqref{eq:Hdr-left}, we get
	\begin{align*}
		\log\frac{H(2r_{2})}{H(r_{1})}
		&\le
		c_{1}+R_{1} + S_1
		+
		\frac{\beta_1}{\gamma_1}
		\left(
		\log\frac{H(r_{3})}{H(2r_{2})}-c_{2}+R_{2} + S_2
		\right).
	\end{align*}
	Multiplying the last inequality by \(\gamma_1\) gives that
	\[
	(\beta_1+\gamma_1)\log H(2r_{2})
	\le
	\gamma_1\log H(r_{1})
	+
	\beta_1\log H(r_{3})
	+
	\gamma_1 c_{1}
	-
	\beta_1 c_{2}
	+
	\gamma_1 R_{1}
	+
	\beta_1 R_{2}
    +
    \gamma_1 S_1
    +
    \beta_1 S_2.
	\]
	Dividing the last estimate by \(\beta_1+\gamma_1\), and using \(\kappa=\dfrac{\gamma_1}{\beta_1+\gamma_1}\), we arrive at
	\[
	\log H(2r_{2})
	\le
	\kappa\log H(r_{1})
	+
	(1-\kappa)\log H(r_{3})
	+
	\Lambda_{H}.
	\]
	Taking exponentiation to both sides of  the last inequality yields \eqref{eq:Hdr-H-three-ball}.
\end{proof}

	Now we set up to choose $\lambda$ and $\alpha$ to optimize the three-ball inequality in Proposition \ref{prop:Hdr-H-three-ball-from-calN}.
	
	\begin{lemma} \label{thm:Hdr-optimized-weighted-three-ball-H-sharp}
		Assume  that
		$G\ge 16$ and $0<r_{1}<r_{2}<2r_{2}<r_{3}<R-\varepsilon.$
		For the corresponding quantities \(\beta_1,\gamma_1,\kappa\) in Proposition \ref{prop:Hdr-H-three-ball-from-calN},
		there exists a constant \(C=C(n,R)>0\) such that
\begin{equation}\label{eq:optimized-three-ball-H-sharp}
			H(2r_{2})
			\le
			\exp\!\left(
			C G^{1/4}\log\frac{2r_{2}}{r_{1}}
			+
			C G^{1/4}\log\frac{r_{3}}{2r_{2}}
			\right)
			H^{\kappa}(r_{1})H^{1-\kappa}(r_{3}).
		\end{equation}
	\end{lemma}
	
	\begin{proof}
	We aim to find the optimal bound for $A_H$.
    We will choose $\lambda>0$, $\varepsilon>0$ to be small and $\alpha>0$ to be large.
		Since \(0<r<R-\varepsilon\),  from \eqref{eq:Hdr-P-expanded}, there exists \(C=C(n,R)>0\) such that for every \(0<r<R-\varepsilon\),
		\begin{equation}\label{eq:Hdr-P-rough-final-sharp}
			P(r)
			\le
			C\left[
			\left(
			\frac{1}{\alpha\lambda}
			+
			\frac{1}{\sqrt{\lambda}}
			\right)\frac{1}{r}
			+
			\left(
			\frac{\lambda G^{2}}{\alpha}
			+
			\sqrt{\lambda}G\varepsilon^{\beta - 1}
            +
            \frac{\sqrt{\lambda}G\varepsilon^{\beta - 2}}{\alpha}
            +
			\frac{G}{\alpha}
			\right)r^{3}
			\right].
		\end{equation}
		
		Since \(0<\kappa<1\) and \(c_{2}\ge 0\),
\begin{equation}\label{eq:c1c2-estimate-sharp}
			\kappa c_{1}-(1-\kappa)c_{2}=c_1
			\le
			(2\alpha+n-2)\log\frac{2r_{2}}{r_{1}}
			\le
C(n)\,\alpha\log\frac{2r_{2}}{r_{1}}.
		\end{equation}
		
	Recall $R_1$ is given in \eqref{eq:Hdr-R1}.
		Since
		\[
		\Phi(t,s)=\Phi(t,2r_{2})-\Phi(s,2r_{2}),
		\]
		we can rewrite \(R_{1}\) as
		\begin{equation}\label{eq:R1-rewrite-sharp}
			R_{1}
			=
			\frac1\alpha
			\int_{r_{1}}^{2r_{2}}
			\frac{e^{\Phi(t,2r_{2})}}{t}
			\left(
			\int_{t}^{2r_{2}}e^{-\Phi(s,2r_{2})}P(s)\,ds
			\right)dt.
		\end{equation}
		It follows from \eqref{eq:R1-rewrite-sharp} that
		\[
		R_{1}
		\le
		\frac1\alpha
		\left(
		\int_{r_{1}}^{2r_{2}}\frac{e^{\Phi(t,2r_{2})}}{t}\,dt
		\right)
		\left(
		\int_{r_{1}}^{2r_{2}}e^{-\Phi(s,2r_{2})}P(s)\,ds
		\right)
		=
		\beta_1 I_{1},
		\]
		where
		\begin{equation}\label{eq:def-I1-sharp}
			I_{1}:=\int_{r_{1}}^{2r_{2}}e^{-\Phi(s,2r_{2})}P(s)\,ds.
		\end{equation}
		Recall that $\kappa=\frac{\gamma_1}{\beta_1+\gamma_1}$ and $\beta_1$ is given in (\ref{eq:Hdr-beta-gamma}). We have
		\begin{equation}\label{eq:kR1-beta-sharp}
			\kappa R_{1}\le \kappa\beta_1 I_{1}\le \gamma_1 I_{1}.
		\end{equation}
		Since \(e^{-\Phi(2r_{2},t)}\le 1\), we can control $\gamma_1$ as follows,
	\begin{equation}\label{eq:gamma-upper-sharp}
			\gamma_1
			=
			\frac1\alpha\int_{2r_{2}}^{r_{3}}\frac{e^{-\Phi(2r_{2},t)}}{t}\,dt
			\le
			\frac1\alpha\log\frac{r_{3}}{2r_{2}}.
		\end{equation}
		
		Next we estimate \(I_{1}\). By \eqref{eq:Hdr-A-expanded}, we can estimate $A(r)$ below as
		\[
		A(r)\ge \frac{1}{2\alpha r\sqrt{\lambda}}+\frac{\sqrt{\lambda}Gr^{3}}{8\alpha}.
		\]
		Therefore, for \(r_{1}\le s\le 2r_{2}\), we have
		\begin{align}
			\Phi(s,2r_{2})
			&=
			\int_{s}^{2r_{2}}A(r)\,dr
			\notag\\
			&\ge
			\frac{1}{2\alpha\sqrt{\lambda}}\log\frac{2r_{2}}{s}
			+
			\frac{\sqrt{\lambda}G}{32\alpha}\bigl((2r_{2})^{4}-s^{4}\bigr).
			\label{eq:Phi-lower-R1-sharp}
		\end{align}
		Set
		\begin{equation}\label{eq:def-sigma-tau-sharp}
			\sigma:=\frac{1}{2\alpha\sqrt{\lambda}},
			\qquad
			\tau:=\frac{\sqrt{\lambda}G}{32\alpha}.
		\end{equation}
		Then \eqref{eq:Phi-lower-R1-sharp} yields
		\begin{equation}\label{eq:e-minus-Phi-R1-sharp}
			e^{-\Phi(s,2r_{2})}
			\le
			\left(\frac{s}{2r_{2}}\right)^{\sigma}
			e^{-\tau\bigl((2r_{2})^{4}-s^{4}\bigr)}.
		\end{equation}
		Substituting \eqref{eq:Hdr-P-rough-final-sharp} and \eqref{eq:e-minus-Phi-R1-sharp} into \eqref{eq:def-I1-sharp} to have
		\begin{align}
			I_{1}
			\le{}&
			C
			\left(
			\frac{1}{\alpha\lambda}
			+
			\frac{1}{\sqrt{\lambda}}
			\right)
			\int_{r_{1}}^{2r_{2}}
			\left(\frac{s}{2r_{2}}\right)^{\sigma}\frac{ds}{s}
			\notag\\
			&
			+
			C
			\left(
			\frac{\lambda G^{2}}{\alpha}
			+
			\sqrt{\lambda}G\varepsilon^{\beta - 1}
            +
            \frac{\sqrt{\lambda}G\varepsilon^{\beta - 2}}{\alpha}
            +
			\frac{G}{\alpha}
			\right)
			\int_{r_{1}}^{2r_{2}}
			\left(\frac{s}{2r_{2}}\right)^{\sigma}
			e^{-\tau\bigl((2r_{2})^{4}-s^{4}\bigr)} s^{3}\,ds.
			\label{eq:I1-split-sharp}
		\end{align}
		The first integral is
		\begin{align}
			\int_{r_{1}}^{2r_{2}}\left(\frac{s}{2r_{2}}\right)^{\sigma}\frac{ds}{s}
			&=
			(2r_{2})^{-\sigma}\int_{r_{1}}^{2r_{2}}s^{\sigma-1}\,ds
			\le
			\frac{C}{\sigma}.
			\label{eq:I1-singular-sharp}
		\end{align}
		For the second integral, since \(\left(\frac{s}{2r_{2}}\right)^{\sigma}\le 1\), we have
		\begin{align}
			\int_{r_{1}}^{2r_{2}}
			\left(\frac{s}{2r_{2}}\right)^{\sigma}
			e^{-\tau\bigl((2r_{2})^{4}-s^{4}\bigr)} s^{3}\,ds
			&\le
			\int_{0}^{2r_{2}}e^{-\tau\bigl((2r_{2})^{4}-s^{4}\bigr)} s^{3}\,ds
			\notag\\
			&=
\frac14\int_{0}^{(2r_{2})^{4}}e^{-\tau y}\,dy
			\le
			\frac{C}{4\tau}.
			\label{eq:I1-regular-sharp}
		\end{align}
		Combining \eqref{eq:I1-split-sharp}, \eqref{eq:I1-singular-sharp}, and \eqref{eq:I1-regular-sharp} gives that
		\begin{equation}\label{eq:I1-before-simplify-sharp}
			I_{1}
			\le
			C\left[
			\left(
			\frac{1}{\alpha\lambda}
			+
			\frac{1}{\sqrt{\lambda}}
			\right)\frac{1}{\sigma}
			+
			\left(
			\frac{\lambda G^{2}}{\alpha}
			+
			\sqrt{\lambda}G\varepsilon^{\beta - 1}
            +
            \frac{\sqrt{\lambda}G\varepsilon^{\beta - 2}}{\alpha}
            +
			\frac{G}{\alpha}
			\right)\frac{1}{\tau}
			\right].
		\end{equation}
		
		Inserting \eqref{eq:def-sigma-tau-sharp} into \eqref{eq:I1-before-simplify-sharp}, we obtain
		\begin{align}
			I_{1}
			&\le
			C\left[
			\left(
			\frac{1}{\alpha\lambda}
			+
			\frac{1}{\sqrt{\lambda}}
			\right)\alpha\sqrt{\lambda}
			+
			\left(
			\frac{\lambda G^{2}}{\alpha}
			+
			\sqrt{\lambda}G\varepsilon^{\beta - 1}
            +
            \frac{\sqrt{\lambda}G\varepsilon^{\beta - 2}}{\alpha}
            +
			\frac{G}{\alpha}
			\right)\frac{\alpha}{\sqrt{\lambda}G}
			\right]
			\notag\\
			&\le
			C\left(
			\alpha+\frac{1}{\sqrt{\lambda}}+\sqrt{\lambda}G + \alpha \varepsilon^{\beta - 1} + \varepsilon^{\beta - 2}
			\right).
			\label{eq:I1-final-sharp}
		\end{align}
		Together with \eqref{eq:kR1-beta-sharp}, \eqref{eq:gamma-upper-sharp}, and \eqref{eq:I1-final-sharp}, we conclude that
	\begin{equation}\label{eq:kR1-final-sharp}
			\kappa R_{1}
			\le
			C\frac{1}{\alpha}\log\frac{r_{3}}{2r_{2}}
			\left(
			\alpha+\frac{1}{\sqrt{\lambda}}+\sqrt{\lambda}G + \alpha \varepsilon^{\beta - 1} + \varepsilon^{\beta - 2}
			\right).
		\end{equation}

	Now we estimate \((1-\kappa)R_{2}\). Note that $R_2$ is given in (\ref{eq:Hdr-R1}).
		We can write
		\begin{align}
			R_{2}
			&=
			\frac1\alpha
			\int_{2r_{2}}^{r_{3}}
			P(s)
			\left(
			\int_{s}^{r_{3}}\frac{e^{-\Phi(s,t)}}{t}\,dt
			\right)ds
			\notag\\
			&=
			\frac1\alpha\int_{2r_{2}}^{r_{3}}P(s)\,K(s)\,ds,
			\label{eq:R2-order-sharp}
		\end{align}
		where we change the order of integration and
	\begin{equation}\label{eq:def-K-sharp}
			K(s):=\int_{s}^{r_{3}}\frac{e^{-\Phi(s,t)}}{t}\,dt.
		\end{equation}
		
		We provide  \(K(s)\) with two lower bounds.
		On one hand,  from the fact that \(A(r)\ge \frac{1}{2\alpha r\sqrt{\lambda}}\), the definition of $\Phi(s,t)$ in (\ref{eq:Hdr-Phi}) shows that
		\[
		\Phi(s,t)\ge \sigma\log\frac{t}{s},
		\]
		with \(\sigma\) defined in \eqref{eq:def-sigma-tau-sharp}. Hence
		\begin{align}
			K(s)
			&\le
			\int_{s}^{r_{3}}\frac{1}{t}\left(\frac{s}{t}\right)^{\sigma}\,dt
			\notag\\
			&=
			s^{\sigma}\int_{s}^{r_{3}}t^{-1-\sigma}\,dt
			\le
			\frac{1}{\sigma}.
			\label{eq:K-log-sharp}
		\end{align}
		
		On the other hand, from the estimates that \(A(r)\ge \frac{\sqrt{\lambda}Gr^{3}}{8\alpha}\), the definition of $\Phi(s,t)$ in (\ref{eq:Hdr-Phi}) indicates that
		\[
		\Phi(s,t)\ge \tau(t^{4}-s^{4}),
		\]
		with \(\tau\) defined in \eqref{eq:def-sigma-tau-sharp}. Therefore
		\begin{align}
			K(s)
			&\le
			\int_{s}^{r_{3}}\frac{1}{t}e^{-\tau(t^{4}-s^{4})}\,dt
			\notag\\
			&\le
			\frac{1}{4s^{4}}
			\int_{0}^{r_{3}^{4}-s^{4}}e^{-\tau y}\,dy
			\le
			\frac{1}{4\tau s^{4}}.
			\label{eq:K-poly-sharp}
		\end{align}
		
	Note that the upper bound of $P$ in (\ref{eq:Hdr-P-rough-final-sharp}).	We now use \eqref{eq:K-log-sharp} for the $\frac 1 r$-terms, and \eqref{eq:K-poly-sharp} for the $r^3$-terms in (\ref{eq:Hdr-P-rough-final-sharp}).
		Combining \eqref{eq:R2-order-sharp}, \eqref{eq:Hdr-P-rough-final-sharp}, \eqref{eq:K-log-sharp}, and \eqref{eq:K-poly-sharp}, we can show that
		\begin{align}
			R_{2}
			\le{}&
			C\frac1\alpha
			\left(
			\frac{1}{\alpha\lambda}
			+
			\frac{1}{\sqrt{\lambda}}
			\right)
			\frac{1}{\sigma}
			\int_{2r_{2}}^{r_{3}}\frac{ds}{s}
			\notag\\
			&\quad
			+
			C\frac1\alpha
			\left(
			\frac{\lambda G^{2}}{\alpha}
			+
			\sqrt{\lambda}G\varepsilon^{\beta - 1}
            +
            \frac{\sqrt{\lambda}G\varepsilon^{\beta - 2}}{\alpha}
            +
			\frac{G}{\alpha}
			\right)
			\frac{1}{\tau}
			\int_{2r_{2}}^{r_{3}}\frac{s^{3}}{s^{4}}\,ds
			\notag\\
			&=
			C\frac1\alpha
			\left[
			\left(
			\frac{1}{\alpha\lambda}
			+
			\frac{1}{\sqrt{\lambda}}
			\right)\frac{1}{\sigma}
			+
			\left(
			\frac{\lambda G^{2}}{\alpha}
			+
			\sqrt{\lambda}G\varepsilon^{\beta - 1}
            +
            \frac{\sqrt{\lambda}G\varepsilon^{\beta - 2}}{\alpha}
            +
			\frac{G}{\alpha}
			\right)\frac{1}{\tau}
			\right]
			\log\frac{r_{3}}{2r_{2}}.
			\label{eq:R2-before-simplify-sharp}
		\end{align}
		Recall $\sigma$ and $\tau$ in \eqref{eq:def-sigma-tau-sharp}, we conclude that
		\begin{equation}\label{eq:R2-final-sharp}
			R_{2}
			\le
			C\frac{1}{\alpha}
			\log\frac{r_{3}}{2r_{2}}
			\left(
			\alpha+\frac{1}{\sqrt{\lambda}}+\sqrt{\lambda}G + \alpha \varepsilon^{\beta - 1} + \varepsilon^{\beta - 2}
			\right).
		\end{equation}
		Since \(0<1-\kappa<1\), it holds that
		\begin{equation}\label{eq:one-minus-kappa-R2-sharp}
			(1-\kappa)R_{2}
			\le
			C\frac{1}{\alpha}
			\log\frac{r_{3}}{2r_{2}}
			\left(
			\alpha+\frac{1}{\sqrt{\lambda}}+\sqrt{\lambda}G + \alpha \varepsilon^{\beta - 1} + \varepsilon^{\beta - 2}
			\right).
		\end{equation}
        
		Next we  estimate \(\kappa S_{1}\) and \((1 - \kappa) S_{2}\). Recall $S_1, S_2$ are given in (\ref{eq:Hdr-S1}).
        By direct calculations and the fact that $\kappa < 1$, we can show that
        \begin{align}
            \kappa S_1 &\leq 
            \int_{r_1}^{2r_2}\frac{C\sqrt{\lambda}t^3G\varepsilon^{\beta}}{2\alpha}\,dt = \frac{C\sqrt{\lambda}G\varepsilon^{\beta}}{8\alpha}\Bigl( (2r_2)^4 - r_1^4\Bigr), \label{eq:Hdr-kS1-final-sharp} \\
            (1 - \kappa)S_2 &\leq \int_{2r_2}^{r_3}\frac{C\sqrt{\lambda}t^3G\varepsilon^{\beta}}{2\alpha}\,dt = \frac{C\sqrt{\lambda}G\varepsilon^{\beta}}{8\alpha}\Bigl(r_3^4 - (2r_2)^4\Bigr).\label{eq:Hdr-kS2-final-sharp}
        \end{align}
        
	We sum up the terms to estimate \(\Lambda_{H}\). Recall that \(\Lambda_{H}\) is given in (\ref{eq:Hdr-LambdaH}).
		From \eqref{eq:c1c2-estimate-sharp}, \eqref{eq:kR1-final-sharp}, \eqref{eq:Hdr-kS1-final-sharp}, \eqref{eq:Hdr-kS2-final-sharp} and \eqref{eq:one-minus-kappa-R2-sharp}, we obtain
		\begin{align}
			\Lambda_{H}
			&=
			\kappa c_{1}-(1-\kappa)c_{2}+\kappa R_{1}+(1-\kappa)R_{2}+ \kappa S_{1}+(1-\kappa)S_{2}
			\notag\\
			&\le
			C\alpha\log\frac{2r_{2}}{r_{1}}
			+
			C\frac{1}{\alpha}\log\frac{r_{3}}{2r_{2}}
			\left(
			\alpha+\frac{1}{\sqrt{\lambda}}+\sqrt{\lambda}G + \alpha \varepsilon^{\beta - 1} + \varepsilon^{\beta - 2}
			\right)
        \notag \\
        &+C\frac{\sqrt{\lambda}G\varepsilon^{\beta}}{8\alpha}\Bigl(r_3^4 - r_1^4\Bigr).
			\label{eq:LambdaH-before-choice-sharp}
		\end{align}

        Our goal is to minimize $\Lambda_{H}$ with respect to three parameters $\lambda, \varepsilon, \alpha$. We fix the radius $r_1, r_2, r_3$.
		We first choose \(\lambda\) and \(\alpha\) so as to minimize the second factor in
		\eqref{eq:LambdaH-before-choice-sharp}, namely
		\[
		F(\varepsilon,\alpha,\lambda,G)
		:=
		1+\frac{1}{\alpha\sqrt{\lambda}}+\frac{\sqrt{\lambda}G}{\alpha} +  \varepsilon^{\beta - 1} + \frac{\varepsilon^{\beta - 2}}{\alpha}.
		\]
		For fixed $G$ and \(\lambda\), $ \varepsilon^{\beta - 1}$ + $\frac{\varepsilon^{\beta - 2}}{\alpha}$ are optimized when we take
	$\varepsilon = \alpha^{-1}.$
		With this choice,
		\[
		F(\varepsilon,\alpha,\lambda,G)
		:=
		1+\frac{1}{\alpha\sqrt{\lambda}}+\frac{\sqrt{\lambda}G}{\alpha} +  \alpha^{1-\beta}.
		\]
		Note that $\frac{\sqrt{\lambda}G\varepsilon^{\beta}}{\alpha}\leq \frac{\sqrt{\lambda}G}{\alpha}$ since $\varepsilon $ is small.
        To optimize \eqref{eq:LambdaH-before-choice-sharp}, we are led to choose
		\[
		\alpha =
		\frac{1}{\alpha \sqrt{\lambda}} =\frac{\sqrt{\lambda}G}{\alpha}.
		\]
        These are balanced by the choice
        \begin{align}
		\alpha=\lfloor G^{1/4} \rfloor+2, \quad \lambda=G^{-1}.
        \label{balance-1}
		\end{align}
		Thus, by selecting $\varepsilon=G^{-1/4}$ for some large $G$, $\alpha$ and $\lambda$
        as in (\ref{balance-1})
        we obtain from  \eqref{eq:LambdaH-before-choice-sharp} that
		\begin{align}
			\Lambda_{H}
			&\le
			C G^{1/4}\log\frac{2r_{2}}{r_{1}}
			+
			C G^{1/4}\log\frac{r_{3}}{2r_{2}}.
			\label{eq:LambdaH-final-sharp}
		\end{align}

		Finally, substituting the upper bound \eqref{eq:LambdaH-final-sharp} into \eqref{eq:Hdr-H-three-ball}, we have
		\[
		H(2r_{2})
		\le
		\exp\!\left(
		C G^{1/4}\log\frac{2r_{2}}{r_{1}}
		+
		C G^{1/4}\log\frac{r_{3}}{2r_{2}}
		\right)
		H^{\kappa}(r_{1})H^{1-\kappa}(r_{3}).
		\]
		Therefore, we arrive at \eqref{eq:optimized-three-ball-H-sharp}.
	\end{proof}
	
	We show the optimized unweighted three-ball inequality for $h(r)$.
	\begin{lemma}\label{thm:optimized-unweighted-three-ball-h}
		Let
		$0<r_{1}<r_{2}<2r_{2}<r_{3}<R-\varepsilon$
        and \(\kappa\) be the quantities from Proposition \ref{prop:Hdr-H-three-ball-from-calN}.
		Then there exists a constant \(C=C(n,R)>0\) such that \begin{equation}\label{eq:optimized-unweighted-three-ball-h}
			h(r_{2})
			\le
			\exp\!\left(
			C G^{1/4}\log\frac{2r_{2}}{r_{1}}
			+
			C G^{1/4}\log\frac{r_{3}}{r_{2}}
			\right)
			h^{\kappa}(r_{1})h^{1-\kappa}(r_{3}).
		\end{equation}
	\end{lemma}
	
	\begin{proof} We need to reduce the estimates of $H(r)$ to $h(r)$. The norm $H(r)$ and $h(r)$ are comparable as those in the Section 2. That is, (\ref{eq:prop1-H-upper}) and (\ref{eq:prop1-H-lower}) hold.
	Applying   \eqref{eq:prop1-H-upper} and (\ref{eq:prop1-H-lower})
    for \(r=r_{1}\) and
	\(r=r_{3}\) to \eqref{eq:Hdr-H-three-ball}, we obtain
	\begin{align}\label{eq:prop4-h-proof}
		h(r_{2})
		&\le
		(3r_{2}^{2})^{1-\alpha}
		e^{\Lambda_{H}}H^{\kappa}(r_{1})H^{1-\kappa}(r_{3})
		\notag\\
		&\le
		(3r_{2}^{2})^{1-\alpha}
		e^{\Lambda_{H}}
		\bigl(r_{1}^{2\alpha-2}h(r_{1})\bigr)^{\kappa}
		\bigl(r_{3}^{2\alpha-2}h(r_{3})\bigr)^{1-\kappa}
		\notag\\
		&=
e^{\Lambda_{h}}h^{\kappa}(r_{1})h^{1-\kappa}(r_{3}),
	\end{align}
    where 
    	\begin{equation}\label{eq:def-Lambda-h-proof}
		\Lambda_h
			:=
			\Lambda_H
			+
			(2\alpha-2)\bigl(\kappa\log r_1+(1-\kappa)\log r_3-\log r_2\bigr)
			+
			(1-\alpha)\log 3.	
		\end{equation}
		
		From Lemma ~\ref{thm:Hdr-optimized-weighted-three-ball-H-sharp}, we have shown that \begin{equation}\label{eq:Lambda-H-proof}
			\Lambda_H
			\le
			C G^{1/4}\log\frac{2r_2}{r_1}
			+
			C G^{1/4}\log\frac{r_3}{2r_2}.
		\end{equation}
		
		We now estimate the other terms in \eqref{eq:def-Lambda-h-proof}. Since \(0<\kappa<1\) and
		\(r_1<r_2<r_3\), we have
		\begin{align}
			\kappa\log r_1+(1-\kappa)\log r_3-\log r_2
			&=
			\kappa(\log r_1-\log r_2)+(1-\kappa)(\log r_3-\log r_2)
			\notag\\
			&\le
			(1-\kappa)\log\frac{r_3}{r_2}
			\le
			\log\frac{r_3}{r_2}.
			\label{eq:extra-term-bound}
		\end{align}
		Moreover, since \(\alpha\ge 2\),
		it is obvious that $(1-\alpha)\log 3\le 0$.
		Therefore, by \eqref{eq:def-Lambda-h-proof}, \eqref{eq:extra-term-bound}, we get
		\begin{equation}\label{eq:Lambda-h-before-final}
			\Lambda_h
			\le
			\Lambda_H
			+
			(2\alpha-2)\log\frac{r_3}{r_2}.
		\end{equation}
Since it is chosen that $\alpha=\lfloor G^{1/4} \rfloor+2$,
from \eqref{eq:Lambda-H-proof}, 
	 we derive that
		\begin{align}
			\Lambda_h
			&\le
			C G^{1/4}\log\frac{2r_2}{r_1}
			+
			C G^{1/4}\log\frac{r_3}{2r_2}
			+
			C G^{1/4}\log\frac{r_3}{r_2}
			\notag\\
			&\le
			C G^{1/4}\log\frac{2r_2}{r_1}
			+
			C G^{1/4}\log\frac{r_3}{r_2}.
			\label{eq:Lambda-h-final}
		\end{align}
		
		Finally, substituting \eqref{eq:Lambda-h-final} into \eqref{eq:prop4-h-proof}, the three-ball inequality \eqref{eq:optimized-unweighted-three-ball-h} is shown.
	\end{proof}

We are ready to give the $L^\infty$ three-ball inequality for $u$.
	
	\begin{proposition}
	    \label{thm:Linfty-three-ball-1-cc}
		Let
		$0<r_{1}<2r_1<r_{2}<4r_{2}<r_{3}<2r_{3}<R-\varepsilon.$ Denote
	$\kappa_1=\kappa(r_1, 2r_2, r_3)$
		where $0<\kappa(r_1, r_2, r_3)<1$ is given in (\ref{eq:Hdr-LambdaH}).
		There exists a constant \(C=C(n,R)>0\) such that
		\begin{align}
			\|u\|_{L^\infty(B_{r_{2}})}
			\le{}&	\left(  \frac{r_{1}^{\kappa_1}  r_{3}^{1-\kappa_1} }{r_{2} }\right)^{n/2}
			\exp\!\left(
			CG^{1/4}(\log\frac{4r_2}{r_1}
			+
			\log\frac{r_3}{2r_2})
			\right)
	\|u\|_{L^\infty(B_{2r_1})}^{\kappa_1}
			\|u\|_{L^\infty(B_{2r_3})}^{1-\kappa_1}.
			\label{eq:Linfty-three-ball-v-cc}
		\end{align}
\end{proposition}
	
	\begin{proof}
		It follows from  the local \(L^\infty\)-\(L^2\) estimate in (\ref{linfty3ball}) that
	\begin{equation}\label{eq:middle-by-h-cc}
			\|u\|_{L^\infty(B_{r_{2}})}
			\le
			C(1+Gr_{2}^{4})^{C}r_{2}^{-n/2}\,h(2r_{2})^{1/2}.
		\end{equation}
		
		Since
		$0<r_{1}<r_{2}<4r_{2}<r_{3}<R,$
		from Lemma \ref{thm:optimized-unweighted-three-ball-h}, we have 
		\begin{equation}\label{thm:optimized-unweighted-three-ball-h-1}
			h(2r_{2})
			\le
			\exp\!\left(
			C G^{1/4}\log\frac{4r_{2}}{r_{1}}
			+
			C G^{1/4}\log\frac{r_{3}}{2r_{2}}
			\right)
			h^{\kappa_1}(r_{1})h^{1-\kappa_1}(r_{3}).
		\end{equation}
	Taking square roots in
		\eqref{thm:optimized-unweighted-three-ball-h-1}, and substituting the result into \eqref{eq:middle-by-h-cc}, we get
		\begin{align}
		\|u\|_{L^\infty(B_{r_{2}})}
			\le{}&
			C(1+Gr_{2}^{4})^{C}r_{2}^{-n/2}
			\exp\!\left(
			CG^{1/4}(\log\frac{4r_2}{r_1}
			+
			\log\frac{r_3}{2r_2})
			\right)
			\notag\\
			&\times
			h(r_{1})^{\kappa_1/2}h(r_{3})^{(1-\kappa_1)/2}.
			\label{eq:middle-before-endpoints-bounded-cc}
		\end{align}
		
		Let \(r<\rho<R-\varepsilon\). From the definition of \(h(r)\),
		it holds that
		\begin{equation}\label{eq:h-general-1-cc}
			h(r)^{1/2}
			\le
			\|u\|_{L^{2}(B_{r})}
			+
			r^{2}\|u\|_{W^{2,2}(B_{r})}.
		\end{equation}
		Also, by the local \(W^{2,2}\)-\(L^\infty\) estimate proved from Lemma \ref{cor:local-W22-Linfty-correct},
	we can obtain that
\begin{equation}\label{eq:h-general-2-cc}
			h(r)^{1/2}
			\le
			C(n,R)\,
			\Bigl[
			r^{n/2}
			+r^{2}(1+G\rho^{4})(\rho-r)^{-2}\rho^{n/2}
			\Bigr]
			\|u\|_{L^\infty(B_{\rho})}.
		\end{equation}
	Thus we have \begin{equation}\label{eq:h-r1-bounded-cc}
			h(r_{1})^{1/2}
			\le
			C(n,R) G r_1^{n/2}\,
			\|u\|_{L^\infty(B_{2r_1})}.
		\end{equation} 
        and
\begin{equation}\label{eq:h-r3-bounded-cc}
			h(r_{3})^{1/2}
			\le
			C(n,R)G r_3^{n/2}\,
			\|u\|_{L^\infty(B_{2r_{3}})}.
		\end{equation}
		
We combine \eqref{eq:h-r1-bounded-cc}, \eqref{eq:h-r3-bounded-cc} 
		and \eqref{eq:middle-before-endpoints-bounded-cc} to have
		\begin{align}
	\|u\|_{L^\infty(B_{r_{2}})}
			\le{}&
			\left(  \frac{r_{1}^{\kappa_1}  r_{3}^{1-\kappa_1} }{r_{2} }\right)^{n/2}
			\exp\!\left(
			CG^{1/4}(\log\frac{4r_2}{r_1}
			+
			\log\frac{r_3}{2r_2})
			\right)
\|u\|_{L^\infty(B_{2r_1})}^{\kappa_1}
	\|u\|_{L^\infty(B_{2r_3})}^{1-\kappa_1}.
		\end{align}
		Thus, we arrive at \eqref{eq:Linfty-three-ball-v-cc}.
	\end{proof}


We show the proof of Theorem \ref{thm:holder-linfty-three-ball} from the last Proposition.
      \begin{proof}[Proof of Theorem \ref{thm:holder-linfty-three-ball}]
The estimates (\ref{three-hh}) and $\theta_0$ follow from \eqref{eq:Linfty-three-ball-v-cc} and $\kappa_1$ in the last  proposition by setting $r_1=r/2$, $r_2=3r$ and $r_3=15r$.
        \end{proof}
        \begin{remark}\label{difference-theta}
        Note that the $\theta_0$ value in Theorem \ref{thm:Linfty-three-ball} is universal and does not depend on the norm of $V(x)$. However, the $\theta_0$ value in Theorem \ref{thm:holder-linfty-three-ball} depends on  $G$, i.e. the norm of $V(x)$.
        \end{remark}

\section{A variant frequency function and vanishing order for H\"older continuous potentials}

In this section, we will study the vanishing order of bi-Laplace equations (\ref{eq:intro-bilaplace-model}) for
 H\"older continuous potentials with $\|V\|_{C^{0,\beta}}\leq G$. 
 We still assume $w:=\Delta u$ and $\mu_{r}(x):=r^{2}-|x|^{2}.$
	As in the last section,
we build smooth approximations for $V(x)$. We introduce $V_{\varepsilon}:=\rho_{\varepsilon}*V,$  where $\rho_\varepsilon$ is standard mollifier given in (\ref{stand-m}). Note that $V_\varepsilon$ and $V_\varepsilon-V$ satisfy the estimates in Lemma \ref{lem:mollbound}. Let $f_{\varepsilon}:=V_{\varepsilon}-V$. Then
	\begin{equation}\label{eq:holder-delta-w}
		\Delta w=-Vu=-V_{\varepsilon}u+f_{\varepsilon}u.
	\end{equation}

	Fix the parameters $\alpha\geq 2$ and $\lambda>0$ which are to be determined.
	For \(0<r<R-\varepsilon\), we define
	\begin{equation}\label{eq:holder-H}
		H(r):=\int_{B_{r}}\bigl(|u|^2+\lambda r^{2}|w|^2\bigr)\mu_{r}^{\alpha-1}\,dx,
	\end{equation}
	\begin{equation}\label{eq:holder-D}
		D(r):=\int_{B_{r}}\bigl(|\nabla u|^{2}+\lambda r^{2}|\nabla w|^{2}\bigr)\mu_{r}^{\alpha}\,dx
		+
		2\lambda\alpha r^{2}\int_{B_{r}}|w|^2\mu_{r}^{\alpha-1}\,dx,
	\end{equation}
	and
	\begin{equation}\label{eq:holder-L-eps}
		L(r):=\int_{B_{r}}u\Delta u\,\mu_{r}^{\alpha}\,dx
		-
		\lambda r^{2}\int_{B_{r}}V_{\varepsilon}uw\,\mu_{r}^{\alpha}\,dx.
	\end{equation}
    Note that $H(r)=\int_{B_{r}}\bigl(|u|^2+\lambda r^{4}|w|^2\bigr)\mu_{r}^{\alpha-1}\,dx $ in the last section.
	We split $L(r)$ as
	\begin{equation}\label{eq:holder-L-eps-split}
		L(r)=L_{1}(r)+L_{2}(r),
	\end{equation}
	where
	\begin{equation}\label{eq:holder-L1-L2-eps}
		L_{1}(r):=\int_{B_{r}}uw\,\mu_{r}^{\alpha}\,dx,
		\qquad
		L_{2}(r):=-\lambda r^{2}\int_{B_{r}}V_{\varepsilon}uw\,\mu_{r}^{\alpha}\,dx.
	\end{equation}
	We also define the error produced by replacing \(\Delta w\) with its smooth part:
	\begin{equation}\label{eq:holder-F-eps}
		F(r):=\lambda r^{2}\int_{B_{r}}f_{\varepsilon}uw\,\mu_{r}^{\alpha}\,dx.
	\end{equation}
	Finally, set
	\begin{equation}\label{eq:holder-I-N-eps}
		I(r):=D(r)+L(r),
		\qquad
		N(r):=\frac{I(r)}{H(r)}.
	\end{equation}

We will study the derivatives of $N(r)$ to obtain the monotonicity results. Since the arguments are similar to those in
the previous section, we will skip the calculations and show the proof which requires different treatments.
    Following the arguments in Lemma \ref{lem:Hdr-HD-prime}, we can show
	the derivatives of \(H\) and \(D\) as follows.
		The following identities hold,
		\begin{equation}\label{eq:holder-H-prime-h}
			H'(r)=\frac{2\alpha+n-2}{r}H(r)+\frac{1}{\alpha r}\bigl(D(r)+L(r)\bigr)+\frac{1}{\alpha r}F(r),
		\end{equation}
		\begin{align}
			D'(r)
			=&
			\frac{2\alpha+n-2}{r}D(r)
			+
			\frac{4\alpha}{r}\int_{B_{r}}
			\Biggl[
			\left(x\cdot \nabla u-\frac{\mu_{r}}{4\alpha}w\right)^{2}
			\notag\\
			&\quad
			+
			\lambda r^{2}
			\left(
			x\cdot \nabla w+w+
			\frac{\mu_{r}}{4\alpha}\bigl(V_{\varepsilon}u-f_{\varepsilon}u\bigr)
			\right)^{2}
			\Biggr]\mu_{r}^{\alpha-1}\,dx
			\notag\\
			&\quad
			-
			\frac{1}{4\alpha r}\int_{B_{r}}
			\bigl(|w|^2+\lambda r^{2}(V_{\varepsilon}u-f_{\varepsilon}u)^{2}\bigr)\mu_{r}^{\alpha+1}\,dx.
			\label{eq:holder-D-prime-eps}
		\end{align}

The same arguments of Lemma \ref{lem:Hdr-L-prime}
show the derivatives $L_{1}'(r)$ and $L_{2}'(r)$. Thus,
		the following identities hold similarly,
		\begin{align}
			L_{1}'(r)
			&=
			\frac{2\alpha+n}{r}L_{1}(r)
			+
			\frac{1}{2(\alpha+1)r}
			\int_{B_{r}}
			\bigl(2\nabla u\cdot \nabla w+|w|^2-V_{\varepsilon}|u|^2+f_{\varepsilon}|u|^2\bigr)
			\mu_{r}^{\alpha+1}\,dx,
			\label{eq:holder-L1-eps-prime}
		\end{align}
		\begin{align}
			L_{2}'(r)
			&=
			\frac{2\alpha+n+2}{r}L_{2}(r)
			+
			\frac{\lambda r^{2}}{2(\alpha+1)r}
			\int_{B_{r}}
			\bigl(
			V_{\varepsilon}^{2}|u|^2-V_{\varepsilon}f_{\varepsilon}|u|^2
			-V_{\varepsilon}|w|^2-2V_{\varepsilon}\nabla u\cdot \nabla w
			\notag\\
			&\quad
			-2u\nabla w\cdot \nabla V_{\varepsilon}
			-2w\nabla u\cdot \nabla V_{\varepsilon}
			-uw\Delta V_{\varepsilon}
			\bigr)\mu_{r}^{\alpha+1}\,dx,
			\label{eq:holder-L2-eps-prime}
		\end{align}
		and therefore
		\begin{align}
			L'(r)
			&=
			\frac{2\alpha+n}{r}L(r)+\frac{2}{r}L_{2}(r)
			+
			\frac{1}{2(\alpha+1)r}
			\int_{B_{r}}
			\bigl(2\nabla u\cdot \nabla w+|w|^2-V_{\varepsilon}|u|^2+f_{\varepsilon}|u|^2\bigr)
			\mu_{r}^{\alpha+1}\,dx
			\notag\\
			&\quad
			+
			\frac{\lambda r^{2}}{2(\alpha+1)r}
			\int_{B_{r}}
			\bigl(
			V_{\varepsilon}^{2}|u|^2-V_{\varepsilon}f_{\varepsilon}|u|^2
			-V_{\varepsilon}|w|^2-2V_{\varepsilon}\nabla u\cdot \nabla w
			\notag\\
			&\hspace{3.5cm}
			-2u\nabla w\cdot \nabla V_{\varepsilon}
			-2w\nabla u\cdot \nabla V_{\varepsilon}
			-uw\Delta V_{\varepsilon}
			\bigr)\mu_{r}^{\alpha+1}\,dx.
			\label{eq:holder-L-eps-prime}
		\end{align}

Based on the derivatives of $H(r)$, $D(r)$ and $L(r)$, we can compute the derivative of $N(r)$. Let us first introduce some notations.
		Define
		\begin{equation}\label{eq:holder-B-eps-h}
			B_{\varepsilon}(r):=\frac{r}{2\sqrt{\lambda}}+\frac{G\sqrt{\lambda}}{2}r^{3},
		\end{equation}
		\begin{equation}\label{eq:holder-b-eps-h}
			b_{\varepsilon}(r):=\frac{1}{2(\alpha+1)\sqrt{\lambda}}+\frac{G\sqrt{\lambda}r^{2}}{2(\alpha+1)},
		\end{equation}
		\begin{align}
			a_{\varepsilon}(r)
			&:=
			\sqrt{\lambda}G r^{2}
			+
			C\sqrt{\lambda}G\varepsilon^{\beta-1}r^{3}
			+
			\frac{C\sqrt{\lambda}G\varepsilon^{\beta-2}}{\alpha+1}r^{3}
			+
			\frac{G}{2(\alpha+1)}r^{3}
			+
			\frac{CG\varepsilon^{\beta}}{\alpha+1}r^{3}
			\notag\\
			&\qquad
			+
			\frac{C\lambda G^{2}\varepsilon^{\beta}}{\alpha+1}r^{5}
			+
\frac{C\sqrt{\lambda}G\varepsilon^{\beta}r^{2}}{\alpha}B_{\varepsilon}(r),
			\label{eq:holder-a-eps}
		\end{align}
		\begin{equation}\label{eq:holder-A-eps-h}
			A_{\varepsilon}(r):=b_{\varepsilon}(r)+\frac{B_{\varepsilon}(r)}{\alpha r},
		\end{equation}
		\begin{equation}\label{eq:holder-d-eps-h}
			d_{\varepsilon}(r):=\frac{r}{4\alpha\lambda}+\frac{\lambda G^{2}}{4\alpha}r^{5}.
		\end{equation}
		Set
		\begin{equation}\label{eq:holder-P-eps-h}
			P_{\varepsilon}(r):=d_{\varepsilon}(r)+a_{\varepsilon}(r)+\left(b_{\varepsilon}(r)+\frac{2}{r}\right)B_{\varepsilon}(r)+\frac{2B^{2}_{\varepsilon}(r)}{\alpha r}.
		\end{equation}
        Since $N'(r)=\frac{I' H-H' I}{H^2}$, as in the Proposition \ref{prop:Hdr-N-differential-corrected},  we can obtain that
		\begin{equation}\label{eq:holder-diff-ineq-N-eps-}
			N'(r)+A_{\varepsilon}(r)N(r)\ge -P_{\varepsilon}(r)
		\end{equation}
for every \(0<r<R-\varepsilon\).

Define
		\begin{equation}\label{eq:holder-Nbar-eps-h}
			\mathcal{N}_{\varepsilon}(r)
			:=
			e^{\int_{r_{0}}^{r}A_{\varepsilon}(t)\,dt}N(r)
			+
			\int_{r_{0}}^{r}e^{\int_{r_{0}}^{s}A_{\varepsilon}(t)\,dt}P_{\varepsilon}(s)\,ds.
		\end{equation}
Thanks to (\ref{eq:holder-diff-ineq-N-eps-}), we can show that  \(\mathcal{N}_{\varepsilon}(r)\) is nondecreasing on \((r_{0},R-\varepsilon)\).

\begin{remark}
    The advantage of choosing $H(r)=\int_{B_{r}}\bigl(|u|^2+\lambda r^{2}|w|^2\bigr)\mu_{r}^{\alpha-1}\,dx$ instead of $H(r)=\int_{B_{r}}\bigl(|u|^2+\lambda r^{4}|w|^2\bigr)\mu_{r}^{\alpha-1}\,dx$ is that there is no $\frac{1}{r}$-term in $A_\varepsilon(r)$ compared with $A(r)$ in (\ref{eq:Hdr-A-expanded}). The $\frac{1}{r}$-term will lead to $\log \frac{r_2}{r_1}$ term in the three-ball inequality which prevents us from getting the sharp vanishing order as discussed in Remark \ref{rem-1}.
\end{remark}

        With aid of the monotonicity of  \(\mathcal{N}_{\varepsilon}(r)\), we aim to show the three-ball inequality. 
		Let
		$0<r_{1}<r_{2}<2r_{2}<r_{3}<R-\varepsilon.$ We still introduce some notations. 
        \begin{equation}\label{eq:holder-Phi-eps-h}
			\Phi_{\varepsilon}(a,b):=\int_{a}^{b}A_{\varepsilon}(t)\,dt.
		\end{equation}
		Define
		\begin{equation}\label{eq:holder-prop3-c1-c2-h}
			c_{1}:=(2\alpha+n-2)\log\frac{2r_{2}}{r_{1}},
			\qquad
			c_{2}:=(2\alpha+n-2)\log\frac{r_{3}}{2r_{2}},
		\end{equation}
		\begin{equation}\label{eq:holder-prop3-beta-gamma-h}
			\beta_{\varepsilon}:=\frac{1}{\alpha}\int_{r_{1}}^{2r_{2}}e^{\Phi_{\varepsilon}(t,2r_{2})}\frac{dt}{t},
			\qquad
			\gamma_{\varepsilon}:=\frac{1}{\alpha}\int_{2r_{2}}^{r_{3}}e^{-\Phi_{\varepsilon}(2r_{2},t)}\frac{dt}{t},
		\end{equation}
		\begin{equation}\label{eq:holder-prop3-R1-R2-h}
			R_{1,\varepsilon}:=\frac{1}{\alpha}\int_{r_{1}}^{2r_{2}}\frac{1}{t}
			\left(\int_{t}^{2r_{2}}e^{\Phi_{\varepsilon}(t,s)}P_{\varepsilon}(s)\,ds\right)dt,
			\qquad
			R_{2,\varepsilon}:=\frac{1}{\alpha}\int_{2r_{2}}^{r_{3}}\frac{1}{t}
			\left(\int_{2r_{2}}^{t}e^{-\Phi_{\varepsilon}(s,t)}P_{\varepsilon}(s)\,ds\right)dt,
		\end{equation}
		and
		\begin{equation}\label{eq:holder-prop3-Q-h}
			Q_{\varepsilon}(r):=\frac{C\sqrt{\lambda}G\varepsilon^{\beta}}{\alpha}r^{2},
			\qquad
			Q_{1,\varepsilon}:=\int_{r_{1}}^{2r_{2}}Q_{\varepsilon}(t)\,dt,
			\qquad
			Q_{2,\varepsilon}:=\int_{2r_{2}}^{r_{3}}Q_{\varepsilon}(t)\,dt.
		\end{equation}
		Furthermore, set
		\begin{equation}\label{eq:holder-prop3-kappa-h}
			\kappa_{\varepsilon}:=\kappa_{\varepsilon}(r_1, r_2, r_3)=\frac{\gamma_{\varepsilon}}{\beta_{\varepsilon}+\gamma_{\varepsilon}},
		\end{equation}
		and
		\begin{equation}\label{eq:holder-prop3-LambdaH-h}
			\widetilde{\lambda}_{H,\varepsilon}
			:=
			\kappa_{\varepsilon}R_{1,\varepsilon}
			+
			(1-\kappa_{\varepsilon})R_{2,\varepsilon}
			+
			\kappa_{\varepsilon}Q_{1,\varepsilon}
			+
			(1-\kappa_{\varepsilon})Q_{2,\varepsilon}.
		\end{equation}
	Following the arguments in Proposition \ref{prop:Hdr-H-three-ball-from-calN}, the three-ball inequality for $H(r)$ holds,
		\begin{equation}\label{eq:holder-prop3-conclusion-h}
			H(2r_{2})
			\le
			e^{\widetilde{\lambda}_{H,\varepsilon}}
			H(r_{1})^{\kappa_{\varepsilon}}H(r_{3})^{1-\kappa_{\varepsilon}}.
		\end{equation}
We optimize the three-ball inequality for $H(r).$
		\begin{proposition}
		Assume moreover that \(G\ge 16\). Let
		\begin{equation}\label{eq:holder-thm3-choice}
			\lambda:=G^{-1},
			\qquad
			\alpha:=G^{1/4},
			\qquad
			\varepsilon:=G^{-\frac{1}{2(2-\beta)}}.
		\end{equation}
		Let
		$0<r_{1}<r_{2}<2r_{2}<r_{3}<R-\varepsilon.
		$
		For the corresponding quantities
		\(\beta_{\varepsilon},\gamma_{\varepsilon},R_{1,\varepsilon},R_{2,\varepsilon},\kappa_{\varepsilon}\)
		in (\ref{eq:holder-prop3-beta-gamma-h})--(\ref{eq:holder-prop3-kappa-h}), there exists a constant
		\(C=C(n,R,\beta)>0\) such that
		\begin{equation}\label{eq:holder-thm3-conclusion}
			H(2r_{2})
			\le
			\exp\!\left(
			CG^{1/4}(\log\frac{r_{3}}{2r_{2}}+1)
			\right)
			H(r_{1})^{\kappa_{\varepsilon}}H(r_{3})^{1-\kappa_{\varepsilon}}.
		\end{equation}
	\end{proposition}
	
	\begin{proof}	
		We aim to bound $\widetilde{\lambda}_{H,\varepsilon}$ in (\ref{eq:holder-prop3-LambdaH-h}).
		We first bound $A_\varepsilon(r)$ and $P_\varepsilon(r)$.
		Under the choice \eqref{eq:holder-thm3-choice}, we have
		\[
		\frac{1}{\sqrt{\lambda}}=\sqrt{G},
		\qquad
		\sqrt{\lambda}G=\sqrt{G},
		\qquad
		\alpha=G^{1/4}.
		\]
		Hence \(A_{\varepsilon}\) has  the same expression
		\begin{align}
			A_{\varepsilon}(r)
			&=
			\frac{\sqrt{G}}{2(\alpha+1)}
			+
			\frac{\sqrt{G}}{2(\alpha+1)}r^{2}
			+
			\frac{\sqrt{G}}{2\alpha}
			+
			\frac{\sqrt{G}}{2\alpha}r^{2}.
			\label{eq:holder-thm3-A-explicit}
		\end{align}
		Thus there exist constants \(c_{0},C_{0}>0\), depending only on \(R\), such that
		\begin{equation}\label{eq:holder-thm3-A-two-sided}
			c_{0}G^{1/4}\le A_{\varepsilon}(r)\le C_{0}G^{1/4}
			\qquad \text{for all }0<r<R-\varepsilon.
		\end{equation}
		
		We next estimate \(P_{\varepsilon}\). Since
		$\varepsilon^{\beta-2}=G^{1/2},$
		$\varepsilon^{\beta-1}\le G^{1/2},$
		we obtain, from \eqref{eq:holder-a-eps},
		\[
		a_{\varepsilon}(r)\le C(R,\beta)G^{3/4}.
		\]
		Moreover,
		\[
		d_{\varepsilon}(r)=\frac{G r}{4\alpha}+\frac{G r^{5}}{4\alpha}\le C(R)G^{3/4}.
		\]
		From \eqref{eq:holder-B-eps-h} and \eqref{eq:holder-b-eps-h}, we get
		\[
		B_{\varepsilon}(r)\le C(R)G^{1/2},
		\qquad
		b_{\varepsilon}(r)\le C(R)G^{1/4}.
		\]
		Consequently,
		\[
		\left(b_{\varepsilon}(r)+\frac{2}{r}\right)B_{\varepsilon}(r)\le C(R)G^{3/4},
		\qquad
		\frac{2B^{2}_{\varepsilon}(r)}{\alpha r}\le C(R)G^{3/4}.
		\]
		Combining these estimates gives
		\begin{equation}\label{eq:holder-thm3-P-upper}
			0\le P_{\varepsilon}(r)\le C(R,\beta)G^{3/4}
			\qquad \text{for all }0<r<R-\varepsilon.
		\end{equation}
		
		We next 
	 estimate $\kappa_{\varepsilon}R_{1,\varepsilon}$.
		The argument is similar to those in the last section. However, we show the details of calculations, because $A_\varepsilon(r)$ is different from $A(r)$ in the last section. We rewrite
		\[
		R_{1,\varepsilon}
		=
		\frac{1}{\alpha}\int_{r_{1}}^{2r_{2}}e^{\Phi_{\varepsilon}(t,2r_{2})}\frac{1}{t}
		\left(\int_{t}^{2r_{2}}e^{-\Phi_{\varepsilon}(s,2r_{2})}P_{\varepsilon}(s)\,ds\right)dt.
		\]
		Define
		\begin{equation}\label{eq:holder-thm3-I1}
			I_{1,\varepsilon}:=
			\int_{r_{1}}^{2r_{2}}e^{-\Phi_{\varepsilon}(s,2r_{2})}P_{\varepsilon}(s)\,ds.
		\end{equation}
		Then
		\[
		R_{1,\varepsilon}\le \beta_{\varepsilon}I_{1,\varepsilon},
		\qquad
		\kappa_{\varepsilon}R_{1,\varepsilon}\le \gamma_{\varepsilon}I_{1,\varepsilon}.
		\]
		It follows from  \eqref{eq:holder-thm3-A-two-sided} that,
		\[
		\Phi_{\varepsilon}(s,2r_{2})\ge c_{0}G^{1/4}(2r_{2}-s).
		\]
		Using \eqref{eq:holder-thm3-P-upper}, we get
		\begin{align}
			I_{1,\varepsilon}
			\le
C(R,\beta)G^{3/4}\int_{r_{1}}^{2r_{2}}e^{-c_{0}G^{1/4}(2r_{2}-s)}\,ds
			\le
			C(R,\beta)G^{1/2}.
			\label{eq:holder-thm3-I1-bound}
		\end{align}
		Since \(A_{\varepsilon}\ge0\), we have
		\begin{equation}\label{eq:holder-thm3-gamma-upper}
			\gamma_{\varepsilon}
			\le
			\frac{1}{\alpha}\log\frac{r_{3}}{2r_{2}}
			=
			G^{-1/4}\log\frac{r_{3}}{2r_{2}}.
		\end{equation}
		Therefore
		\begin{equation}\label{eq:holder-thm3-kappaR1}
			\kappa_{\varepsilon}R_{1,\varepsilon}
			\le
			C(R,\beta)G^{1/4}\log\frac{r_{3}}{2r_{2}}.
		\end{equation}
		Then we 
		estimate \((1-\kappa_{\varepsilon})R_{2,\varepsilon}\).
		Since \(1-\kappa_{\varepsilon}\le1\), it is enough to estimate \(R_{2,\varepsilon}\). Define
		\[
		K_{\varepsilon}(s):=\int_{s}^{r_{3}}e^{-\Phi_{\varepsilon}(s,t)}\frac{dt}{t},
		\qquad 2r_{2}\le s\le r_{3}.
		\]
		It holds that
		\[
		R_{2,\varepsilon}=\frac{1}{\alpha}\int_{2r_{2}}^{r_{3}}P_{\varepsilon}(s)K_{\varepsilon}(s)\,ds.
		\]
		By \eqref{eq:holder-thm3-A-two-sided},
		\[
		K_{\varepsilon}(s)
		\le
		\frac{1}{s}\int_{s}^{r_{3}}e^{-c_{0}G^{1/4}(t-s)}\,dt
		\le
		\frac{C}{G^{1/4}s}.
		\]
		It follows from \eqref{eq:holder-thm3-P-upper} that
		\begin{align}
			R_{2,\varepsilon}
			&\le
			\frac{C(R,\beta)G^{3/4}}{\alpha}\int_{2r_{2}}^{r_{3}}\frac{1}{G^{1/4}s}\,ds
			=
	C(R,\beta)G^{1/4}\log\frac{r_{3}}{2r_{2}}.
			\label{eq:holder-thm3-R2-bound}
		\end{align}
		Thus
		\begin{equation}\label{eq:holder-thm3-one-minus-kappa-R2}
			(1-\kappa_{\varepsilon})R_{2,\varepsilon}
			\le
			C(R,\beta)G^{1/4}\log\frac{r_{3}}{2r_{2}}.
		\end{equation}
		
		We continue to estimate the forcing contribution.
		Thanks to \eqref{eq:holder-thm3-choice}, the function \(Q_{\varepsilon}\) satisfies
		\[
		Q_{\varepsilon}(r)\le C(R,\beta)G^{1/4}r^{2}.
		\]
	 Hence
		\[
		Q_{1,\varepsilon}
		\le
		C(R,\beta)G^{1/4} (2r_2)^3-r_1^3),
		\qquad
		Q_{2,\varepsilon}
		\le
		C(R,\beta)G^{1/4} (r_3^3-(2r_2)^3).
		\]

Summing up the last inequality with the estimates
		from \eqref{eq:holder-thm3-kappaR1} to \eqref{eq:holder-thm3-one-minus-kappa-R2},  to have
		\begin{equation}\label{eq:holder-thm3-LambdaH-bound}
			\widetilde{\lambda}_{H,\varepsilon}
			\le
			C(R,\beta)G^{1/4}(\log\frac{r_{3}}{2r_{2}}+1).
		\end{equation}

		At last, substituting \eqref{eq:holder-thm3-LambdaH-bound} into \eqref{eq:holder-prop3-conclusion-h}, we obtain
		\eqref{eq:holder-thm3-conclusion}.
	\end{proof}

We need to have a three-ball inequality without weight function $\mu(x)$.
		Define
\begin{equation}\label{eq:holder-thm4-h}	h(r):=\int_{B_{r}}\bigl(|u|^2+G^{-1}r^{2}(\Delta u)^{2}\bigr)\,dx.
		\end{equation}
         By the comparison of $H(r)$ and $h(r)$, we can show that there exists a constant
		\(C=C(n,R,\beta)>0\) such that
	\begin{equation}\label{eq:holder-thm4-conclusion}
			h(r_{2})
			\le
			\exp\!\left(
			CG^{1/4}(\log\frac{r_{3}}{r_{2}}+1)
			\right)
			h(r_{1})^{\kappa_{\varepsilon}}h(r_{3})^{1-\kappa_{\varepsilon}}.
		\end{equation}
	
Next we want to get the three-ball inequality in the $L^\infty$ norm.
	For \(r<\rho\), define
	\begin{equation}\label{eq:intro-A-smooth}
		\mathcal B_{G}(r,\rho)
		:=r^{n/2}+G^{1/2}r(\rho-r)^{-2}\rho^{n/2}.
	\end{equation}

Let 
$\kappa^1_{\varepsilon}:=\kappa_{\varepsilon}(r_{1},2r_{2},r_{3}),$ where $\kappa_{\varepsilon}(r_{1},r_{2},r_{3})$ is given in 
(\ref{eq:holder-prop3-kappa-h}). Let $0<r_{1}<2r_1<r_{2}<4r_{2}<r_{3}<2r_{3}<R-\varepsilon.$
		 From (\ref{eq:holder-thm4-conclusion}), there exists \(C=C(n,R,\beta)>0\) such that
		\begin{align}\label{eq:holder-thm5-conclusion-hh}
			\|u\|_{L^{\infty}(B_{r_{2}})}
			&\le
			C(1+Gr_{2}^{4})^{C}r_{2}^{-n/2}
			\exp\!\left(CG^{1/4}\log\frac{r_{3}}{2r_{2}}\right)\notag\\
			&\qquad\times
			\mathcal B_{G}(r_{1},2r_{1})^{\kappa^1_{\varepsilon}}
			\mathcal B_{G}(r_{3}, 2r_3)^{1-\kappa^1_{\varepsilon}}
			\|u\|_{L^{\infty}(B_{2r_1})}^{\kappa^1_{\varepsilon}}
			\|u\|_{L^{\infty}(B_{2r_3})}^{1-\kappa^1_{\varepsilon}}.
		\end{align}

We are ready to give the proof of the vanishing order for the solutions of bi-Laplace equations with H\"older potentials.
	\begin{proof}[Proof of Theorem \ref{thm-4}]
		We apply the \(L^{\infty}\)-three-ball inequality in
		(\ref{eq:holder-thm5-conclusion-hh}) with \(R=10\).
		Fix
		$0<r\le r_{*}, $ with $r_{*}=\frac1{100}G^{-1/4}\le\frac1{200}.$
		We choose
		\[
		r_{1}:=\frac r2,
		\qquad
		r_{2}:=r_{*},
		\qquad
		r_{3}:=8r_{*}.
		\]

	We aim to  estimate the parameter \(\kappa^1_{\varepsilon}\).
	Recall that $\kappa^1_{\varepsilon}=\kappa_\varepsilon(r_1, 2r_2, r_3)$ where $\kappa_{\varepsilon}(r_{1},r_{2},r_{3})$ is given in 
(\ref{eq:holder-prop3-kappa-h}).
		We apply (\ref{eq:holder-thm5-conclusion-hh}) 
        to the triple
		$(r_{1},r_{2},r_{3})
		=
		\left(\frac r2,r_{*},8r_{*}\right).$
		By the definition of \(\beta_{\varepsilon}\) and \(\gamma_{\varepsilon}\) in \eqref{eq:holder-prop3-beta-gamma-h}, we have
		\[
		\beta_{\varepsilon}
		=
		\frac1\alpha
		\int_{r/2}^{4r_{*}}
		e^{\Phi_{\varepsilon}(t,4r_{*})}
		\frac{dt}{t}, \quad 
		\mbox{and} \quad 
		\gamma_{\varepsilon}
		=
		\frac1\alpha
		\int_{4r_{*}}^{8r_{*}}
		e^{-\Phi_{\varepsilon}(4r_{*},t)}
		\frac{dt}{t}.
		\]
		
		We first estimate \(\beta_{\varepsilon}\). For $
		t\in\left[\frac r2,4r_{*}\right],$
		we have, by \eqref{eq:holder-thm3-A-two-sided},
		\[
		\Phi_{\varepsilon}(t,4r_{*})
		=
		\int_{t}^{4r_{*}}A_{\varepsilon}(\tau)\,d\tau
		\le
		CG^{1/4}r_{*} \leq C.
		\]
	
Since \(\alpha=G^{1/4}\), it holds that
		\[
		\beta_{\varepsilon}
		\le
		\frac{C}{\alpha}
		\int_{r/2}^{4r_{*}}\frac{dt}{t}\leq 
		CG^{-1/4}\log\frac{8r_{*}}{r}.
		\]

		Next we estimate \(\gamma_{\varepsilon}\). For $t\in[4r_{*},8r_{*}],$
		we get
		\[
		\Phi_{\varepsilon}(4r_{*},t)
		=
		\int_{4r_{*}}^{t}A_{\varepsilon}(\tau)\,d\tau
		\le
		CG^{1/4}(t-4r_{*})\leq C.
		\]
	
		Therefore
	$e^{-\Phi_{\varepsilon}(4r_{*},t)}\ge c$
		for every $t\in[4r_{*},8r_{*}]$.
		Hence
		\[
		\gamma_{\varepsilon}
		\ge
		\frac{c}{\alpha}
		\int_{4r_{*}}^{8r_{*}}\frac{dt}{t}\ge cG^{-1/4}.
		\]
		
		Combining the estimates for \(\beta_{\varepsilon}\) and \(\gamma_{\varepsilon}\), we get
		\[
		\frac{\beta_{\varepsilon}}{\gamma_{\varepsilon}}
		\le
		C\log\frac{8r_{*}}{r}.
		\]
		Therefore
		\begin{align}\label{eq:holder-small-kappa}
		(\kappa^1_{\varepsilon})^{-1}
		=
		1+
		\frac{\beta_{\varepsilon}}{\gamma_{\varepsilon}}
		\le
		C\left(1+\log\frac{8r_{*}}{r}\right)\leq C\log\frac{16r_{*}}{r} \leq C \log\frac{1}{r} 
		\end{align}
		for \(0<r\le r_{*}\). 
		
	Now we  apply the \(L^{\infty}\)-three-ball inequality (\ref{eq:holder-thm5-conclusion-hh}) to have
		\[
		\begin{aligned}
			\|u\|_{L^{\infty}(B_{r_{*}})}
			&\le
			 Cr_{*}^{-n/2}
			\exp\left(
			CG^{1/4}\log\frac{8r_{*}}{2r_{*}}
			\right)
			\\
			&\quad\times
			\mathcal{B}_G\left(\frac r2,r\right)^{\kappa^1_{\varepsilon}}
		\mathcal{B}_G(8r_{*},16r_{*})^{1-\kappa^1_{\varepsilon}}
			\|u\|_{L^{\infty}(B_{r})}^{\kappa^1_{\varepsilon}}
			\|u\|_{L^{\infty}(B_{16r_{*}})}^{1-\kappa^1_{\varepsilon}}.
		\end{aligned}
		\]
		Since 
		$r_{*}^{-n/2}=100^{n/2}G^{n/8},$
		we get 
		\[
		\begin{aligned}
			\|u\|_{L^{\infty}(B_{r_{*}})}
			&\le
			C\exp(CG^{1/4})
			\mathcal{B}_G\left(\frac r2,r\right)^{\kappa^1_{\varepsilon}}
		\mathcal{B}_G(8r_{*},16r_{*})^{1-\kappa^1_{\varepsilon}}
			\|u\|_{L^{\infty}(B_{r})}^{\kappa^1_{\varepsilon}}
			\|u\|_{L^{\infty}(B_{16r_{*}})}^{1-\kappa^1_{\varepsilon}}.
		\end{aligned}
		\]
		
		Recall that 
			$\mathcal{B}_G$ is given in (\ref{eq:intro-A-smooth}). From its definition and the value of $r_\ast$, we get
		\begin{align*}
			\mathcal{B}_G\left(\frac r2,r\right)
			\leq 
			\left(\frac r2\right)^{n/2}
			+
			2G^{1/2}r^{n/2-1}\leq C G.
		\end{align*}
and
		\[
		\begin{aligned}
			\mathcal B_G(8r_{*},16r_{*})
			\le
			Cr_{*}^{n/2}
			+
			CG^{1/2}
			r_{*}^{n/2-1}\leq CG.
		\end{aligned}
		\]
		
From the normalization assumptions in (\ref{eq:intro-second-order-normalization-h}),
we obtain
		\[
		1
		\le
		C\exp(CG^{1/4})
		\|u\|_{L^{\infty}(B_{r})}^{\kappa^1_{\varepsilon}}
		\mathbb{M}^{1-\kappa^1_{\varepsilon}}.
		\]
		Taking the exponent  \(1/\kappa_{\varepsilon}\) to both sides of the last inequality yields that
		\[
		\|u\|_{L^{\infty}(B_{r})}
		\ge
		\exp\left(-\frac{C(G^{1/4}+\log\mathbb{M})}{\kappa^1_{\varepsilon}}\right).
		\]

		From \eqref{eq:holder-small-kappa},
		we arrive at
		\[
		\|u\|_{L^{\infty}(B_{r})}
		\ge
		\exp\left(
		-C(G^{1/4}+\log\mathbb{M})\log\frac1r
		\right)
		\]
        for $r_*$ small.
		That is,
		\[
		\|u\|_{L^{\infty}(B_{r})}
		\ge
		r^{C(G^{1/4}+\log\mathbb{M})}.
		\]
Thus, we prove the  order of vanishing for solutions is at most $C(G^{1/4}+\log \mathbb{M})$ at the origin.
Since  the obtained three-ball inequalities   are translation invariant, we can apply the same arguments to the balls centered at  any point $x_0\in B_{r_\ast/2}$. Thus the proof of Theorem is completed.

	\end{proof}

    \section{Appendix}
In the appendix, we show the proof of Lemma \ref{lem:weighted-differentiation} and the quantitative regularity estimates for $W^{2,2}$ norm.
    	\begin{proof}[Proof of Lemma \ref{lem:weighted-differentiation}]
		Differentiating \eqref{eq:F-definition} with respect to \(r\), we have
		\begin{equation}\label{eq:F-prime-start}
			F'(r)
			=
			2\gamma r\int_{B_{r}}f(x)\mu_{r}(x)^{\gamma-1}\,dx.
		\end{equation}
		Since
		$r^{2}=\mu_{r}(x)+|x|^{2},$
		we rewrite \eqref{eq:F-prime-start} as
		\begin{align}
			F'(r)
			&=
			\frac{2\gamma}{r}\int_{B_{r}}f(x)\mu_{r}(x)^{\gamma}\,dx
			+
			\frac{2\gamma}{r}\int_{B_{r}}f(x)|x|^{2}\mu_{r}(x)^{\gamma-1}\,dx \notag\\
			&=
			\frac{2\gamma}{r}F(r)
			+
			\frac{2\gamma}{r}\int_{B_{r}}f(x)|x|^{2}\mu_{r}(x)^{\gamma-1}\,dx.
			\label{eq:F-prime-decompose}
		\end{align}
		 Since
		$\nabla\bigl(\mu_{r}^{\gamma}\bigr)=-2\gamma x\,\mu_{r}^{\gamma-1},$
		integration by parts argument yields that
		\begin{align}
			2\gamma\int_{B_{r}}f(x)|x|^{2}\mu_{r}^{\gamma-1}\,dx
			&=
			-\int_{B_{r}}f(x)\,x\cdot \nabla(\mu_{r}^{\gamma})\,dx \notag\\
			&=
			\int_{B_{r}}\operatorname{div}(f(x)x)\,\mu_{r}^{\gamma}\,dx \notag\\
			&=
			nF(r)+\int_{B_{r}}\nabla f(x)\cdot x\,\mu_{r}^{\gamma}\,dx.
			\label{eq:weighted-ibp}
		\end{align}
		Substituting \eqref{eq:weighted-ibp} into \eqref{eq:F-prime-decompose}, we arrive at \eqref{eq:F-prime-first}.
		
		To derive \eqref{eq:F-prime-second}, we integrate by parts once more to have
		\begin{align}
			\int_{B_{r}}\nabla f(x)\cdot x\,\mu_{r}^{\gamma}\,dx
			&=
			\frac{1}{2(\gamma+1)}
			\int_{B_{r}}\nabla f(x)\cdot\bigl(-\nabla(\mu_{r}^{\gamma+1})\bigr)\,dx \notag\\
			&=
			\frac{1}{2(\gamma+1)}
			\int_{B_{r}}\Delta f(x)\,\mu_{r}^{\gamma+1}\,dx.
			\label{eq:second-ibp}
		\end{align}
		Combining \eqref{eq:F-prime-first} and \eqref{eq:second-ibp}, we obtain \eqref{eq:F-prime-second}.
	\end{proof}
	
	To control the $\Delta u$ term in the definition of $H(r)$, we also need the following quantitative local $W^{2,2}$--$L^\infty$ estimate.
	\begin{lemma}\label{cor:local-W22-Linfty-correct}
		Assume \(n\ge 2\). Let \(0<r<\rho<R\). Then
		\begin{equation}\label{eq:local-W22-Linfty-correct}
			\|u\|_{W^{2,2}(B_r)}
			\le
			C(n,R)\,(1+G\rho^{4})\,(\rho-r)^{-2}\rho^{n/2}\,
			\|u\|_{L^\infty(B_\rho)}.
		\end{equation}
	\end{lemma}
	
	\begin{proof}
		We shall use the following interior regularity result for higher order elliptic equations in  Corollary 17.3 in \cite{H07}. That is,  
		for every multi-index \(\beta\) with
		\(|\beta|\le 3\),
		\begin{equation}\label{eq:Zhu-42-used}
			\|d^{|\beta|}D^\beta u\|_{L^2(B_\rho)}
			\le
			C
			\Bigl(
			\|d^{4}\Delta^{2}u\|_{L^2(B_\rho)}+\|u\|_{L^2(B_\rho)}
			\Bigr)^{|\beta|/4}
			\|u\|_{L^2(B_\rho)}^{1-|\beta|/4},
		\end{equation}
		where
		\[
		d(x):=\operatorname{dist}(x,\mathbb R^{n}\setminus B_\rho)=\rho-|x|.
		\]
		
		We only need the cases \(|\beta|=0,1,2\).
		Define
	\begin{equation}\label{eq:def-A-rho}
			A_\rho
			:=
			\|d^{4}\Delta^{2}u\|_{L^2(B_\rho)}+\|u\|_{L^2(B_\rho)}.
		\end{equation}
		Since \(d(x)\le \rho\) on \(B_\rho\), and \(\Delta^{2}u=-Vu\), we obtain
		\begin{align}
		\|d^{4}\Delta^{2}u\|_{L^2(B_\rho)}
			=
			\|d^{4}Vu\|_{L^2(B_\rho)}
			\le
\|V\|_{L^\infty(B_\rho)}\,\|d^{4}u\|_{L^2(B_\rho)}
			\le G\rho^{4}\|u\|_{L^2(B_\rho)}.
			\label{eq:d4Delta2u}
		\end{align}
		Hence
		\begin{equation}\label{eq:A-rho-bound}
			A_\rho
			\le
			(1+G\rho^{4})\|u\|_{L^2(B_\rho)}.
		\end{equation}
		
		Take \(|\beta|=2\) in \eqref{eq:Zhu-42-used}. Then
		\[
		\|d^{2}D^\beta u\|_{L^2(B_\rho)}
		\le
		C A_\rho^{1/2}\|u\|_{L^2(B_\rho)}^{1/2}.
		\]
		Using \eqref{eq:A-rho-bound}, we get
		\begin{align}
			\|d^{2}D^\beta u\|_{L^2(B_\rho)}
			&\le
			C\bigl((1+G\rho^{4})\|u\|_{L^2(B_\rho)}\bigr)^{1/2}
			\|u\|_{L^2(B_\rho)}^{1/2}
			\notag\\
			&=
			C(1+G\rho^{4})^{1/2}\|u\|_{L^2(B_\rho)}.
			\label{eq:second-derivative-d-weighted}
		\end{align}
		Now if \(x\in B_r\), then
		$d(x)=\rho-|x|\ge \rho-r.$
		Therefore
		\begin{align}\label{any-beta}
		(\rho-r)^{2}|D^\beta u(x)|\le d(x)^{2}|D^\beta u(x)|
		\qquad 
		\end{align}
        for $x\in B_r$ and any $\beta$.
		Taking the \(L^2(B_r)\)-norm and using \eqref{eq:second-derivative-d-weighted}, we obtain
		\begin{equation}\label{eq:D2u-final}
			\|D^\beta u\|_{L^2(B_r)}
			\le
			C(1+G\rho^{4})^{1/2}(\rho-r)^{-2}\|u\|_{L^2(B_\rho)}
			\qquad (|\beta|=2).
		\end{equation}
		
		Now we prove the case \(|\beta|=1\) in a similar way using \eqref{eq:Zhu-42-used}. Then
		\[
		\|d\,D^\beta u\|_{L^2(B_\rho)}
		\le
		C A_\rho^{1/4}\|u\|_{L^2(B_\rho)}^{3/4}.
		\]
	From \eqref{eq:A-rho-bound}, we obtain
		\begin{align}
			\|d\,D^\beta u\|_{L^2(B_\rho)}
			&\le
			C\bigl((1+G\rho^{4})\|u\|_{L^2(B_\rho)}\bigr)^{1/4}
			\|u\|_{L^2(B_\rho)}^{3/4}
			\notag\\
			&=
			C(1+G\rho^{4})^{1/4}\|u\|_{L^2(B_\rho)}.
			\label{eq:first-derivative-d-weighted}
		\end{align}
		Using  (\ref{any-beta}) and 
		taking the \(L^2(B_r)\)-norm and using \eqref{eq:first-derivative-d-weighted}, we derive that
		\begin{equation}\label{eq:D1u-final}
			\|D^\beta u\|_{L^2(B_r)}
			\le
			C(1+G\rho^{4})^{1/4}(\rho-r)^{-1}\|u\|_{L^2(B_\rho)}
			\qquad (|\beta|=1).
		\end{equation}
		
		From \eqref{eq:D2u-final} and \eqref{eq:D1u-final}, we have
		\begin{align}
			\|u\|_{W^{2,2}(B_r)}
			\le{}&
			C(n)\Bigl[
			1
			+
			(1+G\rho^{4})^{1/4}(\rho-r)^{-1}
			+
			(1+G\rho^{4})^{1/2}(\rho-r)^{-2}
			\Bigr]
			\|u\|_{L^2(B_\rho)} \notag \\
            \le
		& C(n)\,(1+G\rho^{4})(\rho-r)^{-2}\|u\|_{L^2(B_\rho)}.
			\label{eq:W22-before-absorb}
		\end{align}

		Finally, we conclude that
		\[
		\|u\|_{W^{2,2}(B_r)}
		\le
		C(n,R)\,(1+G\rho^{4})(\rho-r)^{-2}\rho^{n/2}\|u\|_{L^\infty(B_\rho)}.
		\]
		This finishes the proof of the lemma.
	\end{proof}

    \bibliographystyle{plainnat}
	\bibliography{reference.bib}

@article{GL87c,
  title={Unique continuation for elliptic operators:a geometric-variational approach},
  author={Garofalo, Nicola and Lin, Fang-Hua},
  journal={Communications on pure and applied mathematics},
  volume={40},
  number={3},
  pages={347--366},
  year={1987},
  publisher={Wiley Online Library}
}

@article{GR19,
  title={Strong unique continuation for the higher order fractional Laplacian},
  author={Garc{\'\i}a-Ferrero, Mar{\'\i}a-{\'A}ngeles and R{\"u}land, Angkana},
  journal={Mathematics in Engineering},
  volume={4},
   pages={715-774},
  year={2019}
}

@article{K07,
  title={Some recent applications of unique continuation},
  author={Kenig, Carlos E},
  journal={Contemporary Mathematics},
  volume={439},
  pages={25--56},
  year={2007},
  publisher={Providence, RI: American Mathematical Society}
}

@article{lo18,
  title={Nodal sets of Laplace eigenfunctions: polynomial upper estimates of the Hausdorff measure},
  author={Logunov, Alexander},
  journal={Annals of Mathematics},
  volume={187},
  number={1},
  pages={221--239},
  year={2018},
  publisher={Department of Mathematics, Princeton University Princeton, New Jersey, USA}
}

@article{BanerjeeGarofalo2016,
  author  = {Banerjee, Agnid and Garofalo, Nicola},
  title   = {Quantitative uniqueness for elliptic equations at the boundary of {$C^1$}, {Dini} domains},
  journal = {J. Differential Equations},
  volume  = {261},
  number  = {12},
  pages   = {6718--6757},
  year    = {2016}
}

@article{BourgainKenig2005,
  author  = {Bourgain, Jean and Kenig, Carlos E.},
  title   = {On localization in the continuous {Anderson--Bernoulli} model in higher dimension},
  journal = {Invent. Math.},
  volume  = {161},
  number  = {2},
  pages   = {389--426},
  year    = {2005}
}

@article{Davey2020,
  author  = {Davey, Blair},
  title   = {Quantitative unique continuation for {Schr\"odinger} operators},
  journal = {J. Funct. Anal.},
  volume  = {279},
  number  = {4},
  pages   = {108566},
  year    = {2020}
}

@article{LZ22,
  title={Upper bounds of nodal sets for eigenfunctions of eigenvalue problems},
  author={Lin, Fanghua and Zhu, Jiuyi},
  journal={Mathematische Annalen},
  volume={382},
  number={3},
  pages={1957--1984},
  year={2022},
  publisher={Springer}
}

@article{Davey2025,
  title={A frequency function approach to quantitative unique continuation for elliptic equations},
  author={Davey, Blair},
  journal={Journal of Mathematical Analysis and Applications},
  pages={130952},
  year={2026},
  publisher={Elsevier}
}

@article{DaveyZhu2018,
  author  = {Davey, Blair and Zhu, Jiuyi},
  title   = {Quantitative uniqueness of solutions to second order elliptic equations with singular potentials in two dimensions},
  journal = {Calc. Var. Partial Differential Equations},
  volume  = {57},
  number  = {3},
  pages   = {Paper No. 92},
  year    = {2018}
}

@article{DaveyZhu2019,
  author  = {Davey, Blair and Zhu, Jiuyi},
  title   = {Quantitative uniqueness of solutions to second-order elliptic equations with singular lower order terms},
  journal = {Comm. Partial Differential Equations},
  volume  = {44},
  number  = {12},
  pages   = {1217--1251},
  year    = {2019}
}

@article{DG26,
  title={On a class of thin obstacle-type problems for the bi-{Laplacian} operator},
  author={Danielli, Donatella and Gravina, Giovanni},
  journal={Mathematische Annalen},
  volume={395},
  number={1},
  pages={11},
  year={2026},
  publisher={Springer}
}

@article{DonnellyFefferman1988,
  author  = {Donnelly, Harold and Fefferman, Charles},
  title   = {Nodal sets of eigenfunctions on {Riemannian} manifolds},
  journal = {Invent. Math.},
  volume  = {93},
  number  = {1},
  pages   = {161--183},
  year    = {1988}
}

@article{GarofaloLin1986,
  author  = {Garofalo, Nicola and Lin, Fang-Hua},
  title   = {Monotonicity properties of variational integrals, {$A_p$} weights and unique continuation},
  journal = {Indiana Univ. Math. J.},
  volume  = {35},
  number  = {2},
  pages   = {245--268},
  year    = {1986}
}

@article{S90,
  title={Strong uniqueness theorems for second order elliptic differential equations},
  author={Sogge, Christopher D},
  journal={American Journal of Mathematics},
  volume={112},
  number={6},
  pages={943--984},
  year={1990},
  publisher={JSTOR}
}

@article{W92,
  title={A property of measures in {$\mathbb R^n$} and an application to unique continuation},
  author={Wolff, Thomas H},
  journal={Geometric \& Functional Analysis GAFA},
  volume={2},
  number={2},
  pages={225--284},
  year={1992},
  publisher={Springer}
}

@article{TWZ26,
  title={Quantitative unique continuation for elliptic equations with H\"older continuous potentials},
  author={Teng, Long and Wang, Zhiwei and Zhu, Jiuyi},
  journal={ arXiv:2603.21392},
  year={2026}
}

@article{CG99,
  title={Some remarks on strong unique continuation for the Laplace operator and its powers},
  author={Colombini, Ferruccio and Grammatico, Cataldo},
  journal={Communications in partial differential equations},
  volume={24},
  number={5-6},
  pages={1079--1094},
  year={1999},
  publisher={Taylor \& Francis}
}

@article{L07,
  title={Strong unique continuation for $m-$th powers of a Laplacian operator with singular coefficients},
  author={Lin, Ching-Lung},
  journal={Proceedings of the American Mathematical Society},
  volume={135},
  number={2},
  pages={569--578},
  year={2007}
}

@article{CK10,
  title={Strong unique continuation for products of elliptic operators of second order},
  author={Colombini, Ferruccio and Koch, Herbert},
  journal={Transactions of the American Mathematical Society},
  volume={362},
  number={1},
  pages={345--355},
  year={2010}
}

@article{Z18,
  title={Quantitative unique continuation of solutions to higher order elliptic equations with singular coefficients.},
  author={Zhu, Jiuyi},
  journal={Calculus of Variations \& Partial Differential Equations},
  volume={57},
  number={2},
  year={2018}
}

@article{ARV19,
  title={Optimal three spheres inequality at the boundary for the Kirchhoff--Love plate’s equation with Dirichlet conditions},
  author={Alessandrini, Giovanni and Rosset, Edi and Vessella, Sergio},
  journal={Archive for Rational Mechanics and Analysis},
  volume={231},
  number={3},
  pages={1455--1486},
  year={2019},
  publisher={Springer}
}

@article{H83,
  title={uniqueness theorems for 2nd-order elliptic differential equations},
  author={H\"ormander, Lars},
  journal={Communications in Partial Differential Equations},
  volume={8},
  number={1},
  pages={21--64},
  year={1983},
  publisher={MARCEL DEKKER INC 270 MADISON AVE, NEW YORK, NY 10016}
}

@article{LW23,
  title={Quantitative unique continuation for Robin boundary value problems on {$C^{1,1}$} domains},
  author={ Li, Zongyuan and Wang, Weinan},
  journal={Indiana University Mathematics Journal},
  volume={72},
  number={4},
  pages={1429--1460},
  year={2023},
  publisher={Indiana University}
}

@article{LM18,
 title={Review of Yau’s conjecture on zero sets of Laplace eigenfunctions},
author={Logunov, Alexander and Malinnikova, Eugenia},
journal={Current Developments in Mathematics},
volume={2018},
number={1},
pages={1429--1460},
year={2018}
}

@article{K20,
  title={Quantitative, uniqueness, and vortex degree estimates for solutions of the Ginzburg-Landau equation},
  author={Kukavica, Igor},
  journal={Electronic Journal of Differential Equations},
  volume={2000},
  number={61},
  pages={1--15},
  year={2000}
}

@article{JerisonKenig1985,
  author  = {Jerison, David and Kenig, Carlos E.},
  title   = {Unique continuation and absence of positive eigenvalues for {Schr\"odinger} operators},
  journal = {Ann. of Math. (2)},
  volume  = {121},
  number  = {3},
  pages   = {463--494},
  year    = {1985}
}

@article{KochTataru2001,
  author  = {Koch, Herbert and Tataru, Daniel},
  title   = {Carleman estimates and unique continuation for second-order elliptic equations with nonsmooth coefficients},
  journal = {Comm. Pure Appl. Math.},
  volume  = {54},
  number  = {3},
  pages   = {339--360},
  year    = {2001}
}

@article{Kukavica1998,
  author  = {Kukavica, Igor},
  title   = {Quantitative uniqueness for second-order elliptic operators},
  journal = {Duke Math. J.},
  volume  = {91},
  number  = {2},
  pages   = {225--240},
  year    = {1998}
}

@article{LiuTianYang2025,
  author  = {Liu, Hairong and Tian, Long and Yang, Xiaoping},
  title   = {Quantitative unique continuation property for solutions to a bi-{Laplacian} equation with a potential},
  journal = {J. London Math. Soc.},
  volume  = {112},
  number  = {2},
  pages   = {e70265},
  year    = {2025}
}

@article{BG16,
  title={Quantitative uniqueness for zero-order perturbations of generalized Baouendi-Grushin operators},
  author={Banerjee, Agnid and Garofalo, Nicola},
  journal={arXiv:1604.06000},
  year={2016}
}

@article{Meshkov1992,
  author  = {Meshkov, Vladimir Z.},
  title   = {On the possible rate of decay at infinity of solutions of second order partial differential equations},
  journal = {Math. USSR-Sb.},
  volume  = {72},
  number  = {2},
  pages   = {343--361},
  year    = {1992}
}

@article{Zhu2016,
  author  = {Zhu, Jiuyi},
  title   = {Quantitative uniqueness of elliptic equations},
  journal = {Amer. J. Math.},
  volume  = {138},
  number  = {3},
  pages   = {733--762},
  year    = {2016}
}

@article{Zhu2019,
  author  = {Zhu, Jiuyi},
  title   = {Doubling inequality and nodal sets for solutions of bi-{Laplace} equations},
  journal = {Arch. Rational Mech. Anal.},
  volume  = {232},
  pages   = {1543--1595},
  year    = {2019}
}

@book{H07,
  title={The analysis of linear partial differential operators III: Pseudo-Differential Operators},
  author={H{\"o}rmander, Lars},
  year={2007},
  publisher={Springer}
}

\end{document}